\documentclass{amsart}
\usepackage{mathtools}
\usepackage{amssymb}
\usepackage{enumerate}
\usepackage{hyperref}
\usepackage{pgfplots}
\usepackage[utf8x]{inputenc}
\DeclareMathOperator{\rank}{rank}
\DeclareMathOperator{\dist}{dist}
\DeclareMathOperator{\supp}{supp}
\DeclareMathOperator{\diam}{diam}

\DeclareMathOperator{\spn}{span}
\DeclareMathOperator{\gr}{Gr}

\newtheorem{theorem}{Theorem}[section]
\newtheorem{lemma}[theorem]{Lemma}
\newtheorem{proposition}[theorem]{Proposition}
\newtheorem{problem}[theorem]{Problem}
\theoremstyle{remark}\newtheorem{remark}[theorem]{Remark}
\theoremstyle{remark}\newtheorem{claim}[theorem]{Claim}
\numberwithin{equation}{section}

\DeclareFontFamily{U}{mathx}{\hyphenchar\font45}
\DeclareFontShape{U}{mathx}{m}{n}{
      <5> <6> <7> <8> <9> <10>
      <10.95> <12> <14.4> <17.28> <20.74> <24.88>
      mathx10
      }{}
\DeclareSymbolFont{mathx}{U}{mathx}{m}{n}
\DeclareFontSubstitution{U}{mathx}{m}{n}
\DeclareMathAccent{\widecheck}{0}{mathx}{"71}
\DeclareMathAccent{\wideparen}{0}{mathx}{"75}

\title[Improved bounds for restricted projection families]{Improved bounds for restricted projection families via weighted Fourier restriction}
\author[T.~L.~J.~Harris]{Terence L.~J.~Harris}
\address{Department of Mathematics, University of Illinois, Urbana, IL 61801, USA}
\address{Department of Mathematics, Cornell University, Ithaca, NY 14853, USA}
\email{tlh236@cornell.edu}
\thanks{This material is based upon work partially supported by the National Science Foundation under Grant No. DMS-1501041. I would like to thank Burak Erdoğan for advice, and for financial support.}
\begin{document}
\begin{abstract} \begin{sloppypar} It is shown that if $A \subseteq \mathbb{R}^3$ is a Borel set of Hausdorff dimension $\dim A \in (3/2,5/2)$, then for a.e.~$\theta \in [0,2\pi)$ the projection $\pi_{\theta}(A)$ of $A$ onto the 2-dimensional plane orthogonal to $\frac{1}{\sqrt{2}}(\cos \theta, \sin \theta, 1)$ satisfies $\dim \pi_{\theta}(A) \geq \max\left\{\frac{4\dim A}{9} +  \frac{5}{6},\frac{2\dim A+1}{3} \right\}$. This improves the bounds of Oberlin-Oberlin~\cite{oberlin}, and of Orponen-Venieri~\cite{venieri}, for $\dim A \in (3/2,5/2)$. More generally, a weaker lower bound is given for %two possible generalisations of this problem in higher dimensions, and for 
families of planes in $\mathbb{R}^3$ parametrised by curves in $S^2$ with nonvanishing geodesic curvature.  %One of the lower bounds given is for dimension distortion under projections onto $\left(\frac{\gamma(\theta)}{\sqrt{2}},\frac{1}{\sqrt{2}}\right)^{\perp}$, where $\gamma$ is any smooth curve in $S^{d-1}$ such that $\det\left(\gamma, \gamma', \dotsc, \gamma^{(d-1)}\right)$ is nonvanishing. 
\end{sloppypar} \end{abstract}
\maketitle
\section{Introduction} 
\begin{sloppypar} This article gives improved a.e.~lower bounds for Hausdorff dimension under ``restricted'' families of orthogonal projections. The behaviour of Hausdorff dimension under orthogonal projections was first studied in 1954 by Marstrand~\cite{marstrand2}, who showed that if $A$ is a Borel set in the plane, then for $0 \leq \dim A \leq 1$ the projection of $A$ onto a.e.~line through the origin has dimension equal to $\dim A$, and if $\dim A >1$ then the projection of $A$ onto a.e.~line through the origin has positive length. This was generalised to projections onto $k$-planes in $\mathbb{R}^n$ by Mattila~\cite{mattila5}, with respect to the natural rotation invariant probability measure on the Grassmannian. Somewhat more recently, questions of this type were studied for lower dimensional submanifolds of the Grassmannian~\cite{jarvenpaa1, jarvenpaa2,fasslerorponen14,orponen,oberlin,chen,kaenmaki,venieri}, in which case the problem is more difficult. Sets of projections corresponding to planes in lower dimensional subsets of the Grassmannian are referred to as ``restricted projection families''.

The first result given here is for a 1-dimensional family of 2-dimensional planes in $\mathbb{R}^3$. To state it, let $\pi_{\theta}$ be orthogonal projection onto the orthogonal complement of $\frac{1}{\sqrt{2}}(\cos \theta, \sin \theta, 1)$ in $\mathbb{R}^3$, and denote the Hausdorff dimension of a set $A$ by $\dim A$. A subset of a complete separable metric space $X$ is called analytic if it is the continuous image of a Borel subset of $Y$, for some complete separable metric space $Y$ (in particular, every Borel subset of $X$ is analytic).
\begin{theorem} \label{neat} If $A \subseteq \mathbb{R}^3$ is an analytic set with $\dim A \in (3/2,5/2)$, then 
\[ \dim \pi_{\theta}(A) \geq \max\left\{\frac{4\dim A}{9} +  \frac{5}{6},\frac{2 \dim A+1}{3} \right\}, \]
for a.e.~$\theta \in [0,2\pi)$.  \end{theorem}
This improves the previously known bounds if $\dim A \in (3/2,5/2)$, and makes partial progress towards Conjecture~1.6 from~\cite{fasslerorponen14}, for the special curve $\gamma(\theta) = \frac{1}{\sqrt{2}}(\cos \theta, \sin \theta, 1)$. This conjecture asserts that $\dim \pi_{\theta} (A) \geq \min \left\{ \dim A , 2 \right\}$ for a.e.~$\theta$, where $\gamma$ can be any curve with nonvanishing geodesic curvature in $S^2$, and $\pi_{\theta}$ is the projection onto $\gamma(\theta)^{\perp}$. By a rescaling argument (see Lemma~A.1 in~\cite{venieri}), Theorem~\ref{neat} continues to hold if the curve $\frac{1}{\sqrt{2}}(\cos \theta, \sin \theta, 1)$ is replaced by any circle in $S^2$ which is not a great circle. 

The best currently known bounds will be summarised here, omitting some which have since been superseded. The a.e.~lower bound $\dim \pi_{\theta}(A)  \geq \min\left\{\dim A,1 \right\}$ was obtained by Järvenpää, Järvenpää, Ledrappier, and Leikas~\cite{jarvenpaa1}, and still holds if the curve is replaced by any other circle in $S^2$ (even a great circle). For a great circle and for $\dim A \leq 2$ this is the best possible, which can be seen by taking any set of the given dimension contained in the plane of the great circle. 

Oberlin and Oberlin~\cite{oberlin} proved 
\begin{equation} \label{fourier} \begin{aligned} 
\dim \pi_{\theta}(A) &\geq \frac{3\dim A}{4}, && \dim A \in (1,2], \\
\dim \pi_{\theta}(A) &\geq \dim A - \frac{1}{2}, && \dim A \in (2,5/2], \\
\mathcal{H}^2\left( \pi_{\theta}(A) \right) &>0, && \dim A \in (5/2,3], \end{aligned} \end{equation}
for a.e.~$\theta \in [0, 2\pi)$, which (prior to Theorem~\ref{neat}) was the best known a.e.~lower bound for $\dim A \geq 9/4$, whilst the inequality $\dim A > 5/2$ remains the best known sufficient condition that ensures $\mathcal{H}^2\left( \pi_{\theta}(A) \right) >0$ almost everywhere.

Orponen and Venieri~\cite{venieri} proved the sharp a.e.~equality $\dim \pi_{\theta}(A)  = \dim A$ for $\dim A \in (1, 3/2]$, and gave the a.e.~lower bound 
\begin{equation} \label{GMT} \dim \pi_{\theta}(A) \geq 1+ \frac{\dim A}{3}, \quad \dim A \in (3/2, 3]. \end{equation}
Prior to Theorem~\ref{neat}, the lower bound in \eqref{GMT} was the record for $3/2 < \dim A < 9/4$. A comparison between Theorem~\ref{neat} and prior results is shown in Figure \ref{picture}. 

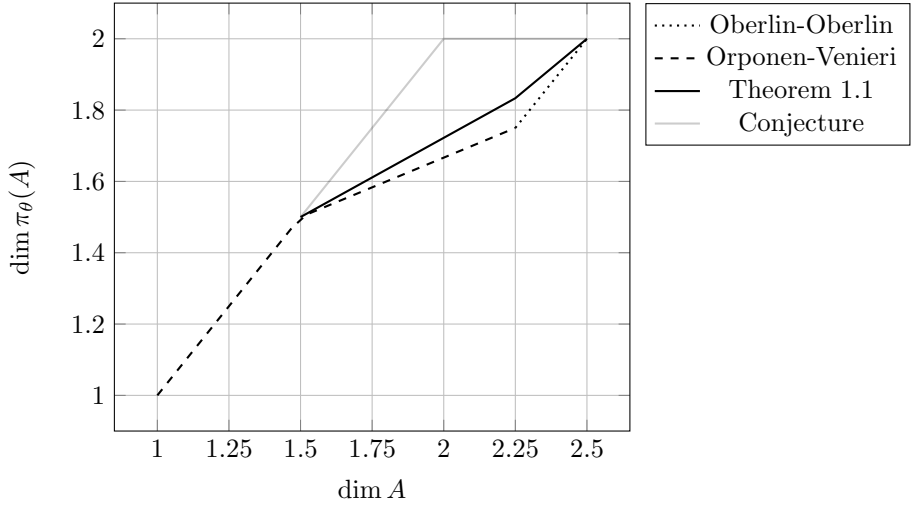
\begin{figure} 
\begin{tikzpicture}[
declare function={
    %func(\x)= (\x<=1.5) * (\x)   +
    %          and(\x>1.5, \x<2.625) * (\x*0.444444+ 0.833333);
   %           and(\x>=2.625, \x<=3) * (2);
	  func(\x)= (\x<=1.5) * (\x)   +
              and(\x>1.5, \x<2.25) * (\x*0.444444+ 0.833333) +
							and(\x>=2.25, \x<2.5) * (\x*0.666666+ 0.333333);
   %           and(\x>=2.625, \x<=3) * (2);
	 % func5(\x)= (\x<=1.5) * (\x)   +
   %           and(\x>1.5, \x<=3) * (\x*0.666666+ 0.333333);
   %           and(\x>=2.625, \x<=3) * (2);
		func2(\x)= (\x<=1.5) * (\x)   +
              and(\x>1.5, \x<=3) * (1+\x*0.333333);
		func3(\x)= (\x<=2) * (\x)   +
              and(\x>2, \x<=3) * (2);
		func4(\x)= (\x<=2) * (0.75*\x)   +
              and(\x>2, \x<=2.5) * (\x-0.5)		+
							and(\x>2.5, \x<=3) * (2);
	%	func4(\x)= (\x<=1) * (0.75*\x)   +
  %            and(\x>2, \x<=2.5) * (\x-0.5)		+
	%						and(\x>2.5, \x<=3) * (2);
  }]
  \begin{axis}[ 
	grid=major,
	xtick distance=0.25,
    xlabel=$\dim A$,
    ylabel={$\dim \pi_{\theta} (A)$},
		legend style={
legend pos=outer north east,
}
  ] 
 %\addplot[loosely dashed,thick,domain=0:1]{x}; 
 \addplot[dotted,thick,domain=2.25:2.5]{func4(x)}; 
 \addplot[dashed,thick,domain=1:2.25]{func2(x)}; 
 \addplot[black,thick,domain=1.5:2.5]{func(x)}; 
 %\addplot[densely dotted,thick,domain=2.25:2.5]{func5(x)}; 
 %\addplot[black,thick,domain=2.25:2.5]{func5(x)}; 
 \addplot[black,thick,opacity = 0.2, domain=1.5:2.5]{func3(x)};
 %\node [black] at (125,833) {\textbullet};
 %\node [black] at (125,750) {\textbullet};
\legend{Oberlin-Oberlin, Orponen-Venieri, Theorem~\ref{neat},Conjecture}
  \end{axis}
\end{tikzpicture}
\caption{The conjectured and best known a.e.~lower bounds for $\dim \pi_{\theta}(A)$, with $\dim A \in (1,5/2)$.}
\label{picture}
\end{figure}

The proof of Theorem~\ref{neat} uses the decomposition of a fractal measure into ``good'' and ``bad'' parts, and a wave version of the ``refined Strichartz inequality'', both recently used on the planar distance set problem in~\cite{GIOW}. The idea of removing the ``bad'' part of a measure also featured earlier in~\cite{shmerkin,orponen2}. Here the ``bad'' part is bounded using the key lemma from Orponen and Venieri's proof of \eqref{GMT}, whereas the ``good'' part is bounded using the refined Strichartz inequality, and by augmenting the Fourier-analytic approach of Oberlin-Oberlin with an ``improvement due to localisation'' technique (which I learnt from~\cite[Lemma~2.3]{cho}; it is used here in \eqref{restrictwavepackets}--\eqref{offdiag3} and \eqref{diag20}--\eqref{offdiag20}). Finally, these two bounds are converted to a projection theorem by adapting Liu's $L^2$-method from~\cite{liu}. For more information about the method, the intuition behind the proof is explained in \cite[pp.~8--13]{thesis}.

The lower bound \eqref{fourier} of Oberlin and Oberlin holds more generally for planes parametrised by curves in $S^2$ with nonvanishing geodesic curvature. The proof of the Orponen-Venieri lemma relies crucially on the constant height property of the curve $\frac{1}{\sqrt{2}}(\cos \theta, \sin \theta, 1)$, so the ``good-bad'' decomposition would likely not yield the same bound for more general curves (although see Problem~\ref{open}). In this more general setting, the following theorem gives an improvement to the lower bound \eqref{fourier} in the intermediate range; the proof uses the aforementioned ``localisation'' technique, without the ``good-bad'' decomposition (see Subsection~\ref{notation} for notation).

\begin{theorem} \label{curvature} Let $\gamma: [a,b] \to S^2$ be a $C^2$ curve with $\det\left(\gamma,\gamma', \gamma''\right)$ nonvanishing, and let $\pi_{\theta}= \pi_{\theta,\gamma}$ be projection onto $\gamma(\theta)^{\perp}$.

Fix $\alpha \in (0, 3)$, and let $\nu$ be a Borel measure on $\mathbb{R}^3$ with $\diam \supp \nu \leq C$. If $\alpha \leq 5/2$ and $0 \leq t< \min\left\{ \alpha, \max\left\{ \frac{1+\alpha}{2}, \alpha - \frac{1}{2} \right\} \right\}$, then
\begin{equation} \label{energyone} \int_a^b I_t\left( \pi_{\theta \#} \nu \right) \, d\theta \lesssim_{\alpha,t,C} c_{\alpha}(\nu) \nu\left( \mathbb{R}^3 \right), \end{equation}
and if $\alpha > 5/2$, then 
\begin{equation} \label{energytwo} \int_a^b \left\lVert \pi_{\theta \#} \nu  \right\rVert_{L^2(\mathbb{R}^3, \mathcal{H}^2 )}^2 \, d\theta \lesssim_{\alpha,C} c_{\alpha}(\nu) \nu\left( \mathbb{R}^3 \right).  \end{equation}Consequently, for any analytic subset $A$ of $\mathbb{R}^3$,
\begin{equation} \label{introref} \begin{aligned}  \dim \pi_{\theta} (A) &= \dim A, \quad && \dim A \in [0,1],  \\
\dim \pi_{\theta} (A) &\geq \frac{\dim A +1}{2}, \quad && \dim A \in (1,2],  \\
\dim \pi_{\theta} (A) &\geq \dim A - \frac{1}{2}, \quad && \dim A \in (2,5/2],  \\
\mathcal{H}^{2}(\pi_{\theta} (A)) &>0,  \quad && \dim A > 5/2, \end{aligned} \end{equation}
for a.e.~$\theta \in [a,b]$.
\end{theorem} The second part \eqref{introref} of Theorem~\ref{curvature} improves the bound of Oberlin-Oberlin (see \eqref{fourier}) and also Orponen's bound from~\cite[Theorem~1.8]{orponen} in the range $1 < \dim A < 2$ (the bound from~\cite{orponen} is qualitative; given as $1+\sigma(\dim A)$ for an unspecified positive function $\sigma$ of $\dim A >1$).  The following (open) problem is suggested by Theorem~\ref{curvature} and the method of proof in Theorem~\ref{neat}.
\begin{problem} \label{open} Let $\gamma$ be a smooth curve in $S^2$ with nonvanishing geodesic curvature, and let $A \subseteq \mathbb{R}^3$ be an analytic set with $\dim A < 5/2$. Does the a.e.~inequality
\[ \dim \pi_{\theta} (A) \geq \min\left\{ \dim A, \max\left\{ \frac{\dim A}{2} + \frac{2}{3}, \frac{2\dim A + 1}{3} \right\} \right\}, \]
necessarily hold?
\end{problem}
A positive answer to Problem~\ref{open} would verify Conjecture~1.6 from~\cite{fasslerorponen14} for $\dim A \in (1,4/3]$. The decoupling theorem is known to hold for general conical surfaces; see~\cite[Exercise~12.5]{demeter}. 

Theorem~\ref{neat} will be proved in Section~\ref{sellingpoint}. In Section~\ref{sectFourier}, Theorem~\ref{curvature} and Proposition~\ref{projection} will both be deduced as consequences of the more general Proposition~\ref{proposition}, the proof of which occupies most of Section~\ref{sectFourier}. %Theorem~\ref{projection2} will be proved in Section~\ref{orponen-venieri-method}. 

\subsection{Notation} \label{notation}
 Given a set $E$ in Euclidean space, let $\mathcal{N}_{\delta}(E)$ be the open $\delta$-neighbourhood of $E$. If $E$ is a box, let $CE$ be the box with the same centre but with side lengths scaled by $C$. Let $\mathcal{H}^s$ be the $s$-dimensional Hausdorff measure in Euclidean space. For $\alpha \geq 0$ and a positive Borel measure $\mu$ supported in the unit ball of $\mathbb{R}^{d+1}$, define 
\[ c_{\alpha}(\mu)  = \sup_{\substack{x \in \mathbb{R}^{d+1} \\
r>0 } } \frac{ \mu(B(x,r)) }{r^{\alpha}}. \]
By Frostman's Lemma (see~\cite{davies,howroyd,mattila3}), the Hausdorff dimension $\dim B$ of a Borel (or analytic) set $B \subseteq \mathbb{R}^{d+1}$ is the supremum over all $\alpha \in [0,d+1]$ for which there exists a nonzero Borel measure $\mu$ supported on $B$ with finite $c_{\alpha}(\mu)$. Hausdorff dimension can also be characterised through the energy
\[ I_{\alpha}(\mu) := \int \int \lvert x-y\rvert^{-\alpha} \, d\mu(x) \,d\mu(y) = c_{\alpha, d} \int_{\mathbb{R}^{d+1}} \lvert \xi\rvert^{\alpha-(d+1)} \left\lvert\widehat{\mu}(\xi)\right\rvert^2 \, d\xi, \]
where $\alpha \in (0,d+1)$; the last integral is called the Fourier energy of $\mu$~\cite[Theorem 3.10]{mattila4}. For $s \in \mathbb{R}$, define the homogeneous Sobolev norms by
\[ \lVert \mu\rVert_{\dot{H}^s}^2 = \lVert \mu\rVert_{\dot{H}^s\left( \mathbb{R}^{d+1} \right)}^2 = \int_{\mathbb{R}^{d+1}} \lvert \xi\rvert^{2s} \left\lvert\widehat{\mu}(\xi)\right\rvert^2 \, d\xi,  \]
so that $ \lVert \mu\rVert_{\dot{H}^{\frac{\alpha-(d+1)}{2}}}^2$ and $I_{\alpha}(\mu)$ are equivalent for $\alpha \in (0,d+1)$.  A Borel measure $\mu$ with $c_{\alpha}(\mu) < \infty$ satisfies $I_s(\mu) < \infty$ for any $s< \alpha$, and if $I_{\alpha}(\mu)<\infty$ then $0< c_{\alpha}\left(\mu\restriction_A\right) < \infty$ for some Borel set $A \subseteq \supp \mu$; see e.g.~\cite[Chapter~2]{mattila4}.

Given measurable spaces $X$, $Y$, a measure $\mu$ on $X$ and a measurable function $f: X \to Y$, the pushforward measure $f_{\#}\mu$ on $Y$ is defined by $(f_{\#}\mu)(E) = \mu\left(f^{-1}(E) \right)$. The definition is the same if $\mu$ is a complex measure.  

Given $d \geq 2$, a set $\Omega \subseteq \mathbb{R}^{d-1}$ and a function $G:  \Omega \to  S^d$, let 
\begin{equation} \label{Gammadefn} \Gamma_R(G) = \{ \rho G(y) :  y \in \Omega, \, R/2 \leq \rho \leq R\}, \quad R >0, \end{equation}
and let $\Gamma(G) = \Gamma_1(G)$. 

\section{Projections onto the planes \texorpdfstring{$(\cos \theta, \sin \theta, 1)^{\perp}$}{(cos, sin, 1)-perp}}

\label{sellingpoint}

\subsection{Setup and preliminaries}

\label{subsectionBl}

Define $\gamma: [0,2\pi) \to S^2$ by $\gamma(\theta) = \frac{1}{\sqrt{2}}(\cos \theta, \sin \theta, 1)$, and let $\pi_{\theta}: \mathbb{R}^3 \to \gamma(\theta)^{\perp}$ be orthogonal projection onto the 2-dimensional plane $\gamma(\theta)^{\perp} \subseteq \mathbb{R}^3$. For each $\theta$, an orthonormal basis for $\gamma(\theta)^{\perp}$ is 
\[ \left\{  \sqrt{2}\gamma'(\theta),  \gamma(\theta) \times \left(\sqrt{2}\gamma'(\theta)\right) \right\} = \left\{(-\sin \theta, \cos \theta, 0 ), \frac{1}{\sqrt{2}} (-\cos \theta, -\sin \theta, 1 ) \right\}. \]
Let $\lvert \cdot \rvert = \lvert  \cdot \rvert_{\bmod 2\pi}$ be the distance on $[0, 2\pi)$ which naturally identifies this interval with the unit circle: 
\[ \lvert  x-y \rvert_{\bmod 2\pi} := \dist((x-y)\bmod 2\pi , \{0,2\pi\} ). \]

For any parameter $K \geq 1$, the truncated cone $\Gamma_1=\Gamma_1(\gamma)$ has a finitely overlapping covering by boxes of dimensions $K^{-1} \times 1 \times K^{-2}$; these will be referred to as the ``standard'' $K^{-1}$-boxes or caps. They come from the standard covering of $S^1$ by rectangles of dimensions $K^{-1} \times K^{-2}$. 

The notation used for the wave packet decomposition here will be similar to that from~\cite{GIOW} for ease of comparison. Let $\epsilon$ be a very small number, which will be sent to zero at the end of the proof.  Let $\Gamma_{\mathbb{R}}$ be the entire light cone with both forward and backward parts:
\[ \Gamma_{\mathbb{R}} = \left\{ \lambda \gamma(\theta) : \theta \in [0,2\pi), \, \lambda \in \mathbb{R} \right\}. \]
Fix a large positive integer $J$ to be chosen later. For each integer $j$ let $\Gamma_{2^j} = \Gamma_{2^j}(\gamma)$ as defined in~\eqref{Gammadefn}. For each $j \geq J$ and $0 \leq k < j$, construct a finitely overlapping cover of 
\[ \left[\mathcal{N}_{2^{j-k}}(\Gamma_{2^j} \cup -\Gamma_{2^j}) \setminus \mathcal{N}_{2^{j-k-1}}\left(\Gamma_{\mathbb{R}}\right) \right] \] by boxes $\tau=\tau_{j,k}$ of dimensions 
\[ \sim 2^{j- \frac{k}{2}} \times 2^j \times 2^{j-k}. \]
For fixed $j$ and $k$, let $\Lambda_{j,k}$ be the set of boxes $\tau = \tau_{j,k}$ corresponding to $j$ and $k$. Each box $\tau \in \Lambda_{j,k}$ is such that $2^{-j} \tau$ is contained in a standard box $\widetilde{\tau}$ at scale $2^{-k/2}$ for the cone $\Gamma_1 \cup -\Gamma_1$. When $k=j$ the boxes $\tau \in \Lambda_{j,j}$ are defined similarly, except that they cover the set $\mathcal{N}_1(\Gamma_{2^j} \cup -\Gamma_{2^j})$. The wave packet decomposition is set up in this way to apply a change of variables later; the $L^2$ integral of the ``good'' part of $\widehat{\mu}$ over the conical ring contained in the union of the boxes $\tau \in \Lambda_{j,k}$ will have a fixed Jacobian under this change of variables, and after rescaling by $2^{k-j}$ on the Fourier side this integral will correspond to a more standard $L^2$ conical average of the ``good'' part of $\widehat{\mu}$, which through duality will be controlled using the decoupling theorem for the cone. The extra rescaling step causes the wave packet decomposition used here to be slightly more complicated than in~\cite{GIOW}. 
 
This construction can be done in such a way that the boxes $2 \tau$ are finitely overlapping as $j$ and $k$ vary, and such that $\dist(2\tau_{j,k}, \Gamma_{\mathbb{R}}) \sim 2^{j-k}$ for all $j\neq k$. Let $\{\psi_{\tau}\}_{j,k, \tau \in \Lambda_{j,k}}$ be a smooth partition of unity subordinate to the cover of the set $\bigcup_{j \geq J} \bigcup_{0 \leq k \leq j} \bigcup_{\tau \in \Lambda_{j,k}} \tau$ by the sets $1.1\tau$, such that each $\psi_{\tau}$ has compact support in $1.1\tau$. Then
\begin{equation} \label{unitypartition} 1 =  \sum_{j \geq J} \sum_{k=0}^j \sum_{\tau \in \Lambda_{j,k}} \psi_{\tau} \quad \text{on} \quad \bigcup_{j \geq J} \bigcup_{0 \leq k \leq j} \bigcup_{\tau \in \Lambda_{j,k}} \tau. \end{equation}

Fix a small $\delta >0$ with $\delta \ll \epsilon$, to be chosen after $\epsilon$. For each triple $(j,k, \tau)$ with $j \geq J$ and $\tau \in \Lambda_{j,k}$, construct a finitely overlapping cover of the ball of radius 2 around the origin in $\mathbb{R}^3$, with tubes $T$ of dimensions 
\[ 2^{-j}2^{k(1/2+\delta)} \times 2^{-j}2^{k(1/2+\delta)} \times (10 \cdot 2^{-(j-k)+k\delta}), \]
such that each rescaled set $2^{j-k} T$ is a tube of diameter $\approx 2^{-k/2}$ and length $\approx 1$, with direction normal to the cone at the rescaled box $2^{-j}\tau$ (which has dimensions $\sim 2^{-k/2} \times 1 \times 2^{-k}$). Each tube is a union of boxes dual to the corresponding cap $\tau$, which are shorter than $T$ in the middle direction. Let $\mathbb{T}_{j,k, \tau}$ be the set of tubes corresponding to $\tau \in \Lambda_{j,k}$. Let $\{\eta_T\}_{T \in \mathbb{T}_{j,k, \tau} }$ be a smooth partition of unity subordinate to this cover of $B_3(0,2)$. For each $T \in \mathbb{T}_{j, k,  \tau}$, define $M_T$ by 
\[ M_Tf = \eta_T \widecheck{\psi_{\tau} \widehat{f}} = \eta_T\mathcal{F}^{-1}\left(\psi_{ \tau} \mathcal{F}(f) \right), \]
for Schwartz $f$. Let 
\[ \mathbb{T}_{j,k} = \bigcup_{ \tau \in \Lambda_{j,k} } \mathbb{T}_{j, k,\tau}, \quad \mathbb{T} = \bigcup_{j=J}^{\infty} \bigcup_{k=0}^j \mathbb{T}_{j,k}. \]
Fix a positive smooth function $\mu$ supported in the unit ball in $\mathbb{R}^3$, identified with the measure $\mu \, dx$.  The set of ``bad'' tubes will be the set of tubes with ``large'' $\mu$ measure, where ``large'' is defined so that the contribution coming from these ``bad'' tubes can be handled by the lemma of Orponen-Venieri (Lemma~\ref{venierilemma}). The contribution from the remaining ``good'' tubes will be controlled using Fourier analysis, which is (roughly) where the improvement over the Orponen-Venieri bound comes from. 

More explicitly, given $\alpha \in (3/2,5/2)$ (to be fixed later, corresponding to the ``dimension'' of $\mu$), define 
\begin{equation} \label{alphastardefn} \alpha^* = \max\left\{ \frac{\alpha}{3} + 1, \alpha -\frac{1}{2} \right\} - \epsilon.  \end{equation}
Define the set of ``bad'' tubes corresponding to $\tau \in \Lambda_{j,k}$ by
\begin{equation} \label{badtubestau} \mathbb{T}_{j,k, \tau, b} = \left\{ T \in \mathbb{T}_{j,k, \tau} : \mu\left(10T\right) \geq 100 \cdot 2^{\frac{k}{2} \left(100 \delta -\alpha^*\right)} 2^{-\alpha(j-k)} \right\}, \end{equation}
and let 
\begin{equation} \label{goodtubestau} \mathbb{T}_{j,k, \tau, g} = \mathbb{T}_{j,k, \tau} \setminus \mathbb{T}_{j,k, \tau, b}. \end{equation}
Let
\[ \mathbb{T}_{j,k, b} = \bigcup_{\tau \in \Lambda_{j,k}} \mathbb{T}_{j, k,\tau, b}, \quad  \mathbb{T}_{j,k, g} = \bigcup_{\tau \in \Lambda_{j,k}} \mathbb{T}_{j,k, \tau, g}. \]
%Similarly let 
%\[ \mathbb{T}_b = \bigcup_{j=J}^{\infty} \bigcup_{k=0}^j  \mathbb{T}_{j, k, b}, \quad \mathbb{T}_g = \bigcup_{j=J}^{\infty} \bigcup_{k=0}^j  \mathbb{T}_{j,k, g}. \]
Define the ``bad'' part of $\mu$ by summing over the ``bad'' tubes with $k$ bounded away from zero in the following quantitative sense:
\begin{equation} \label{mubaddefn} \mu_b =  \sum_{j=J}^{\infty} \sum_{k= \lceil j \epsilon \rceil}^j \sum_{\tau \in \Lambda_{j,k}} \sum_{T \in \mathbb{T}_{j,k, \tau, b}} M_T \mu. \end{equation}
Define the ``good'' part of $\mu$ by 
\[\mu_g = \mu - \mu_b. \]
Since the caps $\tau$ in \eqref{unitypartition} and the caps occurring in \eqref{mubaddefn} do not cover all of $\mathbb{R}^3$ but only a region around the light cone, the measure $\mu_g$ is not equal to the sum over the ``good'' tubes (in contrast to~\cite{GIOW}).

For the specific function $\mu$ used later, only finitely many values of $j$ in \eqref{mubaddefn} will be non-negligible, so there will not be any convergence issues in the infinite sum. The pushforward $\pi_{\theta\#}\mu_b$ of the complex measure $\mu_b$ will be equal to the sum of $\pi_{\theta\#} M_T\mu$ over the tubes defining $\mu_b$. 

%For each $j \geq J$ and $k$ with $0 \leq k \leq j$, let $\{B_l\}_l$ be a finitely overlapping cover of $\overline{B_3(0,1)}$ by balls of radius $2^{-(j-k)}$, and let $\{\phi_{j,k,l}\}_l$ be a smooth partition of unity subordinate to this cover. Let $\mu_{j,k,l} = \phi_{j,k,l} \mu$.

To bound the average $L^1$ norm of the pushforward of the ``bad'' part of a measure, the following lemma from Orponen and Venieri's work on the same problem will be used.
\begin{lemma}[{\cite[Lemma~2.3]{venieri} for $s < 9/4$ and~\cite{oberlin} for $s \geq 9/4$}]  \label{venierilemma} Let $s\in [0, 5/2)$, and suppose that $\nu$ is a compactly supported Borel measure on $\mathbb{R}^3$ such that $\sup_{\substack{x \in \mathbb{R}^3 \\ r >0 }} \frac{ \nu\left( B(x,r) \right) }{r^s}\leq  C_1$ and $\diam\left( \supp \nu \right) \leq C_2$, where $C_1, C_2$ are positive constants. Then for any 
\[ \kappa > \max\left\{ 0, \min\left\{\frac{2s}{3} -1,\frac{1}{2} \right\}\right\}, \]
there exist $\delta_0, \eta>0$, depending only on $C_1,C_2,s$ and $\kappa$, such that
\[ \nu\left\{ y \in \mathbb{R}^3: \mathcal{H}^1\left\{ \theta \in [0,2\pi) : \pi_{\theta\#} \nu\left( B\left( \pi_{\theta}(y), \delta \right)\right) \geq \delta^{s-\kappa} \right\} \geq \delta^{\eta} \right\} \leq \nu\left(\mathbb{R}^3\right)\delta^{\eta}, \]
for all $\delta \in (0, \delta_0)$. 
\end{lemma}
The $\nu\left(\mathbb{R}^3\right)$ factor is not given explicitly in~\cite{venieri}, but follows from their proof. This factor is explained in more detail in Appendix~\ref{appendix2}, along with a proof that the case $s \geq 9/4$ follows from the main inequality in~\cite{oberlin}. %The proof of a higher dimensional version of this lemma will also be given in Lemma~\ref{sufficient}. 
\begin{remark} For ease of reference, the choice and dependence of the parameters in the proof of Theorem~\ref{neat} below will be summarised here. The parameter $\epsilon >0$ is given first and may be arbitrarily small (see the paragraph before \ref{refstar}). The parameters $\delta_0,\eta>0$ are chosen after $\epsilon$, and then $\delta$ is chosen after $\eta$ and $\delta_0$, so that 
\[ 0 < \delta \ll \eta \ll 1, \quad 0 < \delta \ll \delta_0 \ll 1; \]
see the paragraph after~\eqref{refdoublestar}. With these parameters fixed, $\epsilon'>0$ is given (see the paragraph after~\eqref{sprimedefn}), and then $\delta_1>0$ is chosen after $\epsilon'$ such that 
\[ \delta_1 < \delta_0/4, \quad \delta_1^{\eta} < \epsilon', \quad \delta_1 \ll 2^{-1/(\epsilon^{10})}, \]
and such that 
\begin{equation} \label{trivialinequality} \max\left\{x^4 2^{-(\epsilon^2 x)/(100) }, x^2 2^{-(\epsilon^2 \eta x)/(8C) }\right\} \ll \min\left\{ \epsilon', \delta_0 \right\} \quad \text{ for all } x > \left\lvert \log_2 \delta_1 \right\rvert, \end{equation}
where $C>0$ is an absolute constant (see paragraphs after \eqref{sprimedefn}, \eqref{delta1choiceone} and \eqref{Jchoice}). A positive integer $j_0> \left\lvert \log_2 \delta_1 \right\rvert$ is then given (see~\eqref{bad}) and $J$ is defined by $J= \left\lfloor (j_0 \epsilon)/(2C) \right\rfloor$ (see \eqref{Jchoice}).  

The intermediate size parameters are $\dim A$, $\alpha = \dim A -\epsilon$ (see the paragraph before \eqref{refstar}), $\alpha^* = \max\left\{ 1+ \alpha/3, \alpha- 1/2 \right\} - \epsilon$ (see \eqref{alphastardefn}), $s= \max\{ 4\alpha /9 + 5/6, (2\alpha+1)/3 \}$ (see \eqref{sprimedefn}), $s' = s-50\sqrt{\epsilon}$ (see \eqref{sdoubleprimedefn}) and $\kappa = 1-\epsilon/(10^5)$ (see \eqref{kappadefn}). \end{remark}

\subsection{Main part of the proof}

\begin{proof}[Proof of Theorem~\ref{neat}] Let $\dim A \in (3/2,5/2)$ and assume without loss of generality that $A$ is a subset of the unit ball. Let $\epsilon>0$ be small. By Frostman's lemma, there is a compactly supported probability measure $\nu$ on $A$ with $c_{\alpha}(\nu) < \infty$, where $\alpha = \dim A -\epsilon> 3/2$. Let $E \subseteq [0,2\pi)$ be a compact set such that 
\begin{equation} \label{refstar} \dim \pi_{\theta} \supp \nu < s-200 \sqrt{\epsilon} \quad \text{for every } \theta \in E, \end{equation}
where 
\begin{equation} \label{sprimedefn} s:= \max\left\{\frac{4\alpha}{9} + \frac{5}{6}, \frac{2\alpha+1}{3} \right\}. \end{equation}
Let $\epsilon'>0$ be arbitrary. The proof will be carried out by showing that $\mathcal{H}^1(E) \lesssim_{\epsilon} \epsilon'$, then letting $\epsilon' \to 0$, and then letting $\epsilon \to 0$. For $\delta_0, \eta>0$ to be specified later (depending on $\epsilon$), let $\delta_1>0$ be such that $\delta_1 < \delta_0/4$ and $\delta_1^{\eta} < \epsilon '$ (the exact choice of $\delta_1$ will also be made later, but after $\delta_0$ and $\eta$; see the paragraphs after \eqref{delta1choiceone} and \eqref{Jchoice}). For each $\theta \in E$, let $\left\{ B_{\gamma(\theta)^{\perp}}\left(z_i(\theta), \delta_i(\theta) \right) \right\}_i$ be a cover of $\pi_{\theta} \supp \nu$ by balls of dyadic radii $\delta_i(\theta) < \delta_1$, such that 
\begin{equation} \label{scondition} \sum_i \delta_i(\theta)^{s-100\sqrt{\epsilon}} < \epsilon' \end{equation}
(measurability issues of the function $\theta \mapsto \delta_i(\theta)$ will be ignored, they can be dealt with similarly to the proof of Lemma~2.1 in~\cite{harris2}). Let
\[ D_{\theta}^j = \bigcup_{\substack{i  \\
 \delta_i(\theta) = 2^{-j}}}  B\left(z_i(\theta), \delta_i(\theta) \right), \]
and let  $\widetilde{D_{\theta}^j}$ be the $2^{-j}$ neighbourhood of $D_{\theta}^j$. Let 
\begin{equation} \label{convolve2} \nu_j = \nu \ast \phi_j, \quad \phi_j(x) = 2^{3j} \phi(2^j x), \end{equation}
for a smooth positive bump function $\phi$ equal to 1 on the unit ball, satisfying $0 \leq \phi \leq 1$ everywhere and vanishing outside $B(0,2)$. For each $\theta \in E$, 
\begin{equation} \label{patch1} 1 = \nu(A) \leq \sum_{j > \left\lvert \log_2 \delta_1 \right\rvert} \pi_{\theta \#} \nu(D_{\theta}^j )  \lesssim \sum_{j > \left\lvert \log_2 \delta_1 \right\rvert} \int_{\widetilde{D_{\theta}^j}} \pi_{\theta \#} \nu_j \, d\mathcal{H}^2; \end{equation}
the last inequality follows by expanding out each summand of the right-hand side and applying Fubini, the assumption that $\phi=1$ on the unit ball, and the 1-Lipschitz property of the projections $\pi_{\theta}$ (as in~\cite[p.7]{liu}, see also \eqref{projformula} for the formula for the density $\pi_{\theta\#}\nu_j$ with respect to $\mathcal{H}^2$). Integrating \eqref{patch1} over $\theta \in E$ gives
\[ \mathcal{H}^1(E) \lesssim \sum_{j > \left\lvert \log_2 \delta_1 \right\rvert} \int_E \int_{\widetilde{D_{\theta}^j}} \pi_{\theta \#} \nu_j \, d\mathcal{H}^2 \, d\theta. \]
Hence there exists a single $j_0 > \left\lvert \log_2 \delta_1 \right\rvert$, which may depend on $E$ and $\epsilon'$, such that 
\begin{align}\notag  \mathcal{H}^1(E) &\lesssim j_0^2 \int_E \int_{\widetilde{D_{\theta}^{j_0}}} \pi_{\theta \#} \nu_{j_0} \, d\mathcal{H}^2 \, d\theta \\
\label{bad} &\leq j_0^2 \int_0^{2\pi} \left\lVert \pi_{\theta \#} \mu_b \right\rVert_{L^1(\mathbb{R}^3, \mathcal{H}^2)} \, d\theta + j_0^2\int_E \int_{\widetilde{D_{\theta}^{j_0}}} \left\lvert \pi_{\theta \#} \mu_g \right\rvert \, d\mathcal{H}^2 \, d\theta, \end{align}
where 
\[ \mu := \nu_{j_0}, \]
and the $J$ in the ``good-bad'' decomposition will be chosen later (depending on $j_0$ and $\epsilon$; see \eqref{Jchoice}). The proof will proceed by showing each term in \eqref{bad} is $\lesssim_{\epsilon} \epsilon'$, from which the theorem will be shown to follow. 

Assume the first term in \eqref{bad} dominates. If the angle of a tube $T$ is not roughly equal to the angle of projection $\theta$, then $\pi_{\theta\#} (M_T\mu)$ is negligible (the proof of this is via nonstationary phase, and will be postponed until Subsection~\ref{nonstationary}). To be more precise, for each $\tau \in \Lambda_{j,k}$ let $\angle \tau$ be the angle corresponding to $\tau$. For the forward (resp.~backward) light cone, this angle can be defined as the value of $\phi$ such that the line through $(\cos \phi, \sin \phi, 1)$ (resp.~$(\cos \phi, \sin \phi, -1)$) passes through the barycentre of the unique standard box $\widetilde{\tau}$ at scale $2^{-k/2}$ containing $2^{-j}\tau$. Use the notation $(\angle \tau)^*$ to denote $(\pi + \angle\tau) \bmod 2\pi $ if $\tau$ lies in the forward light cone, and $(\angle \tau)^* = \angle\tau$ if $\tau$ lies in the backward light cone. %Use $(\angle T)^*$ to denote $\angle T$ if $\tau(T)$ lies in the forward light cone, and $(\angle T)^* = (\pi+ \angle T)\bmod 2\pi$ if $\tau(T)$ lies in the backward light cone.
 For each $\theta \in [0,2\pi)$, the first integrand in \eqref{bad} satisfies
\begin{align*}  \left\lVert  \pi_{\theta \#} \mu_b \right\rVert_{L^1(\mathbb{R}^3, \mathcal{H}^2)} &= \left\lVert   \sum_{j=J}^{\infty} \sum_{k= \lceil j \epsilon \rceil}^j\sum_{\tau \in \Lambda_{j,k}} \sum_{T \in \mathbb{T}_{j,k, \tau, b}} \pi_{\theta \#} (M_T \mu) \right\rVert_{L^1(\mathbb{R}^3, \mathcal{H}^2)} \\
&\leq   \sum_{j=J}^{\infty} \sum_{k= \lceil j \epsilon \rceil}^j\sum_{\tau \in \Lambda_{j,k}} \sum_{T \in \mathbb{T}_{j,k, \tau, b}} \left\lVert  \pi_{\theta \#} (M_T \mu) \right\rVert_{L^1(\mathbb{R}^3, \mathcal{H}^2)} \\
&=   \sum_{j=J}^{\infty} \sum_{k= \lceil j \epsilon \rceil}^j\sum_{\substack{\tau \in \Lambda_{j,k} \\
 \left\lvert(\angle \tau)^*- \theta \right\rvert \\
\leq 10^32^{k(-1/2+\delta)} }} \sum_{T \in \mathbb{T}_{j,k, \tau, b} } \left\lVert  \pi_{\theta \#} (M_T \mu) \right\rVert_{L^1(\mathbb{R}^3, \mathcal{H}^2)} \\
&\quad + \sum_{j=J}^{\infty} \sum_{k= \lceil j \epsilon \rceil}^j\sum_{\substack{\tau \in \Lambda_{j,k} \\
  \left\lvert(\angle \tau)^*- \theta \right\rvert \\
	> 10^32^{k(-1/2+\delta)}} } \sum_{T \in \mathbb{T}_{j,k, \tau, b} } \left\lVert  \pi_{\theta \#} (M_T \mu) \right\rVert_{L^1(\mathbb{R}^3, \mathcal{H}^2)}.  \end{align*}
By Lemma~\ref{nonstationaryphase}, the second term is $\lesssim_{\delta, \epsilon, N} 2^{-JN}$, and therefore
	\begin{multline} \label{pause1} \left\lVert  \pi_{\theta \#} \mu_b \right\rVert_{L^1(\mathbb{R}^3, \mathcal{H}^2)} \\
	\lesssim_{\delta, \epsilon, N} 2^{-JN} + \sum_{j=J}^{\infty} \sum_{k= \lceil j \epsilon \rceil}^j\sum_{\substack{\tau \in \Lambda_{j,k} \\
 \left\lvert(\angle \tau)^*- \theta \right\rvert \\
\leq 10^32^{k(-1/2+\delta)} }} \sum_{T \in \mathbb{T}_{j,k, \tau, b} } \left\lVert  \pi_{\theta \#} (M_T \mu) \right\rVert_{L^1(\mathbb{R}^3, \mathcal{H}^2)}. \end{multline}
By the formula \eqref{projformula} for the pushforward density, the $L^1$ norm of the pushforward satisfies $\left\lVert  \pi_{\theta \#} f \right\rVert_{L^1(\mathbb{R}^3, \mathcal{H}^2)} \leq \lVert f\rVert_1$ for any $f \in L^1(\mathbb{R}^3)$; this will be applied to $f=M_T\mu$. The function $\widecheck{\psi_{\tau}}$ decays rapidly outside the box $\tau'$ centred at 0 with dual dimensions to $\tau$, and this box is smaller than $T$, so by Fubini the operator $M_T$ satisfies $\lVert M_T \mu\rVert_1  \lesssim_{\delta,N}  \mu\left(2T\right) + 2^{-kN}$.  Hence
\[ \left\lVert  \pi_{\theta \#} (M_T \mu) \right\rVert_{L^1(\mathbb{R}^3, \mathcal{H}^2)} \leq \lVert M_T \mu\rVert_1 \lesssim_{\delta, N} \mu\left(2T\right) + 2^{-kN}. \]
Putting this into \eqref{pause1} yields
\begin{equation} \label{additive} \left\lVert  \pi_{\theta \#} \mu_b \right\rVert_{L^1(\mathbb{R}^3, \mathcal{H}^2)} \lesssim_{\delta, \epsilon, N} 2^{-JN} + \sum_{j=J}^{\infty} \sum_{k= \lceil j \epsilon \rceil}^j\sum_{\substack{\tau \in \Lambda_{j,k} \\
 \left\lvert(\angle \tau)^*- \theta \right\rvert \\
\leq 10^32^{k(-1/2+\delta)} }} \sum_{T \in \mathbb{T}_{j,k, \tau, b} } \mu\left(2T\right). \end{equation}
This will be simplified using ``essential disjointness'' of the inner two sums. Let
\[ B_{j,k}(\theta) = \bigcup_{\substack{\tau \in \Lambda_{j,k} \\
 \left\lvert(\angle \tau)^*- \theta \right\rvert \\
\leq 10^32^{k(-1/2+\delta)} }}  \bigcup_{T \in \mathbb{T}_{j,k,\tau,b}} 2T. \] 
For each $\theta$, the number of $\tau$'s occurring in the third sum of \eqref{additive} is $\lesssim 2^{k\delta}$, so \eqref{additive} becomes
\begin{equation} \label{pause5}  \left\lVert  \pi_{\theta \#} \mu_b \right\rVert_{L^1(\mathbb{R}^3, \mathcal{H}^2)} \lesssim_{\delta, \epsilon, N} 2^{-JN}+  \sum_{j=J}^{\infty}  \sum_{k= \lceil j \epsilon \rceil}^j  2^{k\delta} \mu(B_{j,k}(\theta)). \end{equation}
Define 
\[ B_{j,k} = \left\{ (\theta, x) \in [0,2\pi) \times \mathbb{R}^3 : x \in B_{j,k}(\theta) \right\}.  \]
Integrating \eqref{pause5} over $[0,2\pi)$ gives 
\begin{align} \notag  \int_0^{2\pi} \left\lVert  \pi_{\theta \#} \mu_b \right\rVert_{L^1(\mathbb{R}^3, \mathcal{H}^2)} \, d\theta &\lesssim_{\delta, \epsilon, N} 2^{-JN}+  \sum_{j=J}^{\infty} \sum_{k= \lceil j \epsilon \rceil}^j   2^{k\delta} \int_0^{2\pi} \mu(B_{j,k}(\theta)) \, d\theta \\
\label{bothsums} &= 2^{-JN}+  \sum_{j=J}^{\infty} \sum_{k= \lceil j \epsilon \rceil}^j  2^{k\delta} \left(\mathcal{H}^1 \times \mu\right)(B_{j,k}). \end{align}
If $(\theta_0,x_0) \in B_{j,k}$, then there is a ``bad'' tube $T_0 \in \mathbb{T}_{j,k,b}$ such that $x_0\in 2T_0$, corresponding to a cap $\tau_0$ with 
\begin{equation} \label{angleT} \lvert (\angle \tau_0)^* - \theta_0\rvert \leq 10^3 2^{k(-1/2 + \delta)}. \end{equation} 
%Assume that $\tau_0$ lies in the forward light cone, so that $(\angle \tau_0)^* = \angle \tau_0 + \pi$ and $(\angle T_0)^* = \angle T_0$. The tube $T_0$ is normal to the cone at $\tau_0$ (i.e.~normal to cone at $2^{-j} \tau_0$), which means that $\angle T_0 =(\angle T_0)^* =  (\angle \tau_0)^* \bmod 2\pi$, where $\angle T_0$ is such that $T_0$ has direction $\frac{1}{\sqrt{2}} \left( \cos \angle T_0, \sin \angle T_0, 1 \right)$. Hence
%\begin{equation} \label{angleT} \left\lvert (\angle T_0)^* - \theta\right\rvert \leq 10^3 2^{k(-1/2 + \delta)}. \end{equation} 
%This holds similarly if $\tau_0$ lies in the backward light cone. 

Since the tube $T_0$ is roughly in the direction of $\gamma(\theta_0)$, the image of $T_0$ under $\pi_{\theta_0}$ is approximately a disc of the same radius, and the ``bad'' tube assumption means the projected measure fails a Frostman condition (to be made precise below). The Orponen-Venieri lemma (Lemma~\ref{venierilemma}) gives a bound on the measure of those points which fail Frostman conditions in many directions, so this will now be used to control \eqref{bothsums}. 

By the definition of ``bad'' tubes in \eqref{badtubestau}, the tube $T_0$ satisfies
\[ (\mu \chi_{B_l})\left(10T_0 \right) \geq  2^{\frac{k}{2} \left( 98 \delta - \alpha^* \right)- \alpha(j-k)}, \]
for some $l$ (see Subsection~\ref{subsectionBl} for the definition of $B_l$). By \eqref{angleT},
\[ 10 T_0 \subseteq \pi_{\theta_0}^{-1}\left( B\left( \pi_{\theta_0}(x_0), 10^7 \cdot 2^{-j + k(1/2 + 2\delta) } \right) \right), \]
and so
\begin{multline} \label{frostmanproj}  (\pi_{\theta_0 \#}\mu \chi_{B_l}) \left( B\left( \pi_{\theta_0}(x_0), 10^7 \cdot 2^{-j + k(1/2 + 2\delta) } \right) \right) \geq \mu \chi_{B_l}\left(10 T_0\right) \\
 \geq 2^{\frac{k}{2} \left( 98 \delta - \alpha^* \right)- \alpha(j-k)}.  \end{multline}
Let $B_{j,k,l}$ be the set of points $(\theta,x) \in B_{j,k}$ such that $x \in 100 \cdot 2^{k\delta}B_l$ and such that the outer parts of \eqref{frostmanproj} hold with $(\theta_0,x_0)$ replaced by $(\theta, x)$. More explicitly,
\begin{multline*} B_{j,k,l} := \Big\{ (\theta,x) \in [0,2\pi) \times 100 \cdot 2^{k\delta}B_l :  \\
(\pi_{\theta \#}\mu \chi_{B_l}) \left( B\left( \pi_{\theta}(x), 10^7 \cdot 2^{-j + k(1/2 + 2\delta) } \right) \right) \geq 2^{\frac{k}{2} \left( 98 \delta - \alpha^* \right)- \alpha(j-k)} \Big\}. \end{multline*}
 Using the parameter $\eta >0$ (yet to be chosen), let 
\begin{equation} \label{Zdefn} Z_{j,k,l} = \left\{ x \in \mathbb{R}^3:  \int_0^{2\pi} \chi_{B_{j,k,l}}(\theta,x) \, d \theta \geq \left( 2^{k(-1/2 + 2\delta) }\right)^{\eta} \right\}. \end{equation}
Roughly speaking, $x \in Z_{j,k,l}$ means that $x$ has lots of ``bad'' tubes passing through it, whose projections are discs failing a Frostman condition. The contribution from points in $Z_{j,k,l}$ will be bounded using the Orponen-Venieri lemma (Lemma~\ref{venierilemma}), whilst the contribution from points outside $Z_{j,k,l}$ will be controlled by the negation of the inequality in \eqref{Zdefn}. The set $B_{j,k}$ is contained in $\bigcup_l B_{j,k,l}$ by the reasoning leading to \eqref{frostmanproj}, so each summand of \eqref{bothsums} satisfies
\begin{align} \notag (\mathcal{H}^1 \times \mu)(B_{j,k}) &\leq \sum_l (\mathcal{H}^1 \times \mu)(B_{j,k,l}) \\
 \notag &= \sum_l (\mathcal{H}^1 \times \mu)\left(B_{j,k,l} \cap \left( [0,2\pi)\times Z_{j,k,l} \right) \right) \\
\label{secondsum} &\quad + \sum_l(\mathcal{H}^1 \times \mu)\left(B_{j,k,l} \setminus \left( [0,2\pi)\times Z_{j,k,l} \right) \right). \end{align} 
%To bound the second sum, write $m \sim l$ if $\supp \mu_{j,k,m} \cap 20B_l \neq \emptyset$ (the number of such $m$ is $\lesssim 1$). Then 
The second sum satisfies
\begin{align*} &\sum_l(\mathcal{H}^1 \times \mu)\left(B_{j,k,l} \setminus \left( [0,2\pi)\times Z_{j,k,l} \right) \right) \\
%&\quad \leq \sum_l \sum_{m \sim l} (\mathcal{H}^1 \times \mu_{j,k,m})\left(B_{j,k,l} \setminus \left( [0,2\pi)\times Z_{j,k,l} \right) \right) \\
%&\quad \leq \sum_l  (\mathcal{H}^1 \times \mu\chi_{B_l})\left(B_{j,k,l} \setminus \left( [0,2\pi)\times Z_{j,k,l} \right) \right) \\
%&\quad \leq \sum_l \sum_{m \sim l} \int_{\mathbb{R}^3 \setminus Z_{j,k,l}} \int_0^{2\pi} \chi_{B_{j,k,l}}(\theta,x) \, d \theta \, d\mu_{j,k,m}(x) \\
&\quad \leq \sum_l  \int_{100 \cdot 2^{k\delta} B_l \setminus Z_{j,k,l}} \int_0^{2\pi} \chi_{B_{j,k,l}}(\theta,x) \, d \theta \, d\mu(x) \\
&\quad \lesssim 2^{300 \cdot k\delta} \left( 2^{k(-1/2 + 2\delta) }\right)^{\eta},  \end{align*} 
by the inequality in \eqref{Zdefn} defining $Z_{j,k,l}$. Since $\delta \ll \eta \ll 1$ ($\delta$ has not been chosen yet, but it may be chosen after $\eta$), putting the previous bound into \eqref{secondsum} and then \eqref{bothsums} results in
\begin{multline} \label{lastsum} \int_0^{2\pi} \left\lVert  \pi_{\theta \#} \mu_b \right\rVert_{L^1(\mathbb{R}^3, \mathcal{H}^2)} \, d \theta \\
\lesssim_{\delta, \epsilon, \eta} 2^{-(J\epsilon \eta)/4}+ \sum_{j=J}^{\infty}  \sum_{k= \lceil j \epsilon \rceil}^j \sum_l 2^{k\delta} (\mathcal{H}^1 \times \mu)(B_{j,k,l} \cap \left( [0,2\pi)\times Z_{j,k,l} \right)). \end{multline}
It remains to bound the sum on the right-hand side. By the definitions of $Z_{j,k,l}$ and $B_{j,k,l}$,
\begin{align} \notag &(\mathcal{H}^1 \times \mu)(B_{j,k,l} \cap \left( [0,2\pi)\times Z_{j,k,l} \right)) \\
\notag &\quad \leq (\mathcal{H}^1 \times \mu)\left( [0,2\pi)\times Z_{j,k,l} \right) \\
\notag &\quad \sim \mu\left(Z_{j,k,l} \right) \\
\notag &\quad = \mu\Big\{ x \in 100 \cdot 2^{k\delta} B_l : \mathcal{H}^1\left\{ \theta \in [0,2\pi) : (\theta,x) \in B_{j,k,l} \right\} \geq \left( 2^{k(-1/2 + 2\delta) }\right)^{\eta} \Big\} \\
\notag &\quad =  \mu\Big\{ x \in 100 \cdot 2^{k\delta} B_l : \mathcal{H}^1\Big\{ \theta \in [0,2\pi) : \\ 
 \notag & \qquad \qquad (\pi_{\theta \#}\mu \chi_{B_l}) \left( B\left( \pi_{\theta}(x), 10^7 \cdot 2^{-j + k(1/2 + 2\delta) } \right) \right)  \geq 2^{\frac{k}{2} \left( 98 \delta - \alpha^* \right)- \alpha(j-k)} \Big\} \\
\notag &\qquad \qquad \qquad \geq \left( 2^{k(-1/2 + 2\delta) }\right)^{\eta} \Big\} \\
\label{tagnu} &\quad \leq 2^{-\alpha(j-k-k\delta)}\lambda\Big\{ x \in \mathbb{R}^3 : \mathcal{H}^1\Big\{ \theta \in [0,2\pi) :\\
\notag &\qquad \quad (\pi_{\theta \#}\lambda) \left( B\left( \pi_{\theta}(x), 2^{k(-1/2 + 2\delta) } \right) \right) \geq \left(2^{k(-1/2 + 2\delta)}\right)^{\alpha^*} \Big\} \geq \left( 2^{k(-1/2 + 2\delta) }\right)^{\eta} \Big\},  \end{align}
where 
\[ \lambda :=  2^{\alpha(j-k-k\delta)} \cdot \left[A_{\#}(\mu \chi_{100 \cdot 2^{k\delta} B_l})\right], \]
and $A$ is the linear scaling map $x \mapsto 10^{-7} \cdot 2^{j-k-k\delta}x$. The measure $\lambda$ is supported in a ball of diameter $\lesssim 1$, and satisfies $c_{\alpha} (\lambda) \lesssim 1$ by the property $c_{\alpha}(\mu) \lesssim 1$ of $\mu= \nu_{j_0}$ inherited from the assumption $c_{\alpha}(\nu) \lesssim 1$ (see \eqref{convolve2}). The total mass of $\lambda$ satisfies $\lambda\left(\mathbb{R}^3 \right) = 2^{\alpha(j-k-k\delta)} \mu\left(100 \cdot 2^{k\delta} B_l\right)$. Applying Lemma~\ref{venierilemma} to \eqref{tagnu} therefore gives 
\begin{equation} \label{refdoublestar} (\mathcal{H}^1 \times \mu)(B_{j,k,l} \cap \left( [0,2\pi)\times Z_{j,k,l} \right)) \leq \mu\left(100 \cdot 2^{k\delta} B_l\right) \left( 2^{k(-1/2 + 2\delta)} \right)^{\eta}, \end{equation}
provided $\delta_0$ and $\eta$ are chosen small enough to ensure that Lemma~\ref{venierilemma} and Theorem~\ref{refinedstrichartz} hold. More precisely, Lemma~\ref{venierilemma} is applied with $\alpha$ in place of $s$ and $\alpha-\alpha^*= \min\left\{ \frac{2\alpha}{3} -1, \frac{1}{2} \right\}+\epsilon$ in place of $\kappa$, so that the $\delta_0$ and $\eta$ given by the lemma depend only on $\epsilon$ and $\dim A$. Since $\delta \ll \eta \ll 1$, putting \eqref{refdoublestar} into \eqref{lastsum} and summing over $l$ gives 
\begin{equation} \label{delta1choiceone} \int_0^{2\pi} \left\lVert  \pi_{\theta \#} \mu_b \right\rVert_{L^1(\mathbb{R}^3, \mathcal{H}^2)} \, d\theta \lesssim_{\epsilon} 2^{-(J\epsilon \eta)/4}. \end{equation}
The choice of $J$ (made later in \eqref{Jchoice}) is $J= (j_0\epsilon)/(2C)$ for a large absolute constant $C$, so by choosing $\delta_1$ small enough (depending on $\epsilon'$, $\eta$ and $\delta_0$) and using $j_0 > \left\lvert \log_2 \delta_1 \right\rvert$, the quantity $j_0^22^{-(J\epsilon \eta)/4}$ will be much smaller than $\epsilon'$ and $\delta_0$, which shows that $\mathcal{H}^1(E) \lesssim_{\epsilon} \epsilon'$  if the first integral in \eqref{bad} dominates. The parameter $\delta_1$ is also taken small enough to ensure that the radius occurring in \eqref{tagnu} is smaller than $\delta_0$, so that the application of Lemma~\ref{venierilemma} is valid.

Now suppose that the second term in \eqref{bad} dominates. Applying Cauchy-Schwarz to the double integral in \eqref{bad} gives 
\begin{align} \notag \mathcal{H}^1(E) &\lesssim  j_0^4  \sup_{\theta \in E} \mathcal{H}^2\left( \widetilde{D_{\theta}^{j_0}} \cap \gamma(\theta)^{\perp} \right) \int_0^{2\pi} \left\lVert  \pi_{\theta \#} \mu_g \right\rVert_{L^2(\mathbb{R}^3, \mathcal{H}^2)}^2  \, d\theta \\
\label{HoneE} &\lesssim  j_0^42^{j_0(s-2-100\sqrt{\epsilon})} \int_0^{2\pi}  \left\lVert  \pi_{\theta \#} \mu_g \right\rVert_{L^2(\mathbb{R}^3, \mathcal{H}^2)}^2 \, d\theta, \end{align} 
by \eqref{scondition}.

For each $\theta \in [0,2\pi)$ let $U_{\theta}: \mathbb{R}^3 \to \mathbb{R}^3$ be the unitary satisfying
\[ U_{\theta}\left(\sqrt{2}\gamma'(\theta)\right) = e_1, \quad U_{\theta} \left(\sqrt{2} \gamma(\theta) \times \gamma'(\theta)\right) = e_2, \quad U_{\theta}\gamma(\theta) = e_3, \]
where the $e_i$ are the standard basis vectors in $\mathbb{R}^3$. This rotates the image of $\pi_{\theta}$ to $\mathbb{R}^2 \times \{0\}$. Since $\mathcal{H}^2$ is a rotation invariant measure on $\mathbb{R}^3$, 
\begin{align} \notag &2^{j_0\left(s-2-100\sqrt{\epsilon}\right)}\int_0^{2\pi} \left\lVert  \pi_{\theta \#} \mu_g \right\rVert_{L^2(\mathbb{R}^3, \mathcal{H}^2)}^2 \, d\theta \\
\notag &\quad = 2^{j_0\left(s-2-100\sqrt{\epsilon}\right)}\int_0^{2\pi} \left\lVert  U_{\theta \#}\pi_{\theta\#}\mu_g \right\rVert_{L^2(\mathbb{R}^2)}^2 \, d\theta  \\
\label{preannying} &\quad = 2^{j_0\left(s-2-100\sqrt{\epsilon}\right)}\int_{B\left(0,2^{j_0(1+\delta) } \right)}  \int_0^{2\pi} \left\lvert \widehat{\mu_g}\left(\eta_1 \sqrt{2} \gamma'(\theta) + \eta_2 \sqrt{2} \gamma(\theta) \times \gamma'(\theta) \right)\right\rvert^2 \, d\theta \, d \eta  \\
\notag &\qquad + 2^{j_0\left(s-2-100\sqrt{\epsilon}\right)} \times \\
\notag &\qquad \quad \int_{\mathbb{R}^2 \setminus B\left(0,2^{j_0(1+\delta) } \right)}  \int_0^{2\pi} \left\lvert \widehat{\mu_g}\left(\eta_1 \sqrt{2} \gamma'(\theta) + \eta_2 \sqrt{2} \gamma(\theta) \times \gamma'(\theta) \right) \right\rvert^2 \, d\theta \, d \eta  \\
\label{annying} &\quad\lesssim 2^{-j_0\epsilon} \int_{\mathbb{R}^2}  \int_0^{2\pi} \lvert \eta\rvert^{s-2-50\sqrt{\epsilon}}  \left\lvert \widehat{\mu_g}\left(\eta_1 \gamma'(\theta) + \eta_2 \gamma(\theta) \times \gamma'(\theta) \right)\right\rvert^2 \, d\theta \, d \eta  \\
\notag &\qquad + 2^{j_0\left(s-2-100\sqrt{\epsilon}\right)} \times \\
\notag &\qquad \quad \int_{\mathbb{R}^2 \setminus B\left(0,2^{j_0(1+\delta) } \right)}  \int_0^{2\pi} \left\lvert \widehat{\mu_g}\left(\eta_1 \sqrt{2} \gamma'(\theta) + \eta_2 \sqrt{2} \gamma(\theta) \times \gamma'(\theta) \right) \right\rvert^2 \, d\theta \, d \eta.  \end{align}
The bound of \eqref{preannying} by \eqref{annying} used 
\[ 2^{j_0\left(s-2-100\sqrt{\epsilon}\right) } \leq 2^{-j _0 \epsilon} 2^{j_0(1+\delta)\left(s-2- 50\sqrt{\epsilon}\right)} \leq 2^{-j_0\epsilon} \left\lvert \eta \right\rvert^{s-2-50\sqrt{\epsilon} },\quad \left\lvert \eta\right\rvert \leq 2^{j_0(1+\delta)}. \] 
The second part of \eqref{annying} will be shown to be negligible. Since $\mu = \nu \ast \phi_{j_0}$ and $\mu_g = \mu-\mu_b$, the second integral in \eqref{annying} satisfies 
\begin{multline*} \int_{\mathbb{R}^2 \setminus B\left(0,2^{j_0(1+\delta) } \right)}  \int_0^{2\pi} \left\lvert \widehat{\mu_g}\left(\eta_1 \sqrt{2} \gamma'(\theta) + \eta_2 \sqrt{2} \gamma(\theta) \times \gamma'(\theta) \right) \right\rvert^2 \, d\theta \, d \eta \\
 \lesssim \int_{\mathbb{R}^2 \setminus B\left(0,2^{j_0(1+\delta) } \right)}  \int_0^{2\pi} \left\lvert \widehat{\phi_{j_0}}\left(\eta_1 \sqrt{2} \gamma'(\theta) + \eta_2 \sqrt{2} \gamma(\theta) \times \gamma'(\theta) \right) \right\rvert^2 \, d\theta \, d \eta \\
 +  \int_{\mathbb{R}^2 \setminus B\left(0,2^{j_0(1+\delta) } \right)}  \int_0^{2\pi} \left( \sum_{T \in \mathbb{T}'} \left\lvert \widehat{M_T\mu}\left(\sqrt{2}\eta_1  \gamma'(\theta) + \sqrt{2}\eta_2  \gamma(\theta) \times \gamma'(\theta) \right) \right\rvert\right)^2 \, d\theta \, d \eta, \end{multline*} 
where 
\[ \mathbb{T}' = \bigcup_{j=J}^{\infty} \bigcup_{k= \lceil j \epsilon \rceil}^j \bigcup_{\tau \in \Lambda_{j,k}} \mathbb{T}_{j,k, \tau, b}. \]

The first integral is $\lesssim_N 2^{-j_0N}$ since $\widehat{\phi_{j_0}}$ decays rapidly outside $B(0, 2^{j_0})$. To bound the second, by summing a geometric series it suffices to prove that 
\begin{multline} \label{pause339}  \int_{B(0,2R) \setminus B(0,R) } \int_0^{2\pi} \left( \sum_{T \in \mathbb{T}'_R 
} \left\lvert \widehat{M_T\mu}\left(\eta_1 \sqrt{2} \gamma'(\theta) + \eta_2 \sqrt{2} \gamma(\theta) \times \gamma'(\theta) \right) \right\rvert\right)^2 \, d\theta \, d \eta \\ \lesssim_N R^{-N},  \end{multline} 
and 
\begin{multline} \label{pause439}  \int_{B(0,2R) \setminus B(0,R) }   \int_0^{2\pi} \left( \sum_{T \in \mathbb{T}' \setminus \mathbb{T}'_R} \left\lvert \widehat{M_T\mu}\left(\eta_1 \sqrt{2} \gamma'(\theta) + \eta_2 \sqrt{2} \gamma(\theta) \times \gamma'(\theta) \right) \right\rvert\right)^2 \, d\theta \, d \eta \\ \lesssim_N R^{-N},  \end{multline} 
for each $R \geq 2^{j_0(1+\delta)}$, where $\mathbb{T}'_R$ consists of those $T\in \mathbb{T}'$ such that 
\begin{equation} \label{annulusR} 2\tau(T) \cap \left[ B(0,2R) \setminus B(0,R) \right] \neq \emptyset. \end{equation}
Since $\left\lvert\mathbb{T}'_R\right\rvert \lesssim R^{O(1)}$, to prove \eqref{pause339} it suffices to show that 
\begin{equation} \label{pause539} \int_{B(0,2R) \setminus B(0,R) }  \int_0^{2\pi}\left\lvert \widehat{M_T\mu}\left(\eta_1 \sqrt{2} \gamma'(\theta) + \eta_2 \sqrt{2} \gamma(\theta) \times \gamma'(\theta) \right) \right\rvert^2 \, d\theta \, d \eta \lesssim_N R^{-N},  \end{equation}
for each $T \in \mathbb{T}'_R$. The assumption in \eqref{annulusR} implies that $2\tau \subseteq \mathbb{R}^3 \setminus B\left( 0, R/1000 \right)$, so by the Schwartz decay of $\widehat{\phi_{j_0}}$,  
\[ \left\lvert \widehat{M_T \mu}(\xi) \right\rvert \lesssim_N R^{-N}, \quad  \xi \in \mathbb{R}^3. \]
%By breaking the domain of integration in \eqref{pause539} into the region corresponding to $2\tau$ and the complement of $2\tau$, and using \eqref{taudist} and \eqref{linfinbound}, 
This gives \eqref{pause539} and therefore \eqref{pause339}. For the remaining bound \eqref{pause439}, any $T \in \mathbb{T}'$ satisfies
\[ \left\lvert\widehat{\eta_T}\left(\xi-\xi'\right)\right\rvert \lesssim_N \dist(\xi, \tau)^{-N}, \quad \xi' \in \tau, \quad \xi \notin 2\tau, \]
and therefore 
\begin{equation} \label{taudist} \left\lvert \widehat{M_T\mu}(\xi) \right\rvert \lesssim_N \dist(\xi, \tau)^{-N}, \quad \xi \notin 2\tau. \end{equation}
 Applying Minkowski's inequality and then \eqref{taudist} to the left-hand side of \eqref{pause439} gives  
\begin{align*} &\int_{B(0,2R) \setminus B(0,R) }   \int_0^{2\pi} \left( \sum_{T \in \mathbb{T}' \setminus \mathbb{T}'_R} \left\lvert \widehat{M_T\mu}\left(\eta_1 \sqrt{2} \gamma'(\theta) + \eta_2 \sqrt{2} \gamma(\theta) \times \gamma'(\theta) \right) \right\rvert\right)^2 \, d\theta \, d \eta \\ 
&\lesssim_N \left( \sum_{T \in \mathbb{T}' \setminus \mathbb{T}'_R } \dist(\tau(T), B(0,2R) \setminus B(0,R) )^{-N} \right)^2. \end{align*}
Summing a geometric series gives \eqref{pause439}. Putting this all together shows that the second integral in \eqref{annying} is $\lesssim_N 2^{-j_0N}$. 
Hence
\begin{equation} \label{L2quant} \eqref{annying} \lesssim_N 2^{-j_0N} +  2^{-j_0\epsilon} \int_{\mathbb{R}^2} \int_0^{2\pi} \lvert \eta\rvert^{s-2-50\sqrt{\epsilon}} \left\lvert \widehat{\mu_g}\left( \eta_1\gamma'(\theta) + \eta_2 \gamma(\theta) \times \gamma'(\theta) \right) \right\rvert^2 \, d\theta \,  d\eta. \end{equation}
 When a change of variables is used to rewrite \eqref{L2quant} in terms of the Fourier energy of $\mu$, the corresponding Jacobian will blow up near the light cone. For this reason, the integral will be broken into two parts; one can be written in terms of the Fourier energy of $\mu$, the other behaves like the $L^2$-average of $\widehat{\mu_g}$ over the cone (at various scales). A key tool used in bounding these $L^2$ averages will be the cone decoupling theorem.

 Define $\kappa$ by
\begin{equation} \label{kappadefn} \kappa =1-\frac{\epsilon}{10^5}. \end{equation}
and define $s'$ by 
\begin{equation} \label{sdoubleprimedefn} s' = s - 50\sqrt{\epsilon} = \max\left\{\frac{4\alpha}{9} + \frac{5}{6}, \frac{2 \alpha + 1}{3} \right\} - 50\sqrt{\epsilon}, \end{equation}
where the second equality comes from the definition of $s$ in \eqref{sprimedefn}. The other parts being similar, it will suffice to bound the part of the integral in \eqref{L2quant} over the positive quadrant, which may be written as
\begin{align} \notag &\int_{\mathbb{R}^2_+} \int_0^{2\pi} \lvert \eta\rvert^{s'-2}\left\lvert \widehat{\mu_g}\left( \eta_1\gamma'(\theta) + \eta_2 \gamma(\theta) \times \gamma'(\theta) \right) \right\rvert^2 \, d\theta \,  d\eta \\
\notag &\quad = \int_{\eta_1 > \eta_2^{\kappa} \geq 0} \int_0^{2\pi} \lvert \eta\rvert^{s'-2}\left\lvert \widehat{\mu_g}\left( \eta_1\gamma'(\theta) + \eta_2 \gamma(\theta) \times \gamma'(\theta) \right) \right\rvert^2 \, d\theta \,  d\eta \\
\notag &\qquad + \int_{0 \leq \eta_1  \leq \eta_2^{\kappa}} \int_0^{2\pi} \lvert \eta\rvert^{s'-2}\left\lvert \widehat{\mu_g}\left( \eta_1\gamma'(\theta) + \eta_2 \gamma(\theta) \times \gamma'(\theta) \right) \right\rvert^2 \, d\theta \,  d\eta. \end{align}
The cross product $\gamma(\theta) \times \gamma'(\theta) = \frac{1}{2} (-\cos \theta, -\sin \theta, 1)$ is in the light cone, whilst $\gamma'(\theta)$ is also tangential to the cone at this same point. The domain of the first integral is therefore away from the cone, whilst the second (more difficult) integral has domain close to it.

By taking $J$ small in comparison to $j_0$ (see \eqref{Jchoice}), it will suffice to prove the following two inequalities.
\begin{claim} \label{L2bound1}
\begin{equation} \label{energy} \int_{\eta_1 > \eta_2^{\kappa} \geq 0} \int_0^{2\pi} \lvert \eta\rvert^{s'-2}\left\lvert \widehat{\mu_g}\left( \eta_1\gamma'(\theta) + \eta_2 \gamma(\theta) \times \gamma'(\theta) \right) \right\rvert^2 \, d\theta \,  d\eta \lesssim_{\epsilon} 2^{CJ}. \end{equation} \end{claim}
\begin{claim} \label{L2bound2} 
\begin{equation} \label{singular} \int_{0 \leq \eta_1  \leq \eta_2^{\kappa}} \int_0^{2\pi} \lvert \eta\rvert^{s'-2}\left\lvert \widehat{\mu_g}\left( \eta_1\gamma'(\theta) + \eta_2 \gamma(\theta) \times \gamma'(\theta) \right) \right\rvert^2 \, d\theta \,  d\eta \lesssim_{\epsilon} 2^{CJ}. \end{equation} \end{claim}
The quantity $C$ is a large absolute constant, to be determined later (implicitly). 

To prove Claim~\ref{L2bound1}, let 
\begin{equation} \label{variablechange} \xi = \xi(\eta_1, \eta_2, \theta) = \eta_1\gamma'(\theta) + \eta_2 \gamma(\theta) \times \gamma'(\theta) . \end{equation}
The scalar triple product formula $\det(a,b,c) = \langle a \times b, c \rangle$ gives
\begin{multline} \label{jacobian10} \frac{\partial(\xi_1,\xi_2,\xi_3)}{\partial(\eta_1, \eta_2,\theta)}  = \det\begin{pmatrix} \gamma' & \gamma \times \gamma' & \eta_1 \gamma'' + \eta_2 \gamma \times \gamma'' \end{pmatrix} = \\
 \left \langle \gamma' \times (\gamma \times \gamma'),  \eta_1 \gamma'' + \eta_2 \gamma \times \gamma''  \right \rangle = \left \langle \gamma/2, \eta_1 \gamma'' + \eta_2 \gamma \times \gamma''  \right \rangle = \eta_1\left \langle \gamma/2, \gamma''  \right \rangle = \frac{- \eta_1}{4}. \end{multline}
 Hence
\begin{equation} \label{jacobian22} \left\lvert \frac{ \partial(\eta_1, \eta_2, \theta) } {\partial(\xi_1,\xi_2, \xi_3) } \right\rvert \gtrsim \left\lvert \xi \right\vert^{-\kappa}, \quad  \eta_1 > \eta_2^{\kappa} \geq 0, \quad \left\lvert \eta \right\vert \geq 1. \end{equation}

It will first be shown that if $2^j \leq \lvert \eta\rvert/ \sqrt{2} \leq 2^{j+1}$ for some $j \geq 0$, then the argument of $\widehat{\mu_g}$ in \eqref{energy} has distance $\gtrsim 2^{j\left(1-\epsilon/\left(10^3\right)\right)}$ from the light cone $\Gamma_{\mathbb{R}}$:
\begin{multline} \label{conedistance} \dist\left(\eta_1 \gamma'(\theta) +\eta_2 \gamma(\theta) \times \gamma'(\theta), \Gamma_{\mathbb{R}} \right) \gtrsim 2^{j\left( 1- \epsilon/\left(10^3\right) \right) }, \\\text{ for }  2^j \leq \lvert \eta\rvert/ \sqrt{2} \leq 2^{j+1}, \quad \eta_1 > \eta_2^{\kappa} \geq 0. \end{multline}
This follows from a calculation using the formula
\begin{equation} \label{distanceformula} \dist\left( (x,y,z), \Gamma_{\mathbb{R}} \right)  = \frac{1}{\sqrt{2}}\left\lvert \sqrt{x^2+y^2} - z \right\rvert, \quad z>0; \end{equation}
if $\eta_2 < \eta_1/100$ then \eqref{distanceformula} implies that the left-hand side of \eqref{conedistance} is $\gtrsim 2^j$, whilst if $\eta_2 \geq \eta_1/100$ then putting the left-hand side of \eqref{conedistance} into \eqref{distanceformula} and applying the inequality
\begin{equation} \label{sqrtaylor} \sqrt{1+x} -1 \gtrsim x, \quad 0 \leq x \sim 1, \end{equation}
followed by the equation $\kappa = 1-\epsilon/(10^5)$ (from \eqref{kappadefn}), gives the lower bound in \eqref{conedistance}.

If $j \geq J$, the function $\chi_{\lvert \xi \rvert \sim 2^j }(\xi)\widehat{\mu_b}(\xi)$ is rapidly decaying outside a $\sim 2^{j(1-\epsilon)}$-neighbourhood of $\Gamma_{\mathbb{R}}$ (by the condition $\dist\left(2\tau_{j,k}, \Gamma_{\mathbb{R}}\right) \lesssim 2^{j-k}$ satisfied by the caps in the sum defining $\mu_b$ (see \eqref{mubaddefn})). Hence if $j \geq J$, \eqref{conedistance} gives 
\[ \left\lvert \widehat{\mu_b}\left( \eta_1\gamma'(\theta) + \eta_2 \gamma(\theta) \times \gamma'(\theta) \right) \right\rvert \lesssim_N 2^{-jN}, \quad 2^j \leq \lvert \eta\rvert/ \sqrt{2} \leq 2^{j+1}, \quad \eta_1 > \eta_2^{\kappa} \geq 0, \]
for arbitrarily large $N$. By the definition $\mu_g = \mu - \mu_b$ of $\mu_g$, 
\begin{multline} \label{largeetabound} \left\lvert \widehat{\mu_g}\left( \eta_1\gamma'(\theta) + \eta_2 \gamma(\theta) \times \gamma'(\theta) \right) \right\rvert \lesssim_N \left\lvert \widehat{\mu}\left( \eta_1\gamma'(\theta) + \eta_2 \gamma(\theta) \times \gamma'(\theta) \right) \right\rvert + 2^{-jN}, \\
 \text{ for } 2^j \leq \lvert \eta\rvert/ \sqrt{2} \leq 2^{j+1}, \quad \eta_1 > \eta_2^{\kappa} \geq 0, \quad j \geq J. \end{multline}
For the other values $j \leq J$,
\begin{equation} \label{smalletabound} \left\lvert \widehat{\mu_g}\left( \eta_1\gamma'(\theta) + \eta_2 \gamma(\theta) \times \gamma'(\theta) \right) \right\rvert \lesssim 2^{CJ}, \quad \left\lvert \eta \right\rvert \lesssim 2^J, \end{equation}
for some large absolute constant $C$. This follows for instance by summing the triangle inequality over the trivial bound
\begin{equation} \label{asterisktrivial} \left\lVert \widehat{M_T\mu} \right\rVert_{L^{\infty} } \lesssim m(T) m(\tau(T)), \end{equation}
for each $\widehat{M_T\mu}$ with $\tau_{j,k}(T)$ satisfying $2^j \lesssim 2^J$, where $m$ is the Lebesgue measure. The other terms can be handled by rapid decay.

 Using \eqref{jacobian22}, \eqref{largeetabound}, \eqref{smalletabound}, and applying the change of variables from \eqref{variablechange} to the integral in \eqref{energy}, results in
\begin{align} \notag &\int_{\eta_1 > \eta_2^{\kappa} \geq 0} \int_0^{2\pi} \lvert \eta\rvert^{s'-2}\left\lvert \widehat{\mu_g}\left( \eta_1\gamma'(\theta) + \eta_2 \gamma(\theta) \times \gamma'(\theta) \right) \right\rvert^2 \, d\theta \,  d\eta \\
\notag  &\quad \lesssim_{\epsilon} 2^{CJ}+  \int_{\mathbb{R}^3} \lvert \xi\rvert^{s'-2-\kappa} \left\lvert \widehat{\mu}\left( \xi \right) \right\rvert^2 \,  d\xi \lesssim_{\epsilon} 2^{CJ}, \end{align}
since $s'-2-\kappa < \alpha-3$ by the definition of $\kappa$ (see \eqref{kappadefn}) and since $\alpha > 3/2$. This finishes the proof of Claim~\ref{L2bound1}.

To prove Claim~\ref{L2bound2}, consider the change of variables
\begin{equation} \label{rchange} (\eta_1, \eta_2) = \left( r \sqrt{1-t^2}, rt \right), \quad r > 0, \quad 0 \leq t \leq 1.  \end{equation}
%\begin{equation} \label{rchange} r^2 = \eta_1^2 + \eta_2^2, \quad \eta_2 = rt, \end{equation}
with Jacobian 
\begin{equation} \label{jacobian20} %\left\lvert\frac{ \partial(r,t)}{\partial(\eta_1,\eta_2)} \right\rvert
 \left\lvert\frac{\partial(\eta_1,\eta_2)}{\partial(r,t)} \right\rvert = \frac{r}{\sqrt{1-t^2} }. %\frac{\sqrt{1-t^2}}{r}.
 \end{equation}
Let 
\begin{equation} \label{gammatdefn}  \gamma_t(\theta) =  \sqrt{1-t^2}\gamma'(\theta) + t\gamma(\theta) \times \gamma'(\theta). \end{equation}

Using the change of variables from \eqref{variablechange}, and the one from \eqref{rchange}, it will be shown that the integral in \eqref{singular} satisfies
\begin{align} \label{dummylabel2} & \int_{0 \leq \eta_1 \leq \eta_2^{\kappa} } \int_0^{2\pi} \lvert \eta\rvert^{s'-2} \left\lvert \widehat{\mu_g}\left( \eta_1\gamma'(\theta) + \eta_2 \gamma(\theta) \times \gamma'(\theta) \right) \right\rvert^2 \, d\theta \,  d\eta \\ 
\label{dummylabel} &\quad \lesssim_{\epsilon} 1 + \sum_{j \geq (1-\kappa)^{-1}} \sum_{k=\left\lfloor 2j(1-\kappa) \right\rfloor-2}^{\left\lfloor j (1-\epsilon) \right\rfloor} \\
\label{preneighbourhood} &\qquad \int_{1-2^{-k}}^{1-2^{-(k+1)}}  \int_{2^{j-1}}^{2^j} \int_0^{2\pi} \frac{2^{j(s'-1)}}{\sqrt{1-t^2}} \left\lvert \widehat{\mu_g}\left( r  \gamma_t(\theta) \right) \right\rvert^2 \, d\theta \, dr \, dt \\
\label{presurface} &\qquad +  \sum_{j \geq (1-\kappa)^{-1}} \int_{1-2^{-j(1-\epsilon)}}^1 \int_{2^{j-1}}^{2^j} \int_0^{2\pi} \frac{2^{j(s'-1)}}{\sqrt{1-t^2}} \left\lvert \widehat{\mu_g}\left( r  \gamma_t(\theta) \right) \right\rvert^2 \, d\theta \, dr \, dt \\
\notag &\quad \lesssim_{\epsilon} 1 + \\
\label{neighbourhood}  &\qquad  \sum_{j \geq  (1-\kappa)^{-1}} \sum_{k=\left\lfloor 2j(1-\kappa) \right\rfloor-2}^{\left\lfloor j (1-\epsilon) \right\rfloor} 2^{js' + k/2} \int_{\mathcal{N}_{C2^{-k}}(\Gamma) \setminus \mathcal{N}_{C^{-1} 2^{-k}} (\Gamma_{\mathbb{R}})}  \left\lvert \widehat{\mu_g}\left( 2^j \xi \right) \right\rvert^2 \,   d\xi \\
\label{surface} &\qquad+  \sum_{j \geq (1-\kappa)^{-1}} \int_{1-2^{-j(1-\epsilon)}}^1 \int_{2^{j-1}}^{2^j} \int_0^{2\pi} \frac{2^{j(s'-1)}}{\sqrt{1-t^2}} \left\lvert \widehat{\mu_g}\left( r  \gamma_t(\theta) \right) \right\rvert^2 \, d\theta \, dr \, dt, \end{align}
where $C$ is a large constant to be chosen in a moment. To simplify notation $\Gamma$ is used here to denote the set of points $(\xi,\left\lvert \xi\right\rvert)$ in $\mathbb{R}^3$ with $1/10 \leq \left\lvert \xi \right\rvert \leq 10$, and $\Gamma_{\mathbb{R}}$ is the entire light cone. The purpose of this decomposition is to concentrate the integral on regions where $\left\lvert \xi \right\rvert$ and the Jacobian are roughly constant. The first bound, in \eqref{dummylabel}, \eqref{preneighbourhood} and \eqref{presurface}, is a consequence of the change of variables in \eqref{rchange} and a dyadic decomposition. To be more precise, the part of the integral in \eqref{dummylabel2} with $\lvert \eta \rvert \lesssim 2^{(1-\kappa)^{-1}}$ is absorbed into the $\lesssim_{\epsilon} 1$ term in \eqref{dummylabel} (since $\kappa$ depends only on $\epsilon$; see \eqref{kappadefn}). This explains the condition $j \geq (1-\kappa)^{-1}$ in the summation in \eqref{dummylabel}, which arises after changing variables in \eqref{dummylabel2}, using \eqref{rchange}, and dyadically decomposing the values of $r$. Given $j$, the summation in $k$ comes from then decomposing the domain of the integral in $t$ according to the dyadic value of $1-t$. The integral over the range $1-2^{-j(1-\epsilon)} \leq t \leq 1$ is handled by \eqref{presurface}, which justifies the upper restriction $k \leq \left\lfloor j(1-\epsilon) \right\rfloor$ on the summation in $k$, in \eqref{preneighbourhood}. It remains to justify the lower restriction $k \geq \left\lfloor 2j(1-\kappa) \right\rfloor-2$ (in place of the default restriction $k \geq 0$). The condition $\eta_1 \leq \eta_2^{\kappa}$ on the domain in \eqref{dummylabel2} can be written as 
\[ r \sqrt{1-t^2} \leq r^{\kappa} t^{\kappa}. \]
Applying the trivial bounds $t^{\kappa} \leq 1$ and $\sqrt{1-t} \leq \sqrt{1-t^2}$ to this yields  
\[ t \geq 1- r^{2(\kappa-1)}. \]
Since $r \geq 2^{j-1}$ for a given fixed $j$, and $t \leq 1-2^{-(k+1)}$ for a given fixed $k$, this implies that
\[ k \geq 2(j-1)(1-\kappa) -1 \geq 2j(1-\kappa)  -2, \]
which explains the lower bound $k\geq \left\lfloor 2j(1-\kappa) \right\rfloor-2$ in the summation over $k$. This proves that \eqref{dummylabel2} is bounded by \eqref{dummylabel}, \eqref{preneighbourhood} and \eqref{presurface}.

To prove the second bound it suffices to show that each summand in \eqref{preneighbourhood} is bounded by the corresponding one in \eqref{neighbourhood}. This will be done by changing variables in each summand (using \eqref{jacobian10} and \eqref{jacobian20}), and verifying that the condition
\begin{equation} \label{equality} r\sqrt{1-t^2}\gamma'(\theta) + rt\gamma(\theta) \times \gamma'(\theta) = 2^j \xi,  \end{equation}
where 
\[  2^{j-1} \leq r \leq 2^j, \quad 2^{-(k+1)} \leq 1-t \leq 2^{-k}, \]
implies that $\xi \in \mathcal{N}_{C2^{-k}}(\Gamma)  \setminus \mathcal{N}_{C^{-1} 2^{-k}} (\Gamma_{\mathbb{R}})$, provided $C$ is large enough. Division of \eqref{equality} by $2^j$ gives 
\[ \xi = \lambda_1 \gamma'(\theta) + \lambda_2 \gamma(\theta) \times \gamma'(\theta),  \quad 2^{-2} \cdot 2^{-k/2} \leq \lambda_1 \leq 2 \cdot 2^{-k/2}, \quad  \frac{1}{4} \leq \lambda_2 \leq 1. \]
if $k \geq 1$; if $k=0$ this still holds but without the lower bound $\lambda_2 \geq 1/4$. Applying \eqref{distanceformula} to this gives 
\[ \dist(\xi, \Gamma_{\mathbb{R}}) = \frac{1}{\sqrt{2}} \left( \sqrt{ \left( \frac{\lambda_1}{\sqrt{2}} \right)^2 + \left( \frac{\lambda_2}{2} \right)^2 } - \frac{\lambda_2}{2} \right) \sim \lambda_1^2 \sim 2^{-k}, \]
if $k  \geq 1$, and if $k=0$ this still holds by \eqref{sqrtaylor} since then $\lambda_2 \lesssim \lambda_1$. This verifies that \eqref{preneighbourhood} is bounded by \eqref{neighbourhood}.
%By computing the cross product explicitly, this can be written as 
%\begin{align} \label{lower3} \xi &= \frac{\lambda_2}{\sqrt{2}} \gamma(\theta+\pi) -\lambda_1\gamma'(\theta+\pi) \\
%\label{lower4} &= \frac{\lambda_2}{\sqrt{2}}\gamma(\theta+\pi+h) + O\left(h^2\right), \quad h = \frac{-\lambda_1 \sqrt{2}}{\lambda_2}, \end{align}
%where the notation $O\left(h^2\right)$ means the error is a vector in $\mathbb{R}^3$ of magnitude $\lesssim h^2$, and the implicit constant in $O\left(h^2\right)$ is uniform. Then $h$ satisfies $\lvert h\rvert \lesssim \lambda_1$, and therefore $\dist(\xi, \Gamma) \lesssim \lambda_1^2 \lesssim 2^{-k}$ by \eqref{lower4}. Hence $\xi \in \mathcal{N}_{C2^{-k}}(\Gamma)$ provided $C$ is chosen large enough.  Moreover, \eqref{lower3}, the lower bound on $\lambda_1$ in \eqref{perturb} and the curvature of the cone imply that $\dist(\xi, \Gamma_{\mathbb{R}}) \gtrsim \lambda_1^2 \gtrsim 2^{-k}$, and so $\xi \notin \mathcal{N}_{C^{-1}2^{-k}}(\Gamma_{\mathbb{R}})$ provided $C$ is chosen large enough. This verifies that \eqref{preneighbourhood} is bounded by \eqref{neighbourhood}.

It remains to bound the sums in \eqref{neighbourhood} and \eqref{surface}. The terms in \eqref{neighbourhood} will be bounded first; the terms in \eqref{surface} will be similar to those terms in \eqref{neighbourhood} with $k$ close to $j$. Fix a pair $(j,k)$ occurring in \eqref{neighbourhood}, such that the corresponding $j$ satisfies $j \geq 2J$ (the other terms will be bounded trivially). By \eqref{unitypartition} and the definition $\mu_g = \mu-\mu_b$ of $\mu_g$, the integral in the summand of \eqref{neighbourhood} satisfies
\begin{align}\label{placeholder} &\int_{\mathcal{N}_{C2^{-k}}(\Gamma) \setminus \mathcal{N}_{C^{-1}2^{-k}}(\Gamma_{\mathbb{R}})}  \left\lvert \widehat{\mu_g}\left( 2^j \xi \right) \right\rvert^2 \,   d\xi \\
\notag &\quad = \frac{1}{2^{3j}} \int_{\mathcal{N}_{C2^{j-k}}(2^j\Gamma) \setminus \mathcal{N}_{C^{-1}2^{j-k}}(\Gamma_{\mathbb{R}})}  \left\lvert \widehat{\mu_g}\left( \xi \right) \right\rvert^2 \,   d\xi \\
\label{asteriskk} &\quad = \frac{1}{2^{3j}} \int_{\mathcal{N}_{C2^{j-k}}(2^j\Gamma) \setminus \mathcal{N}_{C^{-1}2^{j-k}}(\Gamma_{\mathbb{R}})}  \left\lvert \sum_{\substack{j' \\
2^{j'} \sim 2^j}} \sum_{\substack{k' \\
2^{k'} \sim 2^k} } \sum_{\tau \in \Lambda_{j',k'}} \sum_{\substack{T \in \mathbb{T}_{j',k', \tau }}} \widehat{M_T\mu}(\xi) \right. \\
\notag &\qquad \left.  -\sum_{j'=J}^{\infty} \sum_{k'= \lceil j' \epsilon \rceil}^{j'} \sum_{\tau \in \Lambda_{j',k'}} \sum_{T \in \mathbb{T}_{j',k', \tau, b}} \widehat{M_T\mu}(\xi) \right\rvert^2 \,   d\xi  + O\left( 2^{-kN} \right) \\
\label{nonerror} &\lesssim \frac{1}{2^{3j}} \int_{\mathcal{N}_{C2^{j-k}}(2^j\Gamma) \setminus \mathcal{N}_{C^{-1}2^{j-k}}(\Gamma_{\mathbb{R}})}  \left\lvert \sum_{\substack{j' \\
2^{j'} \sim 2^j}} \sum_{\substack{k' \\
2^{k'} \sim 2^k} } \sum_{\tau \in \Lambda_{j',k'}} \sum_{\substack{T \in \mathbb{T}_{j',k', \tau,g}}} \widehat{M_T\mu}(\xi) \right\rvert^2  \\
\label{errorterms} &\qquad + \frac{1}{2^{3j}} \left( \sum_{T \in \mathbb{T}''} \left(\int_{\mathcal{N}_{C2^{j-k}}(2^j\Gamma) \setminus \mathcal{N}_{C^{-1}2^{j-k}}(\Gamma_{\mathbb{R}})} \left\lvert \widehat{M_T\mu}(\xi) \right\rvert^2 \,   d\xi\right)^{1/2} \right)^2 + O\left( 2^{-kN} \right), \end{align}
where $\mathbb{T}''$ consists of those $T$ in 
\[ \bigcup_{j'=J}^{\infty} \bigcup_{k'= \lceil j' \epsilon \rceil}^{j'} \bigcup_{\tau \in \Lambda_{j',k'}} \mathbb{T}_{j',k', \tau, b}, \]
such that $2\tau(T)$ does not intersect $\mathcal{N}_{C2^{j-k}}(\Gamma) \setminus \mathcal{N}_{C^{-1}2^{j-k}}(\Gamma_{\mathbb{R}})$. The notation $2^{j'} \sim 2^j$ and $2^{k'} \sim 2^k$ means that 
\[ \lvert j-j'\rvert \leq C', \quad \lvert k-k'\rvert \leq C', \]
for some large constant $C'$ depending only on $C$. In \eqref{asteriskk} $\widehat{\mu}$ was first replaced by 
\begin{equation} \label{lastpause} \sum_{\substack{j' \\
2^{j'} \sim 2^j}} \sum_{\substack{k' \\
2^{k'} \sim 2^k} } \sum_{\tau \in \Lambda_{j',k'}} \psi_{\tau} \widehat{\mu}, \end{equation}
which is valid since the $\psi_{\tau}$'s form a partition of unity (see \eqref{unitypartition}), and the $\psi_{\tau}$'s coming from those $(j',k')$ not occurring in \eqref{lastpause} are vanishing on the domain of integration in \eqref{asteriskk}. It was then used that each $\psi_{\tau} \widehat{\mu}$ is equal to $\sum_{T \in \mathbb{T}_{j', k', \tau}} \widehat{M_T \mu} + O\left(2^{-kN}\right)$.  Furthermore, the assumption $j \geq 2J$ was needed in \eqref{nonerror} to restrict the sum to the ``good'' tubes.  By applying \eqref{taudist} and summing a geometric series, the term in \eqref{errorterms} is $\lesssim_N 2^{-kN}$. 

For the remaining term in \eqref{nonerror}, let $\rho = 2^{j-k}$ and define 
%$\mu_{g,\rho^*}$ by $\widehat{\mu_{g,\rho^*}}(\xi) = \widehat{\mu_g}(\rho \xi)$. The function $\mu_g$ is supported in $B(0,100)$ by construction, and so  $\mu_{g, \rho^*} = \rho^{-3}\mu_g( \cdot/ \rho)$ is supported in $B(0, 100 \rho)$. Correspondingly, define 
$M_{T, \rho^*}\mu$ by $\widehat{M_{T, \rho^*}\mu}(\xi) =\widehat{ M_T\mu}(\rho \xi)$. Define $\eta_{T, 1/\rho}$ by $\eta_{T,1/\rho}(x) = \eta_T(x/\rho)$, and define $\psi_{\tau,\rho}$ by $\psi_{\tau,\rho}(\xi) = \psi_{\tau}(\rho \xi)$. Let $\{B_m\}_m$ be a finitely overlapping cover of $B(0,100\rho)$ by unit balls and let $\{\vartheta_m\}_m$ be a smooth partition of unity subordinate to this cover. Let $\phi$ be a non-negative smooth function supported in $\mathcal{N}_{2C}(2^k \Gamma) \setminus \mathcal{N}_{(2C)^{-1}}(\Gamma_{\mathbb{R}})$, with $\phi \sim 1$ on $\mathcal{N}_C(2^k \Gamma) \setminus \mathcal{N}_{C^{-1}}(\Gamma_{\mathbb{R}})$, such that 
\begin{equation} \label{schwartzdecay} \left\lvert \widecheck{\phi}(x) \right\rvert \lesssim_N 2^{2k} \lvert x\rvert^{-N}, \quad x \in \mathbb{R}^3. \end{equation}
Such a function can be constructed by summing over a smooth partition of unity subordinate to a finitely overlapping cover of $\mathcal{N}_C(2^k \Gamma) \setminus \mathcal{N}_{C^{-1}}(\Gamma_{\mathbb{R}})$ by unit balls. This set has volume $\lesssim 2^{2k}$, so pairing the triangle inequality with the Schwartz decay of each function in the partition gives \eqref{schwartzdecay}. Define
\[ \mu_{g,j,k,\rho^*} = \sum_{\substack{j' \\
2^{j'} \sim 2^j}} \sum_{\substack{k' \\
2^{k'} \sim 2^k} } \sum_{\tau \in \Lambda_{j',k'}} \sum_{\substack{T \in \mathbb{T}_{j',k', \tau, g} }} M_{T, \rho^*} \mu. \]
Then by changing variables and using that $\eqref{errorterms} \lesssim_N 2^{-kN}$, 
\begin{align}  %&\lesssim \int \phi\left(2^k \xi\right) \left\lvert \widehat{\mu_{g,\rho^*}}\left( 2^k \xi \right) \right\rvert^2 \,   d\xi \\
\label{restrictwavepackets} &\eqref{placeholder} \lesssim_N \frac{1}{2^{3k}}\int \phi\left(\xi\right) \left\lvert \sum_m  \mathcal{F}\left(\vartheta_m\mu_{g,j,k,\rho^*}\right)\left(\xi\right) \right\rvert^2 \,   d\xi +2^{-kN} \\
\label{diag3} &\quad \lesssim \frac{\rho^{3\epsilon^2}}{2^{3k}} \sum_m \int \phi(\xi) \left\lvert \mathcal{F}\left(\vartheta_m \mu_{g,j,k,\rho^*}\right)\left(\xi\right) \right\rvert^2 \, d\xi  \\
\label{offdiag3} &\qquad + \frac{1}{2^{3k}}\sum_{\dist(B_m, B_n) \geq \rho^{\epsilon^2}} \left\lvert \int \phi(\xi)\mathcal{F}\left(\vartheta_m \mu_{g,j,k\rho^*}\right)(\xi ) \overline{\mathcal{F}\left(\vartheta_n \mu_{g,j,k,\rho^*}\right)(\xi )} \, d\xi  \right\rvert \\
\notag &\qquad \quad + 2^{-kN}. \end{align}
% The restriction of the wave packets in \eqref{restrictwavepackets} to only those occurring in \eqref{annoydefn} is where the separation of the support of $\phi$ from the cone is used; when $k$ is very close to $j$ the separation is unnecessary. 
%By rapid decay the other wave packets are negligible on the domain of integration in \eqref{restrictwavepackets}, and are absorbed into the error term $2^{-kN}$. 
%The restrictions $2^j \sim 2^{j'}$ and $2^k \sim 2^{k'}$ in \eqref{annoydefn} refer to the notation in \eqref{obvious20}. 
By Plancherel and by \eqref{schwartzdecay}, the summands in \eqref{offdiag3} satisfy 
\begin{align} \notag  &\left\lvert \int \phi(\xi)\mathcal{F}\left(\vartheta_m \mu_{g,j,k,\rho^*}\right)(\xi ) \overline{\mathcal{F}\left(\vartheta_n \mu_{g,j,k,\rho^*}\right)}(\xi ) \, d\xi  \right\rvert \\
\notag &\quad= \left\lvert \int \int \widecheck{\phi}(x-y) \vartheta_m(x) \vartheta_n(y) \, d\mu_{g,j,k,\rho^*}(x) \, d\overline{\mu_{g,j,k,\rho^*}}(y) \right\rvert  \\
\label{offdiag4} &\quad\lesssim_N  2^{O(j)}\rho^{-\epsilon^2 N}, \end{align}
where the last line used the trivial bound (see \eqref{asterisktrivial}):
\[ \left\lVert \mu_{g,j,k,\rho} \right\rVert_{\infty} \leq 2^{O(j)}. \]
By the constraint $k \leq j(1-\epsilon)$, $\rho$ is at least as large as $2^{j\epsilon}$, so by taking $N$ much larger than $\epsilon^{-3}$, the sum over the terms in \eqref{offdiag3} is $\lesssim_{\epsilon} 2^{-100j}$, which will be better than the bound obtained on the other part and hence absorbed into it. The sum in \eqref{diag3} will be controlled through the following.
\begin{claim} \label{L2bound3} Fix $j$, $k$ and $m$ as in \eqref{diag3}. Let $p=6$ and $p'=3$.  If $j \geq 2J$ then for arbitrarily large $N$, 
\[ \int \phi(\xi) \left\lvert \mathcal{F} \left(\vartheta_m \mu_{g,j,k,\rho^*}\right)\left(\xi\right) \right\rvert^2 \, d\xi \lesssim_{\epsilon, N} 2^{k \left( 3-\frac{3}{p'} + \frac{2\alpha}{p'} - \frac{\alpha^*}{p'} + 5\epsilon \right)-j \alpha} \lVert \mu_m\rVert + 2^{-kN}, \]
where $\mu_m$ is the restriction of $\rho_{\#}\mu$ to $2^{10k\delta}B_m$. \end{claim}
The proof of this claim will use the refined Strichartz inequality for the cone (Theorem~\ref{refinedstrichartz}), the precise statement and proof of which is postponed until Subsection~\ref{strichartz}. Summing the bound in Claim~\ref{L2bound3} over $m$, $k$ and $j$ will then give a bound on \eqref{neighbourhood}. A very similar strategy will then be used to bound \eqref{surface}, which will conclude the proof of Claim~\ref{L2bound2}.

% In order to prove Claim~\ref{L2bound3}, by partitioning the wave packet decomposition of $\mu_{g, \rho^*}$ into $\lesssim 1$ measures and applying the triangle inequality, it may be assumed that any two caps in the wave packet decomposition of $\mu_{g, \rho^*}$ are non-adjacent. Similarly it may be assumed that any two tubes in the wave packet decomposition of $\mu_{g, \rho^*}$ corresponding to the same cap $\tau$ are non-adjacent. 
For fixed $m$, 
\begin{multline} \label{pause10} \int \phi(\xi) \left\lvert \mathcal{F} \left(\vartheta_m \mu_{g,j,k,\rho^*}\right)\left(\xi\right) \right\rvert^2 \, d\xi \\
\lesssim \int \left\lvert  \sum_{\substack{j' \\
2^{j'} \sim 2^j}} \sum_{\substack{k' \\
2^{k'} \sim 2^k} } \sum_{\tau \in \Lambda_{j',k'}} \sum_{\substack{T \in \mathbb{T}_{j',k', \tau, g} \\ \rho T \cap B_m \neq \emptyset }} \mathcal{F} \left( \vartheta_m M_{T, \rho^*} \mu \right)(\xi) \right\rvert^2 \, d\xi. \end{multline}
The following inequality will be shown to follow from ``essential orthogonality'' of wave packets:
\begin{equation} \label{plancherelstar} \eqref{pause10} \lesssim_N 2^{-kN} +  \sum_{\substack{j' \\
2^{j'} \sim 2^j}} \sum_{\substack{k' \\
2^{k'} \sim 2^k} } \sum_{\tau \in \Lambda_{j',k'}} \sum_{\substack{T \in \mathbb{T}_{j',k', \tau, g} \\ \rho T \cap B_m \neq \emptyset }}\int \left\lvert(M_{T, \rho^*} \mu)(x) \right\rvert^2 \, dx. \end{equation}
The argument for this is similar to the usual derivation of the wave packet decomposition (see e.g.~\cite[Section~3]{guthpoly}), but will be summarised here for completeness. By Plancherel the term $\vartheta_m$ in \eqref{pause10} can be removed first since it is $\lesssim 1$ everywhere. By the triangle inequality the sums in $j'$ and then $k'$ can then be moved outside the integral, since there are $\lesssim 1$ terms in each sum. The sum in $\tau$ can be expanded out using the inner product on $L^2(\mathbb{R}^3)$; for a given $\tau \in \Lambda_{j',k'}$, there are $\lesssim 1$ sets $\tau' \in \Lambda_{j',k'}$ such that $2 \tau' \cap \tau \neq \emptyset$, whilst the contribution of the other terms is $\lesssim_N 2^{-kN}$ by \eqref{taudist}. The AM-GM inequality then allows the sum over $\tau$ to be moved outside the integral. By Plancherel the sum over $T$ can then be moved outside the integral since the sets $T$ are finitely overlapping, so this proves \eqref{plancherelstar}. 

The decay term can be ignored, so it remains to bound the sum in \eqref{plancherelstar}. The integral in \eqref{plancherelstar} is roughly the $L^2$ average of $\widehat{\mu_g}$ over the cone (ignoring rescaling). The usual method of controlling the $L^2$ averages of $\widehat{\mu}$ over surfaces uses duality and Cauchy-Schwarz to reduce the problem to bounding $\lVert Ef\rVert_{L^2(H)}$, where $Ef$ is a Fourier extension operator and $H$ is a weight function corresponding to $\mu$ (as in e.g.~\cite{duzhang}). Since $\mu_g$ is not a positive measure this duality step will be slightly more involved, but it still works by extracting the measure $\mu$ out of $\mu_g$ as follows. By Plancherel,
\begin{align} \notag &\sum_{\substack{j' \\
2^{j'} \sim 2^j}} \sum_{\substack{k' \\
2^{k'} \sim 2^k} } \sum_{\tau \in \Lambda_{j',k'}} \sum_{\substack{T \in \mathbb{T}_{j',k', \tau, g} \\ \rho T \cap B_m \neq \emptyset }}\int \left\lvert(M_{T, \rho^*} \mu)(x) \right\rvert^2 \, dx \\
\notag &\quad = \sum_{\substack{j' \\
2^{j'} \sim 2^j}} \sum_{\substack{k' \\
2^{k'} \sim 2^k} } \sum_{\tau \in \Lambda_{j',k'}} \sum_{\substack{T \in \mathbb{T}_{j',k', \tau, g} \\ \rho T \cap B_m \neq \emptyset }} \\
\notag &\qquad \frac{1}{\rho^3}\int (M_{T, \rho^*}\mu)(x) \eta_T\left(\frac{x}{\rho}\right) \left[ \int \overline{\widecheck{\psi_{\tau}}\left(\frac{x}{\rho}-y\right)} \, d\mu(y) \right] \, dx  \\
\label{convolution} &\quad =  \int \sum_{\substack{j' \\
2^{j'} \sim 2^j}} \sum_{\substack{k' \\
2^{k'} \sim 2^k} } \sum_{\tau \in \Lambda_{j',k'}} \sum_{\substack{T \in \mathbb{T}_{j',k', \tau, g} \\ \rho T \cap B_m \neq \emptyset }} \left[(M_{T, \rho^*}\mu) \eta_{T, 1/\rho} \right] \ast \widecheck{\psi_{\tau,\rho}} \, d(\rho_{\#}\mu),  \end{align} %\endgroup
where $\rho_{\#}\mu$ is the pushforward of $\mu$ under $y \mapsto \rho y$. To compress the notation, let $\mathbb{W}$ be the entire set of (inflated) tubes $T$ occurring in \eqref{convolution}:
\[ \mathbb{W} = \bigcup_{\substack{j' \\ 2^{j'} \sim 2^j}} \bigcup_{\substack{k' \\ 2^{k'} \sim 2^k} } \bigcup_{\tau \in \Lambda_{j',k'}} \left\{ S= 2 \rho T:  T \in \mathbb{T}_{j',k', \tau, g} :  \rho T \cap B_m \neq \emptyset \right\}, \]
 and for each such $S$ let 
\begin{equation} \label{fdefn} f_S = \left[(M_{T, \rho^*}\mu) \eta_{T, 1/\rho} \right] \ast \widecheck{\psi_{\tau,\rho}}. \end{equation} 
By Cauchy-Schwarz (with respect to the measure $\mu_m$), 
\begin{equation} \label{couchy} \eqref{convolution} = \left\lvert\int \sum_{S \in \mathbb{W}} f_S \, d\rho_{\#}\mu \right\rvert \lesssim_N \left(\int  \left\lvert\sum_{S \in \mathbb{W}} f_S \right\rvert^2 \, d\mu_m \right)^{1/2}\lVert \mu_m\rVert^{1/2} +2^{-kN}, \end{equation}
where $\mu_m$ is the restriction of $\rho_{\#}\mu$ to $2^{10k \delta}B_m$. After some minor adjustments and mollifications, this will be in a form more amenable to the refined Strichartz inequality; the application of which will be the final major step in the proof. Each $f_S$ is essentially supported in a rescaled tube $S$ of dimensions
\[\sim 2^{k(-1/2+\delta)} \times 2^{k(-1/2+\delta)} \times 2^{k\delta}. \]
Each $f_S$ has Fourier transform $\widehat{f_S}$ supported in a cap of dimensions
\[ \sim 2^{k/2} \times 2^k \times 1, \]
in a $\sim 1$ neighbourhood of the cone $2^k \Gamma$. Each rescaled tube $S$ has direction normal to the corresponding cap $\tau$. Let $f=\sum_{S \in \mathbb{W}} f_S$, so that $\widehat{f}$ is supported in a ball around the origin of radius $C''2^k$ for some sufficiently large constant $C''$. Let $\varphi$ be a smooth non-negative even bump function equal to 1 on $B(0,C'')$ and supported in $B(0,2C'')$, and let $\varphi_k(\xi) = \varphi( \xi/ 2^k )$. Then $\widehat{f}=\widehat{f} \varphi_k$, and so by Cauchy-Schwarz,
\[ \left\lvert f \right\rvert^2 = \left\lvert f \ast \widecheck{\varphi_k} \right\rvert^2 \lesssim \left\lvert f\right\rvert^2 \ast \left\lvert\widecheck{\varphi_k} \right\rvert. \]
Putting this into the integral in \eqref{couchy} gives
\begin{multline*} \int\left\lvert\sum_{S \in \mathbb{W}} f_S \right\rvert^2 \, d\mu_m =  \int \left\lvert f\right\rvert^2 \, d\mu_m \lesssim \int \left\lvert f\right\rvert^2 \ast \left\lvert\widecheck{\varphi_k} \right\rvert\, d\mu_m \\
= \int \left\lvert f\right\rvert^2 \, d\left(\mu_m \ast \left\lvert\widecheck{\varphi_k}\right\rvert \right)  \lesssim_N \int \left\lvert f\right\rvert^2 \, d\mu_{m,k},  \end{multline*} 
where 
\[ \mu_{m,k} := \mu_m \ast \zeta_{k,N}, \qquad \zeta_{k,N}(x) := \frac{ 2^{3k} }{1+ 2^{kN} \left\lvert x\right\rvert^N }. \]
It remains to bound $\int \left\lvert f\right\rvert^2 \, d\mu_{m,k}$. By dyadic pigeonholing and the triangle inequality, there is a subset $\mathbb{W}' \subseteq \mathbb{W}$ such that $\lVert f_S\rVert_2$ is constant up to a factor of 2 as $S$ ranges over $\mathbb{W}'$, and 
\begin{equation} \label{pigeoncheck}  \int \left\lvert f\right\rvert^2 \, d\mu_{m,k} \lesssim_N \log\left( 2^j\right)^2 \int \left\lvert \sum_{S \in \mathbb{W}'} f_S \right\rvert^2 \, d\mu_{m,k}+ 2^{-kN}. \end{equation}
Strictly speaking, in order for this pigeonholing argument to work, some trivial lower and upper bounds are needed on $\lVert f_S\rVert_2$ to restrict the range of dyadic values. By \eqref{pigeonbound} below each $\lVert f_S\rVert_2$ satisfies 
\[ \lVert f_S\rVert_2 \lesssim \left( \sum_{S = 2\rho T \in \mathbb{W}} \lVert M_{T, \rho^*}\mu \rVert_2^2 \right)^{1/2}. \]
Since the bound on \eqref{convolution} to eventually be proved below is \eqref{pigeonconstraint}, by \eqref{dimensionboundmu} the contribution to \eqref{convolution} of the sum over those $f_S$ with 
\[ \left\lVert f_S \right\rVert_2  \leq  2^{-10000j} \left( \sum_{S = 2\rho T \in \mathbb{W}} \lVert M_{T, \rho^*}\mu \rVert_2^2 \right)^{1/2}, \]
will automatically satisfy \eqref{pigeonconstraint}, so by the triangle inequality it may be assumed that all every $S \in \mathbb{W}$ has $\left\lVert f_S\right\rVert_2$ lying in a range of $\lesssim \log\left( 2^j \right)$ dyadic values, so that the inequality \eqref{pigeoncheck} holds.

Cover the support of $\mu_{m,k}$ with $2^{-k/2}\mathbb{Z}^3$-lattice cubes $Q$, and partition the cubes $Q$ according to the dyadic number of tubes $S \in \mathbb{W}'$ such that $2S$ intersects $Q$. By pigeonholing again, there is a union $Y$ of $2^{-k/2}\mathbb{Z}^3$-lattice cubes $Q$, contained in a ball of radius 1, such that each $Q$ intersects the same dyadic number $M$ of tubes $2S$ with $S \in \mathbb{W}'$ (up to a factor of 2), and such that
\begin{equation} \label{pause15}  \int \left\lvert \sum_{S \in \mathbb{W}'} f_S\right\rvert^2 \, d\mu_{m,k} \lesssim \log\left( 2^k \right) \int_Y \left\lvert \sum_{S \in \mathbb{W}'} f_S\right\rvert^2 \, d\mu_{m,k}. \end{equation}
Recall that $p=6$ and $p'= 3$, so that $\frac{1}{2} = \frac{1}{p} + \frac{1}{p'}$. By Hölder's inequality with respect to Lebesgue measure, the integral in the right-hand side of \eqref{pause15} satisfies
\begin{equation} \label{holder} \int_Y \left\lvert \sum_{S \in \mathbb{W}'} f_S\right\rvert^2 \, d\mu_{m,k} \leq \left\lVert  \sum_{S \in \mathbb{W}'} f_S\right\rVert_{L^p(Y)}^2 \left( \int_Y \mu_{m,k}^{p'/2} \right)^{2/p'}. \end{equation} 
Hence it remains to bound each term of the product in \eqref{holder}.  Recall that each tube $S \in \mathbb{W}'$ is of the form $S=2\rho  T$, where $T$ is a tube of dimensions 
\[ \sim 2^{-j+k\left(1/2+\delta\right)} \times 2^{-j+k\left(1/2+\delta\right)} \times (10 \cdot 2^{-(j-k)+k\delta}), \]
and $\mu\left(10 T\right) \lesssim 2^{k(50\delta-\alpha^*/2)- \alpha(j-k)}$ by the definition of good tubes (see \eqref{badtubestau} and \eqref{goodtubestau}). Since $j \geq 2J$, the measure $\mu_{m,k}$ satisfies
\begin{align*} \mu_{m,k}(4S) &\leq \int_{\mathbb{R}^3} \int_{4S} \zeta_{k,N}(x-y) \, dx \, d(\rho_{\#}\mu)(y) \\
&\lesssim_{N, \delta} \int_{\mathbb{R}^3} \int_{4S \cap B\left(y, 2^{-k(1-\delta)} \right) } \zeta_{k,N}(x-y) \, dx \, d(\rho_{\#}\mu)(y) +2^{-kN} \\
&\leq \int_{5S} \int_{\mathbb{R}^3} \zeta_{k,N}(x-y) \, dx \, d(\rho_{\#}\mu)(y) + 2^{-kN} \\
&\lesssim \int_{5S} \, d(\rho_{\#}\mu)(y) + 2^{-kN} \\
&= \int_{10 T}  \, d\mu(y) + 2^{-kN} \\
&\lesssim 2^{k(50\delta-\alpha^*/2)- \alpha(j-k)}. \end{align*} 
This can be interpreted as saying that ``good'' tubes for $\mu$ are automatically ``good'' tubes for its mollified and localised version $\mu_{m,k}$. Similarly, by the uncertainty principle, for any $x \in \mathbb{R}^3$ and $r >0$,
\begin{equation} \label{dimensionboundmu} \mu_{m,k}(B(x,r)) \lesssim 2^{-\alpha(j-k)} r^{\alpha}, \quad \lVert \mu_{m,k}\rVert_{\infty} \lesssim 2^{3k-j\alpha}, \end{equation}
see e.g.~\cite[Lemma~3.1]{erdogan} for a proof. The second term of the product in \eqref{holder} therefore satisfies
\begin{align} \notag \int_Y \mu_{m,k}^{p'/2} &\lesssim 2^{(3k-j\alpha)(p'/2 -1)} \int_Y \, d\mu_{m,k} \\
\notag &\lesssim \frac{2^{(3k-j\alpha)(p'/2 -1)}}{M}\sum_{S \in \mathbb{W}'}  \int_{4S} \, d\mu_{m,k} \\
\label{secondterm} &\lesssim 2^{k[3(p'/2-1) +50\delta-\alpha^*/2+\alpha]-j\alpha p'/2} \left( \frac{\lvert \mathbb{W}'\rvert}{M} \right). \end{align}
This bounds the second term in \eqref{holder}.

The first term in \eqref{holder} will be bounded using the refined Strichartz inequality; see Theorem~\ref{refinedstrichartz} for the precise statement and subsequent proof. The inequality in Theorem~\ref{refinedstrichartz} is for $R^{1/2}$-cubes in $B(0,R)$ with $R \geq 1$, so applying Theorem~\ref{refinedstrichartz} with $R=2^k$ and $p=6$ to the rescaled functions $g_S(x) = f_S\left(2^{-k}x\right)$ gives 
\begin{equation}  \label{firstterm} \left\lVert \sum_{S \in \mathbb{W}'} f_S \right\rVert_{L^p(Y)} \lesssim_{\epsilon, \delta} 2^{\frac{k}{p'} \left[ 3/2 + \epsilon/2 \right]}  \left( \frac{M}{\lvert \mathbb{W}'\rvert} \right)^{\frac{1}{2} - \frac{1}{p}}\left( \sum_{S \in \mathbb{W}'} \left\lVert f_S \right\rVert_2^2 \right)^{1/2}. \end{equation}

 Applying~\eqref{secondterm} and \eqref{firstterm} to (the square root of) each term in \eqref{holder} gives 
\begin{multline} \label{pause80} \left\lVert \sum_{S \in \mathbb{W}'} f_S \right\rVert_{L^p(Y)} \left( \int \mu_{m,k}^{p'/2} \right)^{1/p'} \lVert \mu_m\rVert^{1/2} \\
\lesssim_{\epsilon}  2^{\frac{k}{p'}[3(p'/2-1) -\alpha^*/2+\alpha + 3/2+\epsilon]-j\alpha/2}\lVert \mu_m\rVert^{1/2}\left( \sum_{S \in \mathbb{W}'} \left\lVert f_S \right\rVert_2^2 \right)^{1/2},\end{multline}
since $\delta \ll \epsilon$ depends only on $\epsilon$. Applying Plancherel twice to the functions $f_S$ (see \eqref{fdefn}) yields
\begin{equation} \label{pigeonbound} \left\lVert f_S\right\rVert_2 \lesssim \left\lVert M_{T, \rho^*} \mu \right\rVert_2. \end{equation}
Putting this chain of inequalities together will conclude the proof of Claim~\ref{L2bound3}. To summarise, %\begingroup
%\allowdisplaybreaks
\begin{align} \notag &\int \phi(\xi) \left\lvert \mathcal{F} \left(\vartheta_m \mu_{g,j,k,\rho^*}\right)\left(\xi\right) \right\rvert^2 \, d\xi \\
\notag &\quad \lesssim \int \left\lvert  \sum_{2\rho T \in \mathbb{W}} \mathcal{F} \left( \vartheta_m M_{T, \rho^*} \mu \right) (\xi) \right\rvert^2 \, d\xi  &&\text{(by \eqref{pause10})} \\
\label{L2sum} &\quad \lesssim  \sum_{2\rho T \in \mathbb{W}}\int \left\lvert   M_{T, \rho^*} \mu(x) \right\rvert^2 \, dx + 2^{-kN} &&\text{(by \eqref{plancherelstar})} \\
\notag &\quad =  \int \sum_{2\rho T\in \mathbb{W}} \left[(M_{T, \rho^*}\mu) \eta_{T, 1/\rho} \right] \ast \widecheck{\psi_{\tau,\rho}} \, d(\rho_{\#}\mu) + 2^{-kN} &&\text{(by \eqref{convolution})} \\
\notag &\quad \leq \left(\int \left\lvert\sum_{S \in \mathbb{W}} f_S \right\rvert^2 \, d \mu_m \right)^{1/2} \lVert \mu_m\rVert^{1/2} + 2^{-kN} &&\text{(by \eqref{couchy})} \\
\notag &\quad \lesssim_N \log\left(2^j\right)^{3/2} \left(\int_Y \left\lvert\sum_{S \in \mathbb{W}'} f_S \right\rvert^2 \, d\mu_{m,k} \right)^{1/2} \lVert \mu_m\rVert^{1/2} + 2^{-kN} &&\text{(by \eqref{pause15})} \\
\notag &\quad \leq \log\left(2^j\right)^{3/2}  \left\lVert \sum_{S \in \mathbb{W}'} f_S \right\rVert_{L^p(Y)} \left( \int_Y \mu_{m,k}^{p'/2} \right)^{1/p'} \lVert \mu_m\rVert^{1/2} +2^{-kN} &&\text{(by \eqref{holder})} \\
\notag &\quad \lesssim_{\epsilon,N}  2^{\frac{k}{p'}[3(p'/2-1) -\alpha^*/2+\alpha + 3/2+2\epsilon]-j\alpha/2}   \times \\
\label{pigeonconstraint} &\qquad \left\lVert \mu_m \right\rVert^{1/2}\left( \sum_{2\rho T \in \mathbb{W}} \lVert M_{T, \rho^*}\mu \rVert_2^2 \right)^{1/2}  + 2^{-kN} &&\text{(by \eqref{pause80})}. \end{align}% \endgroup
In the last line the $\log\left(2^j\right)$ factors were absorbed into the $2^{k \epsilon}$ term; see \eqref{trivialinequality} and \eqref{Jchoice}. The sum in the last line is the same as the one in \eqref{L2sum}, so division gives 
\begin{equation*}  \sum_{2\rho T \in \mathbb{W}}\int \left\lvert  M_{T, \rho^*} \mu(x) \right\rvert^2 \, dx \lesssim_{\epsilon, N}  2^{\frac{2k}{p'}[3(p'/2-1) -\alpha^*/2+\alpha + 3/2+ 5\epsilon]-j\alpha} \lVert \mu_m\rVert +2^{-kN}. \end{equation*} 
 %Using the formula $\alpha^* = \frac{\alpha}{3} + 1 - \epsilon$ from \eqref{alphastardefn}, and 
Putting this bound into \eqref{L2sum} gives
\begin{align*} \int \phi(\xi) \left\lvert \mathcal{F} \left(\vartheta_m \mu_{g,j,k,\rho^*}\right)\left(\xi\right) \right\rvert^2 \, d\xi &\lesssim_{\epsilon,N} 2^{\frac{2k}{p'}[3(p'/2-1)-\alpha^*/2+\alpha + 3/2+5\epsilon]-j\alpha} \lVert \mu_m\rVert + 2^{-kN} \\
&\leq 2^{k \left( 3-\frac{3}{p'} + \frac{2\alpha}{p'} - \frac{\alpha^*}{p'} + 5\epsilon \right)-j \alpha} \lVert \mu_m\rVert + 2^{-kN}. \end{align*}
This finishes the proof of Claim~\ref{L2bound3}.

By summing the bound from Claim~\ref{L2bound3} over $m$ and substituting into \eqref{diag3}, using $k \geq \frac{j \epsilon}{10^5}-2$, and using \eqref{offdiag4} for the off-diagonal terms (since $k \leq j(1-\epsilon)$, the choice $N \gg \epsilon^{-4}$ works), 
\begin{equation} \label{pause90} \int_{\mathcal{N}_{C2^{-k}}(\Gamma) \setminus \mathcal{N}_{C^{-1}2^{-k}}(\Gamma_{\mathbb{R}})}  \left\lvert \widehat{\mu_g}\left( 2^j \xi \right) \right\rvert^2 \,   d\xi \lesssim_{\epsilon}  2^{\frac{k}{p'}\left( -3 + 2\alpha - \alpha^*+ 10^7 \epsilon\right)-j \alpha }.  \end{equation}
Putting this into \eqref{neighbourhood} gives
\begin{align} \notag &\sum_{j \geq (1-\kappa)^{-1}} \sum_{k=\left\lfloor 2j(1-\kappa) \right\rfloor-2}^{\left\lfloor j (1-\epsilon) \right\rfloor} \int_{\mathcal{N}_{C2^{-k}}(\Gamma) \setminus \mathcal{N}_{C^{-1} 2^{-k}} (\Gamma_{\mathbb{R}})} 2^{js' + k/2} \left\lvert \widehat{\mu_g}\left( 2^j \xi \right) \right\rvert^2 \,   d\xi \\
\notag &\quad \lesssim_{\epsilon} 2^{CJ} + \sum_{j \geq \max\left\{2J, (1-\kappa)^{-1} \right\}} \sum_{k=\left\lfloor 2j(1-\kappa) \right\rfloor-2}^{\left\lfloor j (1-\epsilon) \right\rfloor} 2^{j(s'-\alpha)} \cdot 2^{k\left[\frac{1}{p'}\left( -3 + 2\alpha-\alpha^*\right) + \frac{1}{2} +10^7\epsilon\right] } \\
\notag &\quad\lesssim  2^{CJ} + \sum_{j \geq \max\left\{2J, (1-\kappa)^{-1} \right\}} 2^{j\left[s'- \alpha+ \frac{1}{p'} \left( -3 + 2\alpha-\alpha^* \right) + \frac{1}{2} +10^7\epsilon\right] } \\
\label{finalbound} &\quad \lesssim_{\epsilon} 2^{CJ}, \end{align}
for a sufficiently large constant $C$, since $p'=3$, $3/2 < \alpha < 5/2$, $\epsilon \ll 1$ and by the definition of $s'$ in \eqref{sdoubleprimedefn}. The trivial bound $\left\lvert \widehat{\mu_g}(\xi) \right\rvert \lesssim R^{O(1)}$ for $\xi  \in B(0,R)$, coming from \eqref{asterisktrivial}, was used to handle the terms with $j \leq 2J$. This bounds the sum in \eqref{neighbourhood}.

It remains to bound the sum in \eqref{surface}, which may be written as
\begin{equation} \label{pause101}  \sum_{j \geq (1-\kappa)^{-1}} 2^{j(s'-1)} \int_{1-2^{-j(1-\epsilon)}}^1 \frac{1}{\sqrt{1-t^2}} \left(\int_{2^{j-1}}^{2^j} \int_0^{2\pi}  \left\lvert \widehat{\mu_g}\left( r  \gamma_t(\theta) \right) \right\rvert^2 \, d\theta \, dr\right) \, dt. \end{equation}
By applying a similar argument to \eqref{placeholder}-\eqref{errorterms}, for each $j \geq 2J$ in \eqref{pause101} the function $\widehat{\mu_g}$ can be replaced by 
\[ \widehat{\mu_{g,j}} := \sum_{\substack{ j' \\ 2^{j'} \sim 2^j } } \sum_{k=\lfloor j(1-2\epsilon) \rfloor}^j \sum_{\tau \in \Lambda_{j',k}} \sum_{T \in \mathbb{T}_{j',k, \tau,g}} \widehat{M_T\mu}. \]
The function $\mu_{g,j}$ is supported in the ball of radius $2^{10j\delta}$ centred at the origin, and therefore satisfies $\mu_{g,j} = \mu_{g,j} \varphi\left(x/ 2^{10j \delta}\right)$ for a smooth non-negative bump function $\varphi$ equal to 1 on $B(0,1)$, which vanishes outside $B(0,2)$. Hence for $r, \theta$ and $t$ in the domain of integration in \eqref{pause101},  the Schwartz decay of $\widehat{\varphi}$ gives
\[ \left\lvert \widehat{\mu_{g,j}}\left( r  \gamma_t(\theta) \right) \right\rvert \lesssim_N  \left(\left\lvert \widehat{\mu_{g,j}}\right\rvert \ast \zeta_N\right)\left( r  \gamma_t(\theta) \right), \]
where $\zeta_N(x) = \frac{2^{30j\delta}}{1+2^{10j\delta N}\lvert x\rvert^N}$. The function $\zeta_N$ has the property that 
\begin{equation} \label{constantproperty} \zeta_N(x+w) \lesssim_N 2^{10j\delta N} A^N \zeta_N(x) \quad w, x \in \mathbb{R}^3, \quad \lvert w \rvert \leq A, \end{equation}
for any constant $A \geq 1$, and therefore the function $\left\lvert \widehat{\mu_{g,j}}\right\rvert \ast \zeta_N$ also has this property. By the definition of $\gamma_t$ in \eqref{gammatdefn},
\begin{align*} \gamma_t(\theta) &= \sqrt{1-t^2} \gamma'(\theta) + t \gamma(\theta) \times \gamma'(\theta) \\
&= \frac{t}{\sqrt{2}}\gamma(\theta+\pi) - \sqrt{1-t^2} \gamma'(\theta+\pi) \\
&= \frac{t}{\sqrt{2}}\gamma\left(\theta+\pi - \frac{\sqrt{2}\sqrt{1-t^2}}{t} \right) + O\left(1-t\right), \end{align*} 
where $t \sim 1$. Since $0 \leq 1-t \leq 2^{-j(1-\epsilon)}$ and $r \sim 2^j$, the constancy property \eqref{constantproperty} of $\left\lvert \widehat{\mu_{g,j}}\right\rvert \ast \zeta_N$ therefore gives
\[ \left\lvert \widehat{\mu_{g,j}}\left( r  \gamma_t(\theta) \right) \right\rvert \lesssim_N 2^{jN(\epsilon +10\delta)}  \left(\left\lvert \widehat{\mu_{g,j}} \right\rvert \ast \zeta_N\right)\left( \frac{r}{\sqrt{2}}  \gamma\left(\theta+\pi - \frac{\sqrt{2}\sqrt{1-t^2}}{t} \right) \right), \] 
on the domain of integration in \eqref{pause101}. Applying Cauchy-Schwarz and using $\delta \ll \epsilon$ gives
\[ \left\lvert \widehat{\mu_{g,j}}\left( r  \gamma_t(\theta) \right) \right\rvert^2 \lesssim_N  2^{4j \epsilon N} \left(\left\lvert \widehat{\mu_{g,j}} \right\rvert^2 \ast \zeta_N\right)\left( \frac{r}{\sqrt{2}}  \gamma\left(\theta+\pi - \frac{\sqrt{2}\sqrt{1-t^2}}{t} \right) \right). \] 
Hence
\begin{align} \notag &\int_{1-2^{-j(1-\epsilon)}}^1 \frac{1}{\sqrt{1-t^2}} \left(\int_{2^{j-1}}^{2^j} \int_0^{2\pi}  \left\lvert \widehat{\mu_{g,j}}\left( r  \gamma_t(\theta) \right) \right\rvert^2 \, d\theta \, dr\right) \, dt \\
\label{uncertainty10} &\quad \lesssim_N 2^{j(-1/2+5\epsilon N)} \int_{2^{j-1}}^{2^j} \int_0^{2\pi}  \left(\left\lvert \widehat{\mu_{g,j}} \right\rvert^2 \ast \zeta_N \right)\left( \frac{r}{\sqrt{2}}  \gamma(\theta) \right) \, d\theta \, dr. \end{align}
By the essentially constant property of $\zeta_N$, this double integral satisfies 
\begin{align} \notag  &\int_{2^{j-1}}^{2^j} \int_0^{2\pi}  \left(\left\lvert \widehat{\mu_{g,j}} \right\rvert^2 \ast \zeta_N \right)\left( \frac{r}{\sqrt{2}}  \gamma(\theta) \right) \, d\theta \, dr \\
\notag &\quad \lesssim_N  2^{j(10\delta N - 1)}\int_{\mathcal{N}_1(2^j \Gamma) }  \left(\left\lvert \widehat{\mu_{g,j}} \right\rvert^2 \ast \zeta_N \right)\left( \xi \right) \, d\xi \\
\notag &\quad \lesssim_N 2^{j(10\delta N - 1)}\int_{B(0, 2^{j \epsilon } ) } \left[ \int_{\mathcal{N}_1(2^j\Gamma) } \left\lvert \widehat{\mu_{g,j}}\left(\xi - y\right) \right\rvert^2 \, d\xi \right] \, dy  + 2^{-j\epsilon N} \\
\label{uncertainty20} &\quad \leq 2^{j(10\delta N - 1)} \int_{\mathcal{N}_{2 \cdot 2^{j \epsilon}}(2^j\Gamma) } \left\lvert \widehat{\mu_{g,j}}\left(\xi\right) \right\rvert^2 \, d\xi  + 2^{-j\epsilon N}.  \end{align} 
This integral is essentially a special case of the integral in \eqref{placeholder} with $k \approx j$; if $j \geq 2J$ then similar working to the proof of Claim~\ref{L2bound3}, used to obtain \eqref{pause90}, gives 
\begin{equation} \label{theendisnear} \int_{\mathcal{N}_{2 \cdot 2^{j \epsilon}}(2^j\Gamma) } \left\lvert \widehat{\mu_{g,j}}\left(\xi\right) \right\rvert^2 \, d\xi \lesssim_{\epsilon}  2^{j\left[3-\alpha + \frac{1}{p'}\left( -3+2\alpha-\alpha^*\right)+10^7\epsilon\right] }. \end{equation} Putting all of this together, the sum in \eqref{surface} satisfies 
\begin{align} \notag  &\sum_{j \geq (1-\kappa)^{-1}} 2^{j(s'-1) } \int_{1-2^{-j(1-\epsilon)}}^1 \frac{1}{\sqrt{1-t^2}} \times \\
\notag &\qquad \qquad \left(\int_{2^{j-1}}^{2^j} \int_0^{2\pi}  \left\lvert \widehat{\mu_g}\left( r  \gamma_t(\theta) \right) \right\rvert^2 \, d\theta \, dr\right) \, dt \\
\notag &\quad\lesssim_N 2^{CJ} + \sum_{j \geq \max\left\{2J, (1-\kappa)^{-1} \right\}}  \\
\notag &\qquad 2^{j(s'-3/2+5N \epsilon) } \int_{2^{j-1}}^{2^j} \int_0^{2\pi}  \left(\left\lvert \widehat{\mu_{g,j}} \right\rvert^2 \ast \zeta_N \right)\left( r  \gamma(\theta) \right) \, d\theta \, dr &&\text{(by \eqref{uncertainty10})} \\
\notag &\quad\lesssim_N 2^{CJ} + \sum_{j \geq \max\left\{2J, (1-\kappa)^{-1} \right\}} \\
\notag &\qquad 2^{j(s'-5/2+6N\epsilon) } \int_{\mathcal{N}_{2 \cdot 2^{j \epsilon}}(\Gamma_{2^j}) } \left\lvert \widehat{\mu_{g,j}}\left(\xi\right) \right\rvert^2 \, d\xi + 2^{-j\epsilon N}   &&\text{(by \eqref{uncertainty20})} \\
\notag &\quad\lesssim_{\epsilon} 2^{CJ} +\sum_{j \geq \max\left\{2J, (1-\kappa)^{-1} \right\}} 2^{j\left[s'+ \frac{1}{2}-\alpha + \frac{1}{p'}\left( -3+2\alpha-\alpha^*\right)+7N\epsilon\right] }   &&\text{(by \eqref{theendisnear})} \\
\label{finalbound2} &\quad\lesssim 2^{CJ},  \end{align} 
for some large constant $C$, since $p'=3$, by the definition of $s'$ in \eqref{sdoubleprimedefn}, and by taking $N \sim \epsilon^{-1/4}$. Putting \eqref{finalbound} and \eqref{finalbound2} into \eqref{neighbourhood} and \eqref{surface} shows that the integral in \eqref{singular} satisfies
\[ \int_{0 \leq \eta_1  \leq \eta_2^{\kappa}} \int_0^{2\pi} \lvert \eta\rvert^{s'-2}\left\lvert \widehat{\mu_g}\left( \eta_1\gamma'(\theta) + \eta_2 \gamma(\theta) \times \gamma'(\theta) \right) \right\rvert^2 \, d\theta \,  d\eta \lesssim_{\epsilon} 2^{CJ}, \]
which finishes the proof of Claim~\ref{L2bound2}.

Putting the bounds from Claim~\ref{L2bound1} and Claim~\ref{L2bound2} into \eqref{L2quant}, and then into \eqref{HoneE}, gives 
\[ \mathcal{H}^1(E) \lesssim_{\epsilon} j_0^4 2^{-j_0\epsilon} 2^{CJ}. \]
At this point, choose 
\begin{equation} \label{Jchoice} J = \left\lfloor (j_0\epsilon)/(2C) \right\rfloor, \end{equation}
so that 
\[ \mathcal{H}^1(E) \lesssim_{\epsilon} \epsilon', \]
provided $\delta_1$ is sufficiently small (depending on $\epsilon'$), since $j_0 > \left\lvert \log_2 \delta_1 \right\rvert$. This bounds $\mathcal{H}^1(E)$ in the case where the second term of \eqref{bad} dominates. Since $\mathcal{H}^1(E) \lesssim \epsilon'$ in either case and $\epsilon'$ is arbitrary (with the implicit constant independent of $\epsilon'$), this implies that $\mathcal{H}^1(E)=0$. By inner regularity of the Lebesgue measure, this shows that  
\[ \dim \pi_{\theta} (A) \geq  \dim \pi_{\theta} (\supp \nu) \geq \max\left\{\frac{4\alpha}{9} + \frac{5}{6}, \frac{2\alpha+1}{3} \right\} - 200\sqrt{\epsilon}, \]
for a.e.~$\theta \in [0,2\pi)$. Since a countable union of measure zero sets has measure zero, sending $\epsilon \to 0$ along a countable sequence finishes the proof.  \end{proof}

\subsection{Nonstationary phase} \label{nonstationary}
The following lemma states that if the angle of a tube $T$ is not equal to the angle of projection $\theta$, then $\pi_{\theta\#} (M_Tf)$ is negligible. Recall the notation $(\angle \tau)^* = (\pi + \angle\tau) \bmod 2\pi$ if $\tau$ lies in the forward light cone, and $(\angle \tau)^* = \angle\tau$ if $\tau$ lies in the backward light cone.
\begin{lemma} \label{nonstationaryphase} Fix $\theta \in [0,2\pi)$ and a cap $\tau\in \Lambda_{j,k}$. If
\begin{equation} \label{angleassumption} \left\lvert (\angle \tau)^* - \theta \right\rvert \geq 10^3 \cdot 2^{k(-1/2 + \delta)}, \end{equation}
then for any positive smooth function $f$ supported in the unit ball of $\mathbb{R}^3$ and any $T \in \mathbb{T}_{j,k, \tau}$,
\[ \lVert \pi_{\theta\#} (M_Tf) \rVert_{L^1(\mathbb{R}^3, \mathcal{H}^2)} \lesssim_N 2^{-kN}\lVert f\rVert_1, \]
for arbitrarily large $N$.
\end{lemma}
\begin{proof} Assume that $\tau$ lies in the forward light cone; the proof for the backward light cone is similar. For any smooth function $g$ supported in the unit ball and any $x \in \gamma(\theta)^{\perp}$,
\begin{equation} \label{projformula} (\pi_{\theta\#} g)(x) = \int_{\mathbb{R}} g\left( x + t \gamma(\theta) \right) \, dt, \end{equation}
where the density $(\pi_{\theta\#} g)(x)$ is characterised by 
\[ \int_F (\pi_{\theta\#} g)(x) \, d\mathcal{H}^2(x) = (\pi_{\theta\#} g)(F), \] 
for all Borel sets $F \subseteq \gamma(\theta)^{\perp}$. Applying \eqref{projformula} and then Fubini to the function $g=M_Tf$ shows that for any $x \in \gamma(\theta)^{\perp}$,
\begin{equation} \label{parts3} (\pi_{\theta\#} M_T f)(x) = \int_{\mathbb{R}^3} f(y) \left[ \int_{\mathbb{R}} \eta_T(x+t\gamma(\theta)) \widecheck{\psi_{\tau}}(x+t\gamma(\theta) -y ) \, dt \right]\, dy. \end{equation}
The innermost integral is 
\begin{multline} \label{parts2} \int_{\mathbb{R}} \eta_T(x+t\gamma(\theta)) \widecheck{\psi_{\tau}}(x+t\gamma(\theta) -y ) \, dt  \\
= \int_{\mathbb{R}^3} \psi_{\tau}(\xi) e^{2\pi i \langle \xi, x-y \rangle} \left[ \int_{\mathbb{R}} \eta_T(x+t\gamma(\theta) ) e^{2\pi i t \langle \xi, \gamma(\theta) \rangle } \, dt\right] \, d\xi. \end{multline} 
By integrating by parts $m$ times, the innermost integral of this is 
\begin{multline} \label{parts} \int_{\mathbb{R}} \eta_T(x+t\gamma(\theta) ) e^{2\pi i t\langle \xi, \gamma(\theta) \rangle } \, dt \\
= \left( \frac{-1}{2\pi i \langle \xi, \gamma(\theta) \rangle } \right)^m \int_{\mathbb{R}} e^{2\pi i t \langle \xi, \gamma(\theta) \rangle } \left( \frac{d}{dt} \right)^m \eta_T(x+t\gamma(\theta)) \, dt. \end{multline} 
Therefore it will suffice to show that the right-hand side of $\eqref{parts}$ is $\lesssim_m 2^{-2k\delta m}$ in size for any $m$, whenever the variable $\xi$ occurring in \eqref{parts} lies in $1.1\tau$. Assume without loss of generality that $\angle \tau=0$. By translating $x$ and changing variables in $t$ it may be further assumed that $T$ is centred at the origin. Since $\angle \tau = 0$ this means that
\[ T = UA \left([0,1]^3 \right), \]
where $A: \mathbb{R}^3 \to \mathbb{R}^3$ is the linear map
\[ (x_1,x_2,x_3) \mapsto \left(2^{-j+k(1/2+\delta)}x_1, 2^{-j+k(1/2+\delta)} x_2 , 2^{-(j-k)+k \delta} x_3\right), \]
and $U$ is the unitary defined through the standard basis by
\[ e_1 \to \frac{1}{\sqrt{2}} (1,0,1), \quad e_2 \to e_2, \quad e_3 \to \frac{1}{\sqrt{2}} (-1,0,1). \]
By \eqref{angleassumption},
\begin{equation} \label{anglerestriction} \varepsilon := \lvert \theta-\pi\rvert \geq 10^3 2^{k(-1/2+ \delta)}. \end{equation}
Write 
\[ \xi = \frac{\xi_1}{\sqrt{2}} (1,0,1) + \xi_2 (0,1,0) + \frac{\xi_3}{\sqrt{2}} (-1,0,1), \quad \xi_1, \xi_2, \xi_3 \in \mathbb{R}. \]
Since $\xi \in 1.1\tau$, the coefficients $\xi_1$, $\xi_2$, $\xi_3$ satisfy
\begin{equation} \label{coordinates} \frac{2^j}{10} \leq \xi_1 \leq 10 \cdot 2^j, \quad \lvert \xi_2\rvert \leq 10\cdot 2^{j-k/2}, \quad \frac{2^{j-k}}{10} \leq \lvert \xi_3\rvert \leq 10 \cdot 2^{j-k}. \end{equation}
The inner product of $\xi$ and $\gamma(\theta)$ is 
\begin{equation} \label{innerproduct} \langle \xi, \gamma(\theta) \rangle =  \frac{ \xi_1}{2} \left( 1+\cos \theta\right) + \frac{\xi_2 \sin \theta}{\sqrt{2}} + \frac{\xi_3}{2} \left( 1-\cos \theta \right). \end{equation}
Hence by \eqref{anglerestriction}, \eqref{coordinates}, \eqref{innerproduct} and the triangle inequality, 
\begin{equation} \label{bound1} \left\lvert \left\langle \xi, \gamma(\theta) \right\rangle \right\rvert \gtrsim 2^j \varepsilon^2 \quad (\gtrsim 2^{j-k(1-2\delta)}). \end{equation}

It remains to bound the integrand of \eqref{parts}. The function $\eta_T$ can be written as 
\begin{multline*} \eta_T(x) = \eta(A^{-1} U^*x) \\
= \eta\left( 2^{j - k\left( 1/2 + \delta \right) } \left( \frac{x_1+x_3}{\sqrt{2}} \right), 2^{j - k\left( 1/2 + \delta \right) } x_2, 2^{j-k-k\delta}\left( \frac{x_3-x_1}{\sqrt{2}} \right) \right), \end{multline*}
where $\eta$ is a smooth function vanishing outside $B(0,2)$ and satisfying $\eta \sim 1$ on $[0,1]^3$. Hence 
\begin{equation} \label{directional} \eta_T(x+ t\gamma(\theta)) = \eta(A^{-1} U^* x + tA^{-1} U^* \gamma(\theta)), \end{equation}
and it will suffice to bound $\left\lvert A^{-1} U^* \gamma(\theta)\right\rvert$ from above. By the definition of $U$,
\[ U^* \gamma(\theta)  = \left( \frac{1}{\sqrt{2}}  \langle (1,0,1) ,\gamma(\theta) \rangle, \langle(0,1,0), \gamma(\theta) \rangle, \frac{1}{\sqrt{2}} \langle \gamma(\theta), (-1,0,1) \rangle \right), \]
and so 
\begin{multline} \label{spoon} A^{-1} U^* \gamma(\theta) = \Bigg( \frac{2^{j - k\left( 1/2 + \delta \right) }}{\sqrt{2}}  \langle (1,0,1) , \gamma(\theta) \rangle, \\
2^{j - k\left( 1/2 + \delta \right) } \langle(0,1,0), \gamma(\theta) \rangle, \frac{2^{j-k-k\delta}}{\sqrt{2}} \langle \gamma(\theta), (-1,0,1) \rangle \Bigg). \end{multline}
It will be shown that the definition of $\varepsilon$ and the assumption in \eqref{anglerestriction} imply that
\begin{equation} \label{bound2} \lvert A^{-1} U^*\gamma(\theta) \rvert \lesssim \varepsilon 2^{j-k(1/2+\delta)}. \end{equation}
To see this, expanding out the Euclidean norm of \eqref{spoon} gives
\begin{equation} \label{threeterms} \lvert A^{-1} U^*\gamma(\theta) \rvert  \sim \left\lvert 1+ \cos \theta \right\rvert 2^{j-k(1/2+ \delta) } + \left\lvert \sin \theta \right\rvert 2^{j-k(1/2+\delta) } + 2^{j-k-k\delta} \left\lvert 1- \cos \theta \right\rvert. \end{equation}
The last term satisfies 
\[ 2^{j-k-k\delta} \left\lvert 1- \cos \theta \right\rvert \lesssim 2^{j-k-k\delta} \lesssim \varepsilon 2^{j-k(1/2+\delta)}, \]
where the last inequality follows from the assumed lower bound \eqref{anglerestriction} on $\varepsilon$. The second term in \eqref{threeterms} satisfies 
\[ \left\lvert \sin \theta \right\rvert 2^{j-k(1/2+\delta) } \lesssim \left\lvert \theta - \pi \right\vert  2^{j-k(1/2+\delta) }  = \varepsilon  2^{j-k(1/2+\delta) }, \]
again by the definition of $\varepsilon$ in \eqref{anglerestriction}. Similarly, the first term in \eqref{threeterms} satisfies 
\begin{align*} \left\lvert 1+ \cos \theta \right\rvert 2^{j-k(1/2+ \delta) } &= \left\lvert \cos \theta - \cos \pi \right\rvert 2^{j-k(1/2+ \delta) } \\
&\lesssim \left\lvert \theta - \pi \right\rvert  2^{j-k(1/2+ \delta) } \\
&= \varepsilon  2^{j-k(1/2+\delta) }, \end{align*} 
which proves \eqref{bound2}.
%Differentiating \eqref{directional} $m$ times with the chain rule gives 
%\begin{align*}   \left(\frac{d}{dt} \right)^m \eta_T(x+t\gamma(\theta)) &= \left( \left\langle A^{-1} U^* \gamma(\theta), \nabla \right\rangle^m \eta \right)\left( x+tA^{-1} U^* \gamma(\theta) \right). \end{align*}
Differentiating \eqref{directional} $m$ times and using \eqref{bound2} gives
\begin{equation} \label{bound3} \left\lvert\left( \frac{d}{dt} \right)^m \eta_T(x+t\gamma(\theta))\right\rvert %\lesssim_m \lvert A^{-1} U^*\gamma(\theta) \rvert^m \left\lVert \eta\right\rVert_{W^{m,\infty} } 
\lesssim_m \left(\varepsilon 2^{j-k(1/2+\delta)} \right)^m. \end{equation}
Putting \eqref{bound1} and \eqref{bound3} into \eqref{parts} and then using \eqref{anglerestriction} gives
\[ \left\lvert\int_{\mathbb{R}} \eta_T(x+t\gamma(\theta) ) e^{2\pi i t\langle \xi, \gamma(\theta) \rangle } \, dt\right\rvert \lesssim_m 2^{-2k\delta m}. \]
Putting this into \eqref{parts}, then \eqref{parts2} and then \eqref{parts3}, and taking $m$ large enough, gives
\[ \lVert \pi_{\theta\#} M_Tf\rVert_{L^{\infty}\left(\mathbb{R}^3, \mathcal{H}^2 \right)} \lesssim_N \lVert f\rVert_1 2^{-kN}. \]
Since $M_Tf$ is supported in a ball of radius $\sim 1$, this gives
\[ \lVert \pi_{\theta\#} M_Tf\rVert_{L^1(\mathbb{R}^3, \mathcal{H}^2)} \lesssim_N \lVert f\rVert_1 2^{-kN}. \qedhere \] \end{proof}

\subsection{Refined Strichartz inequality} \label{strichartz}
This main inequality of this subsection will make use of the decoupling theorem for the cone, from \cite{bourgain} (stated here only in $\mathbb{R}^3$ with $p \leq 6$). 
\begin{theorem}[{\cite[Theorem~1.2]{bourgain}}]  \label{decoupling} Let $K \geq 1$. Let $\{\tau\}$ be a finitely overlapping cover of the truncated cone
\[ \Gamma = \left\{ \left(\xi, \left\lvert \xi \right\rvert\right) \in \mathbb{R}^3 : 1 \leq \left\lvert \xi \right\rvert \leq 2 \right\} \]
by boxes $\tau$ of dimensions $1 \times K^{-1} \times K^{-2}$, with each $\tau$ contained in the $\sim K^{-2}$-neighbourhood of $\Gamma$. Let $F = \sum_{\tau} F_{\tau}$ be such that each $F_{\tau}$ has Fourier transform supported in $\tau$. Then for $2 \leq p \leq 6$ and any $\epsilon >0$,
\[ \left\lVert F \right\rVert_p \leq C_{\epsilon} K^{\epsilon} \left( \sum_{\tau} \left\lVert F_{\tau} \right\rVert_p^2 \right)^{1/2}. \]
\end{theorem} Fix $R \geq 1$ and $\delta>0$. The proof of the refined Strichartz inequality for the cone in $\mathbb{R}^3$ given here is similar to the paraboloid case from~\cite{GIOW}, with a few extra steps (similar to those in~\cite{me}) needed to deal with the obstruction that boxes dual to $R^{-1/2}$-caps in the cone do not intersect $R^{1/2}$ cubes in a clean way. Let $Y$ be a disjoint union of cubes $Q$ of side length $R^{1/2}$, all contained in $B_R= B_3(0,R)$. The truncated cone 
\[ \Gamma = \left\{ \left(\xi, \left\lvert \xi \right\rvert \right) \in \mathbb{R}^3 : 1 \leq \left\lvert \xi \right\rvert \leq 2 \right\} \]
 has a finitely overlapping cover by boxes (or caps) $\theta$ of dimensions
\[ 1 \times R^{-1/2} \times R^{-1}. \]
For each $\theta$, $B_R$ has a finitely overlapping cover by tubes $T$ of dimensions
\[ R^{(1/2)(1+\delta)} \times R^{(1/2)(1+\delta)} \times 3R, \]
with long axis normal to the cone at $\theta$. Let $\mathbb{T}_{\theta}$ be the set of tubes corresponding to $\theta$ and let $\mathbb{T} = \bigcup_{\theta} \mathbb{T}_{\theta}$.
 \begin{theorem} \label{refinedstrichartz} For any $p$ with $2 \leq p \leq 6$ and for any $\epsilon >0$, there exists $\delta_0 \ll \epsilon$ such that the following holds whenever $\delta \in (0, \delta_0)$. Suppose that
\[ f = \sum_{T \in \mathbb{W}} f_T, \]
where $\mathbb{W} \subseteq \mathbb{T}$ is nonempty, and $f_T{\restriction_{B_R}}$ is essentially supported in $T$ with $\widehat{f_{T}}$ essentially supported in $\theta(T)$. Suppose further that $\lVert f_T\rVert_2$ is constant over $T \in \mathbb{W}$ up to a factor of 2, and that each $Q \subseteq Y$ intersects at most $M$ tubes $2T$ with $T \in \mathbb{W}$, where $M \geq 1$. Then 
\begin{equation} \label{actualinequality} \lVert f\rVert_{L^p(Y)} \leq C_{\epsilon, \delta} R^{\epsilon} \left(\frac{MR^{-3/2}}{\left\lvert\mathbb{W}\right\rvert} \right)^{\frac{1}{2} - \frac{1}{p}} \left( \sum_{T \in \mathbb{W}} \lVert f_T\rVert_2^2 \right)^{1/2}. \end{equation}
\end{theorem} 
\begin{remark} The definition assumed here for ``$f_T{\restriction_{B_R}}$ is essentially supported in $T$'' will be that $\left\lVert f_T\right\rVert_{L^2(B_R \setminus T ) } \leq C_N R^{-N}\left\lVert f_T \right\rVert_2$ for some sufficiently large but fixed positive integer $N \geq 1$. Similarly ``$\widehat{f_{T}}$ is essentially supported in $\theta(T)$'' will mean that $\left\lVert\widehat{f_T}\right\rVert_{L^2(\mathbb{R}^3 \setminus \theta(T) ) } \leq C_N R^{-N}\left\lVert f_T \right\rVert_2$. The constant $C_{\epsilon,\delta}$ in the inequality may also depend implicitly on the choice of $N$ and corresponding $C_N$, but the dependence will be ignored since it is not important (the proof works for any sufficiently large fixed $N$). Similarly, the power $R^{-N}$ in the proof may vary from line to line, but this will be suppressed in the notation. \end{remark}
\begin{proof}[Proof of Theorem~\ref{refinedstrichartz}] The proof is by induction on the dyadic range of $R$; the inductive step will be handled by Lorentz rescaling and decoupling. By dyadic pigeonholing of the cubes $Q \subseteq Y$, it may be assumed that each cube contributes equally to the left-hand side of \eqref{actualinequality}, up to a factor of 2. Create a finitely overlapping cover of the cone $\Gamma$ with larger boxes $\tau$ of dimensions 
\[ \sim 1 \times R^{-1/4} \times R^{-1/2},  \]
and for each $\tau$, create a finitely overlapping cover of $B_R$ by boxes $\Box$ of dimensions 
\[ R^{(1/2)(1+ \delta)} \times R^{(3/4)(1+ \delta)} \times R, \] 
with directions dual to $\tau(\Box)$. Each $T \in \mathbb{W}$ has a finitely overlapping cover by boxes $T'$ of dimensions
\[ \frac{R^{1/4 + \delta/2}}{100} \times \frac{R^{1/2+ \delta/2}}{100} \times 3R, \]
with long axis normal to the cone at $\theta(T)$, and short axis in the flat direction of the cone at $\theta(T)$. For each $T'$ in the cover of $T$, define $\theta(T') = \theta(T)$ and $f_{T'} = \chi_{T'}f_T$, where the functions $\chi_{T'}$ form a smooth partition of unity such that each $\chi_{T'}$ (restricted to $B(0,R)$) is essentially supported in $T'$ with Fourier transform supported in a box of dimensions 
\[   \frac{R^{-1/4}}{100} \times \frac{R^{-1/2}}{100} \times \frac{R^{-1}}{100},\]
centred at the origin and with axis directions dual to $T'$. Such a partition can be constructed using the Poisson summation formula; the functions $\chi_{T'}$ may be complex-valued but satisfy $\left\lvert \chi_{T'} \right\rvert \lesssim 1$. Each $f_{T'}$ has Fourier transform essentially supported in $1.1\theta(T')$. 

 For each $T'$ there are $\lesssim 1$ caps $\tau$ such that $\theta(T') \cap \tau \neq \emptyset$, so each $\theta(T')$ can be paired to exactly one such $\tau=\tau(\theta(T'))$. Similarly, for this $\tau$, the tube $T'$ can be paired to exactly one set $\Box= \Box(\tau(\theta(T')))$ corresponding to $\tau$ such that $T' \subseteq \Box$, and each $T$ can be similarly paired to exactly 1 set $\Box$ (this is not trivial since $\theta(T)$ may not be in the centre of $\tau$, but requires only some elementary geometry of the cone). For each dyadic value $\sigma$ and each $\Box$, let
%\[ \mathbb{W}_{\Box} = \{ T'\subseteq T : T  \in \mathbb{W}: \theta(T') \subseteq \tau(\Box) , \quad T' \subseteq \Box  \}, \]
\begin{multline} \label{Wdefn} \mathbb{W}_{\Box,\sigma}' = \{ T' : T' \sim T \text{ for some } T \in \mathbb{W}, \\
 \theta(T') \sim \tau(\Box), \quad T' \sim \Box, \quad \left\lVert f_{T'} \right\rVert_2 \in [\sigma,2\sigma) \}, \end{multline}
where the use of $\sim$ in \eqref{Wdefn} refers to the aforementioned pairings. For each dyadic value $\mu$ let $\mathbb{W}_{\Box,\sigma,\mu}'$ be the subset consisting of those $T' \in \mathbb{W}_{\Box,\sigma}'$ such that the number of boxes $T'' \in \mathbb{W}_{\Box,\sigma}'$ in the larger tube $T=T(T')$ which contains $T'$ lies in $[\mu,2\mu)$. By this definition,
\begin{equation} \label{dyadicineq} \left\lVert f_{T'} \right\rVert_2^2 \lesssim c^2/\mu, \quad T' \in  \mathbb{W}_{\Box,\sigma,\mu}', \end{equation}
where $c$ is the dyadically constant value of $\left\lVert f_T\right\rVert_2$ assumed in the theorem statement. Let
\[ f_{\Box,\sigma, \mu} = \sum_{T' \in \mathbb{W}_{\Box,\sigma,\mu}'}  f_{T'}, \]
Then $f = \sum_{\Box,\sigma,\mu} f_{\Box,\sigma,\mu} + O\left( R^{-N} \sup_T \left\lVert f_T \right\rVert_2 \right)$ in $B(0,R)$. Each $\Box$ has a finitely overlapping cover by boxes $Q_{\Box}$ of dimensions 
\[ R^{1/4 + \delta/4} \times R^{1/2 + \delta/4} \times R^{3/4 + \delta/4}, \]
with the same axis orientations as $\Box$. For each $\Box$ let $\{\chi_{Q_{\Box}}\}_{Q_{\Box}}$ be a smooth partition of unity, such that each $\chi_{Q_{\Box}}$ is essentially supported in and rapidly decaying outside of $Q_{\Box}$, and has Fourier transform supported in a box of dimensions 
\[ R^{-1/4} \times R^{-1/2} \times R^{-3/4}, \]
centred at the origin and with long direction in the flat direction of $\tau(\Box)$. %The compromise on the control of the Fourier support in the last coordinate is necessary for the spatial control in the long direction of $Q_{\Box}$. The existence of this partition follows again from Poisson summation (as in \cite{me}), by first partitioning $Q_{\Box}$ into smaller pieces along the long direction and adding up functions from the finer partition with support inside $Q_{\Box}$.

 For each dyadic triple $(M',\sigma, \mu)$, let $Y_{\Box,M',\sigma,\mu}$ be the union of those boxes $Q_{\Box}$ with the property that $Q_{\Box}$ intersects $N \in [M',2M')$ sets $10T'$ with $T' \in \mathbb{W}_{\Box,\sigma,\mu}'$. Define 
\[ \chi_{Y_{\Box,M',\sigma,\mu} }= \sum_{Q_{\Box} \subseteq Y_{\Box,M',\sigma,\mu}} \chi_{Q_{\Box}}. \]
By dyadic pigeonholing, there exists a fixed triple $\left(M',\sigma,\mu\right)$ such that 
\[ \left\lVert f \right\rVert_{L^p(Q)} \lesssim_N (\log R)^3\left\lVert \sum_{\Box} \chi_{Y_{\Box,M',\sigma,\mu} }f_{\Box,\sigma,\mu} \right\rVert_{L^p(Q)} + R^{-N}\sup_{T} \left\lVert f_T\right\rVert_2, \]
for a fraction $\gtrsim 1/(\log R)^3$ of the cubes $Q \subseteq Y$ (the $\sup$ is not strictly necessary since $\lVert f_T\rVert_2$ is essentially constant in $T$). By dyadically pigeonholing the remaining cubes, there exists a collection $\mathbb{B}$ of sets $\Box$ such that $\left\lvert \mathbb{W}_{\Box,\sigma,\mu}' \right\rvert$ is constant up to a factor of 2 as $\Box$ ranges over $\mathbb{B}$, and such that 
\begin{equation} \label{pause50} \left\lVert f \right\rVert_{L^p(Q)} \lesssim_N (\log R)^4\Bigg\lVert  \sum_{\Box \in \mathbb{B}} \chi_{Y_{\Box} }f_{\Box,\sigma,\mu} \Bigg\rVert_{L^p(Q)} + R^{-N}\sup_T \left\lVert f_T\right\rVert_2, \end{equation}
for a fraction $\gtrsim 1/(\log R)^4$  of the cubes $Q \subseteq Y$, where $Y_{\Box} = Y_{\Box,M',\sigma,\mu}$ for the fixed triple $(M',\sigma,\mu)$ and for each $\Box$. By dyadically pigeonholing the remaining cubes again, there is a dyadic number $M''$ and a fraction $\gtrsim 1/(\log R)^5$ of the cubes $Q \subseteq Y$ satisfying \eqref{pause50}, such that each cube $2Q$ intersects $\sim M''$ different sets $Y_{\Box}$, as $\Box$ ranges over $\mathbb{B}$. Let $Y'$ be the union of these remaining cubes $Q \subseteq Y$.

 The $R^{1/2}$-cubes $Q$ are smaller than the sets $\Box$ by at least a factor of $R^{\delta/2}$ in every direction, which means that for every cap $\tau$, each $R^{1/2}$-cube $Q$ intersects $\lesssim 1$ sets $\Box \in \mathbb{B}$ corresponding to $\tau$. Moreover, since the short axis of each box $Q_{\Box}$ points in the long direction of $\tau(\Box)$, the functions $\chi_{Y_{\Box} }f_{\Box,\sigma,\mu}$ each have Fourier transform essentially supported in $C\tau(\Box)$ away from the origin.  Applying the decoupling theorem for the cone (see Theorem~\ref{decoupling}) to \eqref{pause50} therefore gives
\begin{align*} &\lVert f\rVert_{L^p(Q)} \\
&\quad \lesssim_{N,\epsilon,\delta} R^{\epsilon/4+O(\delta)} \left(  \sum_{\substack{\Box \in \mathbb{B} \\ Y_{\Box} \cap 2Q \neq \emptyset} } \left\lVert \chi_{Y_{\Box} }f_{\Box,\sigma,\mu} \right\rVert_{L^p\left(R^{\delta}Q\right)}^2 \right)^{1/2} + R^{-N}\sup_T \left\lVert f_T\right\rVert_2 \\
&\quad \lesssim R^{\epsilon/4+O(\delta)} \left(M''\right)^{\frac{1}{2} - \frac{1}{p} } \left(  \sum_{\Box \in \mathbb{B} } \left\lVert \chi_{Y_{\Box} }f_{\Box,\sigma,\mu} \right\rVert_{L^p\left(R^{\delta}Q\right)}^p \right)^{1/p} + R^{-N}\sup_T \left\lVert f_T\right\rVert_2. \end{align*}
In the first inequality, the restriction $Y_{\Box} \cap 2Q \neq \emptyset$ defining the sum comes from restricting the sum in \eqref{pause50} first, using the constraint that each $\chi_{Q_{\Box}}$ is essentially supported in $Q_{\Box}$, by construction. Taking both sides to the power $p$ and summing over $Q \subseteq Y'$ yields
\begin{multline} \label{lorentz} \lVert f\rVert_{L^p(Y)} \lesssim \left( \log R\right)^5 \lVert f\rVert_{L^p(Y')} \\
 \lesssim_{N, \epsilon,\delta} R^{\epsilon/3} \left( M'' \right)^{\frac{1}{2} - \frac{1}{p} }  \left(\sum_{\Box \in \mathbb{B}  } \left\lVert  f_{\Box,\sigma,\mu} \right\rVert_{L^p\left(Y_{\Box} \right) }^p \right)^{1/p} + R^{-N}\sup_{T} \left\lVert f_T\right\rVert_2. \end{multline}
After applying a Lorentz rescaling $L$ on the Fourier side, which is unitarily equivalent to $(\xi_1, \xi_2, \xi_3) \to \left( \xi_1, R^{1/4} \xi_2, R^{1/2}\xi_3 \right)$, the set $\Box$ becomes a cube $L^{-1}\left(\Box \right)$ of side lengths $\approx R^{1/2}$. The boxes $Q_{\Box}$ become cubes $L^{-1}(Q_{\Box})$ of radius $R^{1/4+\delta/4}$.  The function $L$ can be defined as the unique linear map which rescales $\tau(\Box)$ by $R^{1/4}$ along its intermediate axis and by $R^{1/2}$ along its short axis; this map leaves the flat direction of the cone at $\tau$ fixed and applies a parabolic rescaling along parabolic sections of the cone in the other directions (see e.g.~\cite[pp.~84--85]{demeter} for a description). Under this rescaling, the boxes $T'$ become tubes $L^{-1}(T')$ of dimensions 
\[  \sim R^{1/4+\delta/2} \times R^{1/4+\delta/2}\times R^{1/2}, \]
normal to the rescaled boxes $L(\theta(T'))$\footnote{This is not trivial since $\theta$ may not lie exactly in the centre of $\tau$; one justification is as follows. The map $L^{-1}$ sends the unit normal at $\theta(T')$ to a normal at $L(\theta(T'))$, whilst $L^{-1}$ sends the short axis $\frac{1}{\sqrt{2}}(\cos \theta, \sin \theta, 1)$ of $T'$ to a vector of size $\sim 1$, which lies in the same half of the light cone as $L(\theta(T'))$, and this vector therefore makes an angle $\sim 1$ with the normal to $L(\theta)$. The third axis of $T'$ makes an angle $\lesssim R^{-1/4}$ with the vector $(-\sin \tau, \cos \tau, 0)$, whose image under $L^{-1}$ is the same vector scaled by $R^{-1/4}$. Since the directions of these three image vectors are $\sim 1$ linearly independent and since $L$ has determinant $R^{3/4}$, this determines the direction and radius of $L^{-1}(T')$.}. The rescaled version of each $f_{T'}$ has Fourier transform essentially supported in $1.1L(\theta(T'))$.

 For a given cube $L^{-1}(Q_{\Box})$, the number of tubes $2L^{-1}(T')$ intersecting $L^{-1}(Q_{\Box})$ is $\lesssim M'$ by the pigeonholing step. Assume inductively that the theorem holds with $R$ replaced by any $\widetilde{R} \leq R^{3/4}$. Applying Lorentz rescaling and this inductive assumption at scale $\approx R^{1/2}$ to each summand in the right-hand side of \eqref{lorentz} gives
\[ \left\lVert  f_{\Box,\sigma,\mu} \right\rVert_{L^p\left(Y_{\Box} \right) } \lesssim_{\epsilon,\delta} C_{\epsilon,\delta} R^{\epsilon/2+ O(\delta)} \left( \frac{M'R^{-3/2}}{\left\lvert\mathbb{W}_{\Box,\sigma,\mu}' \right\rvert }  \right)^{\frac{1}{2} - \frac{1}{p}} \left( \sum_{T' \in \mathbb{W}_{\Box,\sigma,\mu}' } \left\lVert  f_{T'} \right\rVert_2^2 \right)^{1/2}. \]
The $R^{-3/2}$ factor in the preceding inequality comes from the $\approx R^{-3/4}$ factor in the inequality assumed at scale $\approx R^{1/2}$, multiplied with the $R^{-3/4}$ factor arising from the Jacobian of the Lorentz rescaling map. 

Putting this into \eqref{lorentz}, and recalling that $\left\lvert \mathbb{W}_{\Box,\sigma,\mu}' \right\rvert$ is essentially constant as $\Box$ ranges over $\mathbb{B}$, gives 
\begin{multline} \label{pause60} \lVert f\rVert_{L^p(Y)} \lesssim_{\epsilon,\delta} C_{\epsilon,\delta}R^{5\epsilon/6+ O(\delta)} \\
\times \left( \frac{M'M''R^{-3/2}}{\left\lvert \mathbb{W}_{\Box,\sigma,\mu}' \right\rvert } \right)^{\frac{1}{2} - \frac{1}{p} } \left( \sum_{\Box \in \mathbb{B}} \left(  \sum_{T' \in \mathbb{W}_{\Box,\sigma,\mu}' } \left\lVert  f_{T'} \right\rVert_2^2 \right)^{p/2} \right)^{1/p}. \end{multline}
By \eqref{dyadicineq}, the dyadically constant assumption on $\left\lvert \mathbb{W}_{\Box,\sigma,\mu}' \right\rvert$, the dyadically constant assumption $\left\lVert  f_T \right\rVert_2\sim c$ in the theorem statement, and the dyadically constant property of $\left\lVert  f_{T'} \right\rVert_2$ from \eqref{Wdefn}, 
\begin{align*} \sum_{\Box \in \mathbb{B}} \left(  \sum_{T' \in \mathbb{W}_{\Box,\sigma,\mu}' } \left\lVert  f_{T'} \right\rVert_2^2 \right)^{p/2} &\lesssim \left\lvert \mathbb{B} \right\rvert \left( \left\lvert \mathbb{W}_{\Box,\sigma,\mu}' \right\rvert  \frac{c^2}{\mu} \right)^{p/2} \\
&\lesssim \left\lvert \mathbb{B} \right\rvert \left( \frac{\left\lvert \mathbb{W}_{\Box,\sigma,\mu}' \right\rvert}{\left\lvert\mathbb{W}\right\rvert \mu} \sum_{T \in \mathbb{W}}  \left\lVert f_T \right\rVert_2^2 \right)^{p/2}.  \end{align*} 
Taking both sides to the power $1/p$ gives
\[  \left( \sum_{\Box \in \mathbb{B}} \left(  \sum_{T' \in \mathbb{W}_{\Box,\sigma,\mu}' } \left\lVert  f_{T'} \right\rVert_2^2 \right)^{p/2} \right)^{1/p} \lesssim  \left\lvert \mathbb{B} \right\rvert^{1/p}  \left( \frac{\left\lvert \mathbb{W}_{\Box,\sigma,\mu}' \right\rvert }{\left\lvert\mathbb{W}\right\rvert \mu} \sum_{T \in \mathbb{W}} \left\lVert f_T \right\rVert_2^2 \right)^{1/2}. \]
Putting this into \eqref{pause60} gives 
\begin{multline} \label{homebase} \lVert f\rVert_{L^p(Y)} \lesssim_{\epsilon,\delta} C_{\epsilon,\delta}R^{5\epsilon/6+ O(\delta)} \\
\times \left( \frac{M'M''R^{-3/2}}{\left\lvert \mathbb{W} \right\rvert\mu } \right)^{\frac{1}{2} - \frac{1}{p} } \left( \frac{  \left\lvert \mathbb{B} \right\rvert \left\lvert \mathbb{W}_{\Box,\sigma,\mu}' \right\rvert }{\left\lvert \mathbb{W} \right\rvert\mu } \right)^{1/p} \left( \sum_{T \in \mathbb{W}} \left\lVert  f_T \right\rVert_2^2  \right)^{1/2}. \end{multline} 

It remains to bound the two terms out the front of the right-hand side. For the second term, 
\[ \left\lvert \mathbb{W} \right\rvert\mu \quad = \quad  \sum_{T \in \mathbb{W}} \mu \quad \gtrsim  \quad \sum_{\Box \in \mathbb{B}} \sum_{T' \in \mathbb{W}_{\Box,\sigma,\mu}'} 1 \quad \sim \quad \left\lvert \mathbb{B} \right\rvert \left\lvert \mathbb{W}_{\Box,\sigma,\mu}' \right\rvert.  \]
Division gives 
\begin{equation} \label{Wbound} \frac{\left\lvert \mathbb{B} \right\rvert\left\lvert \mathbb{W}_{\Box,\sigma,\mu}' \right\rvert}{\left\lvert \mathbb{W} \right\rvert\mu} \lesssim 1, \end{equation}
which bounds the second bracketed term in \eqref{homebase}. For the first term in \eqref{homebase}, fix any cube $Q \subseteq Y'$. Then (denoting Lebesgue measure by $m$)
\begin{align}\notag  M'M''  &\lesssim \sum_{\substack{\Box \in \mathbb{B} \\ Y_{\Box} \cap 2Q \neq \emptyset }} M' \\
\notag &\lesssim \sum_{\substack{\Box \in \mathbb{B} \\ Y_{\Box} \cap 2Q \neq \emptyset }} \sum_{Q_{\Box} \subseteq Y_{\Box}} \frac{ M' m(Q_{\Box} \cap 3Q)}{m(Y_{\Box} \cap 3Q ) } \\
\label{precursor} &\lesssim \sum_{\substack{\Box \in \mathbb{B} \\ Y_{\Box} \cap 2Q \neq \emptyset }} \sum_{Q_{\Box} \subseteq Y_{\Box}} \sum_{\substack{T' \in \mathbb{W}_{\Box,\sigma,\mu}'   \\
 10T' \cap Q_{\Box} \neq \emptyset}} \frac{ m(Q_{\Box} \cap 3Q)}{m(Y_{\Box} \cap 3Q ) } \\
\label{explain} &\leq \sum_{\substack{\Box \in \mathbb{B} \\ Y_{\Box} \cap 2Q \neq \emptyset }} \sum_{Q_{\Box} \subseteq Y_{\Box}} \sum_{\substack{T' \in \mathbb{W}_{\Box,\sigma,\mu}'   \\
 2T(T') \cap Q \neq \emptyset }} \frac{ m(Q_{\Box} \cap 3Q)}{m(Y_{\Box} \cap 3Q ) } \\
\notag &\lesssim \sum_{\Box \in \mathbb{B}} \sum_{\substack{T' \in \mathbb{W}_{\Box,\sigma,\mu}'   \\
 2T(T') \cap Q \neq \emptyset }}  1 \\
\label{Mbound} &\lesssim \mu M, \end{align}  
where $T(T')$ is the large tube such that $T'$ is one of the boxes covering $T(T')$. Abbreviate $T(T')$ by $T$. The only step in the preceding chain of inequalities that does not follow straightforwardly from the definitions is that \eqref{precursor} is bounded by \eqref{explain}; to check this it suffices to show that given $\Box \in \mathbb{B}$, $T' \in \mathbb{W}_{\Box,\sigma,\mu}'$ and a cube $Q \subseteq Y'$
\begin{equation} \label{annoying10}  Q_{\Box} \cap 3Q \neq \emptyset \text{ and } 10T' \cap Q_{\Box} \neq \emptyset \quad  \Rightarrow \quad 2T(T') \cap Q \neq \emptyset. \end{equation}
To verify this, the assumption $10T' \cap Q_{\Box} \neq \emptyset$ implies that $Q_{\Box} \subseteq 20 T'$ (the set $L^{-1}(Q_{\Box})$ is a cube of side lengths $R^{1/4 + \delta/4}$ whilst $L^{-1}(T')$ is a tube of side lengths $\sim R^{1/4 + \delta/2} \times R^{1/4+\delta/2} \times R^{1/2}$ as outlined in the Lorentz rescaling step). The set $1.5T$ contains $B(0,2R) \cap 20T' \supseteq Q_{\Box}$ (for sufficiently large $R$), since the two shorter side lengths of $T$ exceed those of $T'$ by a factor of $100$. Therefore $1.5T$ contains $Q_{\Box} \cap 3Q$ and in particular intersects $3Q$. Hence $2T \cap Q \neq \emptyset$ since the side lengths of $T$ are much bigger than those of $Q$. This verifies \eqref{annoying10}. 

 Putting \eqref{Wbound} and \eqref{Mbound} into \eqref{homebase} gives 
\[ \lVert f\rVert_{L^p(Y)}  \lesssim_{\epsilon, \delta} C_{\epsilon,\delta}R^{5\epsilon/6+ O(\delta)} \left( \frac{MR^{-3/2}}{\left\lvert \mathbb{W} \right\rvert } \right)^{\frac{1}{2} - \frac{1}{p} } \left( \sum_{T \in \mathbb{W}} \left\lVert  f_T \right\rVert_2^2  \right)^{1/2}. \]
The induction will close if $\delta$ is small enough compared to $\epsilon$, and if $R \geq R_0$ for some large constant $R_0$, depending on $\epsilon$ and $\delta$, which is large enough to eliminate implicit constants (the theorem holds with constant $C_{\epsilon,\delta}$ if $R \leq R_0$). This finishes the proof.  \end{proof}

\section{Projections onto 2-dimensional planes: Proofs of Theorem~\ref{curvature}, and of Proposition~\ref{projection}} \label{sectFourier}
Theorem~\ref{curvature} (and a higher dimensional result, Proposition~\ref{projection}) will both be deduced as corollaries of the following more general proposition. To state it, let $G \in C^2\left( \overline{\Omega}, S^d \right)$ for some nonempty bounded open set $\Omega \subseteq \mathbb{R}^{d-1}$, and define $\beta(\alpha) = \beta_G(\alpha)$ to be the supremum over all $\beta \geq 0$ such that
\begin{equation} \label{surfacebound2} \int_{\Omega} \int_{1/2}^1 \left\lvert \widehat{\mu} (R \rho G(y) ) \right\rvert^2 \, d\rho \, dy \leq C_{G, \beta} \lVert \mu\rVert c_{\alpha}(\mu) R^{-\beta}, \end{equation}
for all Borel measures $\mu$ supported in the unit ball and all $R \geq 1$. Equivalently, $\beta(\alpha)$ is the supremum over all $\beta \geq 0$  such that 
\begin{equation} \label{nhdbound2} \int_{\mathcal{N}_{1/R}(\Gamma(G))} \left\lvert\widehat{\mu}(R\xi)\right\rvert^2 \, d\xi \leq C_{G,\beta}\lVert \mu\rVert c_{\alpha}(\mu) R^{-1 - \beta}, \end{equation}
for all Borel measures $\mu$ supported in the unit ball and all $R \geq 1$. See e.g.~\cite{oberlin} for the equivalence of \eqref{surfacebound2} and \eqref{nhdbound2}, and see Eq.~\eqref{Gammadefn} in Subsection~\ref{notation} on notation for the definition of $\Gamma(G)$.
\begin{proposition} \label{proposition} Let $\Omega$ be a nonempty bounded open set in $\mathbb{R}^{d-1}$ and fix $F,G \in C^2\left( \overline{\Omega}, S^d \right)$, such that $\langle F, G \rangle \equiv 0$,
\begin{equation} \label{rankassumption} \rank \begin{pmatrix} G & DF \end{pmatrix} \equiv d, \quad \rank \begin{pmatrix} F & DG \end{pmatrix} \equiv \rank DG  \equiv  d-1, \end{equation}
and 
\begin{equation} \label{cdn2} \min_{\substack{\lambda \in \mathbb{R} \\
y \in \overline{\Omega} }} \left\lvert \det \begin{pmatrix} F(y) & G(y) & DF(y) + \lambda DG(y) \end{pmatrix} \right\rvert>0, \end{equation}
where $DF$ and $DG$ denote the $(d+1) \times (d-1)$ matrices of partial derivatives of $F$ and $G$. Let $\pi_y$ be projection onto $\spn\{ F(y),G(y) \}$.

Fix $\alpha \in (0, d+1)$, and let $\mu$ be a Borel measure on $\mathbb{R}^{d+1}$ with $\diam \supp \mu \leq C$. If $s< \min\left\{ \alpha, \beta_G(\alpha)+\frac{1}{2} \right\}\leq 2$ then
\[ \int_{\Omega} I_s\left( \pi_{y \#} \mu \right) \, dy \lesssim_{\alpha,s,C,F,G,\Omega} c_{\alpha}(\mu) \mu\left( \mathbb{R}^{d+1} \right), \]
and if $\min\left\{ \alpha, \beta_G(\alpha) + \frac{1}{2} \right\} > 2$, then 
\[ \int_{\Omega} \left\lVert \pi_{y\#} \mu  \right\rVert_{L^2(\mathbb{R}^{d+1}, \mathcal{H}^2 )}^2 \, dy \lesssim_{\alpha,C,F,G,\Omega} c_{\alpha}(\mu) \mu\left( \mathbb{R}^{d+1} \right).  \]

Consequently, for any analytic set $B \subseteq \mathbb{R}^{d+1}$, 
\[ \dim \pi_y (B) \geq \min\left\{2, \dim B, \beta_G(\dim B) + \frac{1}{2} \right\}  \quad \text{for a.e.~$y \in \Omega$,} \]
and if $\min\left\{ \dim B, \beta_G(\dim B) + \frac{1}{2}\right\} >2$ then $\mathcal{H}^2(\pi_y(B)) >0$ for a.e.~$y \in \Omega$. 
\end{proposition}
% Most of this section will be devoted to proving Theorem~\ref{proposition}. 
\begin{remark} Roughly speaking, the condition \eqref{rankassumption} means that $F$ is a partial derivative of $G$, and that the 2-dimensional plane spanned by $F(y)$ and $G(y)$ does not get stuck in any hyperplane as $y$ varies. The two examples that Proposition~\ref{proposition} is based on are Theorem~\ref{curvature} and Proposition~\ref{projection}, which will both be deduced as corollaries in Subsection~\ref{corollaries} by verifying the conditions of Proposition~\ref{proposition} in each case. The conditions \eqref{rankassumption} and \eqref{cdn2} were originally obtained by proving Theorem~\ref{curvature} and Proposition~\ref{projection} separately, and then formalising the assumptions which make the proof work; the inequality \eqref{cdn2} ensures that a Jacobian in the proof does not vanish. \end{remark}
\begin{proof}[Proof of Proposition~\ref{proposition}] Assume that $\Omega = (0,1)^{d-1} + B(0, \delta)$ for some small $\delta >0$; without loss of generality it suffices to bound the integral over $(0,1)^{d-1}$. %Assume also that $\dim B >0$ since otherwise there is nothing to prove. 
Given $\alpha$ and $\mu$ with $c_{\alpha}(\mu) < \infty$, let  
\begin{equation} \label{epsilondefn5} \epsilon \in \left(0,\min\left\{ \frac{\alpha}{4}, \frac{1}{(20(d+1))^4} \right\} \right). \end{equation} 
Set
\begin{equation} \label{alphadefn5} \alpha^* = \min  \left\{ 2,\alpha-\epsilon,\beta(\alpha) + \frac{1}{2}-10(d+1)\epsilon^{1/4} \right\}, \quad \kappa = 1-\frac{\epsilon}{2(d-1)}. \end{equation}
Then $\alpha^* >0$ and $\kappa \in (0,1)$ by \eqref{epsilondefn5}. Using Frostman's lemma and the Fourier energy, it suffices to show that 
\begin{equation} \label{L2step2}  \int_{\mathbb{R}^2} \lvert \eta\rvert^{\alpha^*-2}\int_{[0,1]^{d-1}} \left\lvert \widehat{\mu}\left( \eta_1F(y) + \eta_2 G(y) \right) \right\rvert^2 \, dy \,  d\eta \lesssim \max\left\{I_{\alpha}(\mu), \left\lVert \mu\right\rVert c_{\alpha}(\mu) \right\} \end{equation}
To show this, by symmetry it suffices to bound the integral over the positive quadrant $\mathbb{R}^2_+$. This integral can be written as
\begin{align}\notag &\int_{\mathbb{R}^2_+} \lvert \eta\rvert^{\alpha^*-2}\int_{[0,1]^{d-1}} \left\lvert \widehat{\mu}\left( \eta_1F(y) + \eta_2 G(y) \right) \right\rvert^2 \, dy \,  d\eta \\
\label{energy2} &\quad = \int_{\eta_1 > \eta_2^{\kappa} \geq 0} \lvert \eta\rvert^{\alpha^*-2}\int_{[0,1]^{d-1}} \left\lvert \widehat{\mu}\left( \eta_1F(y) + \eta_2 G(y) \right) \right\rvert^2 \, dy \,  d\eta  \\
\label{singular2} &\qquad + \int_{0 \leq \eta_1  \leq \eta_2^{\kappa}} \lvert \eta\rvert^{\alpha^*-2}\int_{[0,1]^{d-1}} \left\lvert \widehat{\mu}\left( \eta_1F(y) + \eta_2 G(y) \right) \right\rvert^2 \, dy \,  d\eta. \end{align}
 Let 
\begin{equation} \label{variablechange2} \xi = \xi(\eta, y) = \eta_1F(y) + \eta_2 G(y). \end{equation}
Then
\[ \left\lvert\frac{d\xi}{d\eta \, dy }\right\rvert = \left\lvert \det\begin{pmatrix} F(y)  & G(y) & \eta_1 DF(y) + \eta_2 DG(y) \end{pmatrix} \right\rvert \gtrsim \eta_1^{d-1}, \]
by the assumption in \eqref{cdn2}.  Applying the change of variables from \eqref{variablechange2} to the integral in \eqref{energy2} therefore results in
\begin{multline*} \int_{\eta_1 > \eta_2^{\kappa} \geq 0} \lvert \eta\rvert^{\alpha^*-2}\int_{[0,1]^{d-1}} \left\lvert \widehat{\mu}\left( \eta_1F(y) + \eta_2 G(y) \right) \right\rvert^2 \, dy \,  d\eta \\
 \lesssim \|\mu\|^2+  \int_{\mathbb{R}^{d+1}} \lvert \xi\rvert^{\alpha^* + (1-\kappa)(d-1)-(d+1)} \left\lvert \widehat{\mu}\left( \xi \right) \right\rvert^2 \,  d\xi \lesssim \|\mu\|c_{\alpha}(\mu) + I_{\alpha}(\mu), \end{multline*}
since $\alpha^* + (1-\kappa)(d-1) < \alpha$ by \eqref{alphadefn5}.

For the remaining integral in \eqref{singular2}, define $r$ and $t$ as functions of $\eta_1$ and $\eta_2$ by 
\begin{equation} \label{rchange2} r^2 = \eta_1^2 + \eta_2^2, \quad \eta_2 = rt, \end{equation}
so that 
\[ \left\lvert\frac{ dr \, dt}{d\eta_1 \, d\eta_2 } \right\rvert = \frac{\sqrt{1-t^2}}{r} \] 

 Using the change of variables from \eqref{variablechange2} and \eqref{rchange2}, the integral in \eqref{singular2} will be shown to satisfy
\begin{align} \notag & \int_{0 < \eta_1 \leq \eta_2^{\kappa} } \lvert \eta\rvert^{\alpha^*-2}\int_{[0,1]^{d-1}} \left\lvert \widehat{\mu}\left( \eta_1F(y) + \eta_2 G(y) \right) \right\rvert^2 \, dy \,  d\eta \\ 
\notag &\quad \lesssim_{\epsilon} 1  \\
\label{preneighbourhood2} &\qquad + \sum_{j \geq (1-\kappa)^{-1}} \sum_{k= \left\lfloor 2j(1-\kappa) \right\rfloor-2}^{\left\lfloor j (1-\epsilon) \right\rfloor} \\
\notag &\qquad \quad \int_{1-2^{-(k-1)}}^{1-2^{-(k+1)}}  \int_{2^{j-1}}^{2^j} \int_{[0,1]^{d-1}} \frac{2^{j(\alpha^*-1)}}{\sqrt{1-t^2}} \left\lvert \widehat{\mu}\left( r  G_t(y) \right) \right\rvert^2 \, dy \, dr \, dt \\
\notag &\qquad +  \sum_{j \geq (1-\kappa)^{-1}} \\
\notag &\qquad \quad  \int_{1-2^{-j(1-\epsilon)}}^1  \int_{2^{j-1}}^{2^j} \int_{[0,1]^{d-1}} \frac{2^{j(\alpha^*-1)}}{\sqrt{1-t^2}} \left\lvert \widehat{\mu}\left( r  G_t(y) \right) \right\rvert^2 \, dy \, dr \, dt \\
\notag &\quad \lesssim_{\epsilon} 1  \\
\label{neighbourhood2}  &\qquad+  \sum_{j \geq (1-\kappa)^{-1}} \sum_{k=\left\lfloor 2j(1-\kappa) \right\rfloor-2}^{\left\lfloor j (1-\epsilon) \right\rfloor} \int_{\mathcal{N}_{C2^{-k}}(\Gamma(G) \cup \frac{1}{2} \Gamma(G))} 2^{j\alpha^* + \frac{k}{2}} \left\lvert \widehat{\mu}\left( 2^j \xi \right) \right\rvert^2 \,   d\xi \\
\label{surface2} &\qquad+  \sum_{j \geq (1-\kappa)^{-1}} \int_{1-2^{-(j-1)(1-\epsilon)}}^1 \int_{2^{j-1}}^{2^j} \int_{[0,1]^{d-1}}  \frac{2^{j(\alpha^*-1)}}{\sqrt{1-t^2}} \left\lvert \widehat{\mu}\left( r  G_t(y) \right) \right\rvert^2 \, dy \, dr \, dt, \end{align}
where $C>2$ is a large constant to be chosen in a moment, 
\[  G_t(y) :=  \sqrt{1-t^2}F(y) + tG(y), \]
and $\Gamma(G)$ assumes $G$ has domain $[0,1]^{d-1} + B(0, \delta)$. The only preceding inequality that does not follow from the change of variables and a dyadic decomposition is that \eqref{preneighbourhood2} is bounded by \eqref{neighbourhood2}. To see this, it suffices to check that the equality   
\begin{equation} \label{equality2} r\sqrt{1-t^2}F(y) + rtG(y) = 2^j \xi, \quad 2^{j-1} \leq r \leq 2^j \end{equation}
implies that $\xi \in \mathcal{N}_{C2^{-k}}(\Gamma(G) \cup \frac{1}{2} \Gamma(G))$. This holds trivially if $k=0$ so assume $k \geq 1$. Division of \eqref{equality2} by $2^j$ gives 
\begin{equation} \label{perturb2} \xi = \lambda_1 F(y) + \lambda_2 G(y), \quad \text{where} \quad \lvert \lambda_1\rvert \leq 2^{1/2} \cdot 2^{-k/2} \text{ and } 1/4 \leq \lambda_2 \leq 1. \end{equation}
By the assumption \eqref{rankassumption} in the proposition and by compactness,
\[ F(y)  = DG(y) x, \]
for some $x \in \mathbb{R}^{d-1}$ with $\lvert x\rvert \lesssim_{F,G} 1$. Letting $h = \frac{\lambda_1x}{\lambda_2}$ in \eqref{perturb2} gives
\[ \xi = \lambda_2 G(y+h) + O\left(\left\lvert h\right\vert^2\right), \]
where the implicit constant is uniform depending only on $G$. Since $x$ is uniformly bounded, $h$ satisfies $\lvert h\rvert \lesssim \lambda_1$, and therefore $\dist(\xi, \Gamma(G) \cup \frac{1}{2} \Gamma(G)) \lesssim \lambda_1^2 \lesssim 2^{-k}$. Hence $\xi \in \mathcal{N}_{C2^{-k}}(\Gamma(G) \cup \frac{1}{2} \Gamma(G))$ provided $C$ is chosen large enough.  This verifies that \eqref{preneighbourhood2} is bounded by \eqref{neighbourhood2}.

It remains to bound the sums in \eqref{neighbourhood2} and \eqref{surface2}. Let $\rho = 2^{j-k}$ and let $\mu_{\rho}$ be the pushforward measure $\rho_{\#} \mu$, which is defined by $\mu_{\rho}(E) = \mu( \rho^{-1} E)$ for any Borel set $E$. Let $\{B_m\}$ be a finitely overlapping cover of $B(0,\rho)$ by unit balls and let $\{\psi_m\}$ be a smooth partition of unity subordinate to this cover. Let $\zeta \in C^{\infty}\left(\mathbb{R}^{d+1}\right)$ be a smooth bump function equal to 1 on $B(0,C')$ and supported in $B(0,2C')$, for some constant $C'>C$ to be chosen later. Let $\nu_k$ be the pushforward of the Lebesgue measure on $([0,1]^{d-1} +B(0,\delta))\times \left[2^{k-2},2^k \right]$ under $(y, \lambda) \mapsto \lambda G(y)$, and define $\phi$ on $\mathbb{R}^{d+1}$ by 
\begin{align*} \phi(\xi) &= 2^{k(d-1)}(\nu_k \ast \zeta)(\xi)  \\
&= 2^{k(d-1)} \int \zeta(\xi - \eta) \, d\nu_k(\eta)  \\
&= 2^{k(d-1)} \int_{[0,1]^{d-1} +B(0,\delta)} \int_{2^{k-1}}^{2^k} \zeta(\xi - \lambda G(y) ) \, d\lambda \, dy. \end{align*}
 Then $\phi$ has support in $\mathcal{N}_{2C'}(\Gamma_{2^k}(G) \cup \Gamma_{2^{k-1}}(G))$, with $\phi \sim 1$ on $\mathcal{N}_{C}(\Gamma_{2^k} \cup \Gamma_{2^{k-1}}(G))$, where $C'$ is now chosen large enough to ensure this. The inverse Fourier transform of $\phi$ satisfies 
\[ \left\lvert \widecheck{\phi}(x) \right\rvert \lesssim_N 2^{kd} \lvert x\rvert^{-N}, \]
since $\zeta$ is Schwartz. Hence 
\begin{align}\notag &\int_{\mathcal{N}_{C2^{-k}}(\Gamma(G) \cup \frac{1}{2} \Gamma(G))}  \left\lvert \widehat{\mu}\left( 2^j \xi \right) \right\rvert^2 \,   d\xi \\
 \notag &\quad \lesssim \int \phi\left(2^k \xi\right) \left\lvert \widehat{\mu_{\rho}}\left( 2^k \xi \right) \right\rvert^2 \,   d\xi \\
\notag & \quad= \frac{1}{2^{(d+1)k}}\int \phi\left(\xi\right) \left\lvert \sum_m  \widehat{\psi_m\mu_{\rho}}\left(\xi\right) \right\rvert^2 \,   d\xi \\
\label{diag20} &\quad\lesssim \rho^{(d+1)\epsilon^2} \sum_m \int_{\mathcal{N}_{2C'2^{-k}}\left(\Gamma(G) \cup \frac{1}{2} \Gamma(G)\right)} \left\lvert\widehat{\psi_m \mu_{\rho}}\left(2^k \xi\right) \right\rvert^2 \, d\xi  \\
\label{offdiag20} &\qquad + \frac{1}{2^{(d+1)k}}\sum_{\dist(B_m, B_n) \geq \rho^{\epsilon^2}} \left\lvert \int  \phi(\xi)\widehat{\psi_m \mu_{\rho}}(\xi ) \overline{\widehat{\psi_n \mu_{\rho}}(\xi )} \, d\xi  \right\rvert. \end{align}

By the definition of $\beta(\alpha)$ (see \eqref{nhdbound2}), the summands in \eqref{diag20} satisfy
\begin{equation} \label{diag4} \int_{\mathcal{N}_{2C'2^{-k}}\left(\Gamma(G) \cup \frac{1}{2} \Gamma(G) \right)} \left\lvert\widehat{\psi_m \mu_{\rho}}\left(2^k \xi\right) \right\rvert^2 \, d\xi \lesssim_{\epsilon} \lVert \psi_m \mu_{\rho} \rVert c_{\alpha}( \mu_{\rho})  2^{k(\epsilon- \beta(\alpha)-1)}. \end{equation} 
By the version of Plancherel's Theorem for measures (see e.g.~\cite[Eq. 3.27]{mattila4}), the summands in \eqref{offdiag20} satisfy 
\begin{align} \notag  \left\lvert \int_{\mathbb{R}^{d+1}} \phi(\xi)\widehat{\psi_m \mu_{\rho}}(\xi ) \overline{\widehat{\psi_n \mu_{\rho}}(\xi )} \, d\xi  \right\rvert &= \left\lvert \int \int \widecheck{\phi}(x-y) \psi_m(x) \psi_n(y) \, d\mu_{\rho}(x) \, d\mu_{\rho}(y) \right\rvert  \\
\label{offdiag30} &\lesssim_N 2^{(d+1)k} \rho^{-\epsilon^2 N}\lVert \psi_m \mu_{\rho}\rVert\lVert \psi_n\mu_{\rho}\rVert. \end{align} 
Substituting \eqref{diag4} and \eqref{offdiag30} with $N = \left\lceil \epsilon^{-4} \right\rceil$ into \eqref{diag20} and \eqref{offdiag20} gives
\begin{align*} &\int_{\mathcal{N}_{C2^{-k}}(\Gamma(G) \cup \frac{1}{2} \Gamma(G))}  \left\lvert \widehat{\mu}\left( 2^j \xi \right) \right\rvert^2 \,   d\xi \\
&\quad \lesssim_{\epsilon} \rho^{(d+1)\epsilon^2} \left\lVert \mu_{\rho}\right\rVert c_{\alpha}( \mu_{\rho}) 2^{k(\epsilon- \beta(\alpha)-1)} + 2^{-k}\rho^{\frac{-1}{\epsilon^2}}\left\lVert \mu_{\rho} \right\rVert^2  \\
&\quad \leq  \lVert \mu\rVert c_{\alpha}(\mu) \left(2^{j((d+1)\epsilon^2-\alpha)} 2^{k(\alpha-\beta(\alpha) -1+\epsilon)} +  2^{\frac{-j}{\epsilon^2}} 2^{-k+\frac{k}{\epsilon^2}}\right) \\
&\quad \lesssim \lVert \mu\rVert c_{\alpha}(\mu) 2^{j((d+1)\epsilon^2-\alpha)} 2^{k(\alpha-\beta(\alpha) -1+\epsilon)},
 \end{align*}
since $k \leq j(1-\epsilon)$ and $\epsilon$ is small (see the definition in \eqref{epsilondefn5}). Applying this bound to \eqref{neighbourhood2} gives
\begin{multline*} \eqref{neighbourhood2} \lesssim_{\epsilon} \lVert \mu\rVert c_{\alpha}(\mu) \sum_{j \geq (1-\kappa)^{-1}}  2^{j\left(\alpha^* - \alpha +(d+1)\epsilon^2\right)} \sum_{k = \left\lfloor 2j(1-\kappa) \right\rfloor-2}^{\left\lceil j(1-\epsilon) \right\rceil } 2^{k(\alpha-\beta(\alpha) -\frac{1}{2}+\epsilon)}\\ \lesssim_{\epsilon} \lVert \mu\rVert c_{\alpha}(\mu), \end{multline*}
by the definition of $\alpha^*$ and $\kappa$ in \eqref{alphadefn5}. This finishes the bound on \eqref{neighbourhood2}.

It remains to bound the integral in \eqref{surface2}. Let $\zeta_N(x) = \frac{1}{1+\lvert x\rvert^N}$ for some large $N$ to be chosen later. Then $\lvert \widehat{\mu}\rvert \lesssim_N \left\lvert\widehat{\mu}\right\rvert \ast \zeta_N$ since $\mu = \mu \varphi$ for a fixed Schwartz function $\varphi$ equal to 1 on the unit ball. Hence 
\begin{equation} \label{convolve} \int_{[0,1]^{d-1}}  \left\lvert \widehat{\mu}\left( r  G_t(y) \right) \right\rvert^2 \, dy \lesssim_N \int_{[0,1]^{d-1}}  \left(\left\lvert \widehat{\mu} \right\rvert \ast \zeta_N \right)(rG_t(y))^2 \, dy. \end{equation}
Cover the cube $[0,1]^{d-1}$ with cubes of side length $1/M$. If $M$ is large enough, then by the assumptions on $F$ and $G$ from \eqref{rankassumption}, in each cube $F(y)$ can be written as $F(y) = DG(y) x_y$, where $x$ is a smooth function of $y$ in each cube with $\lvert x\rvert \lesssim 1$ and the Jacobian of $y \mapsto x_y$ is $\lesssim 1$. Hence 
\[ G_t(y) = \sqrt{1-t^2}F(y) + tG(y) = tG\left(y + \frac{x_y\sqrt{1-t^2}}{t} \right) + O\left(2^{-j(1-\epsilon)} \right), \]
where the implicit constant is uniform. Using the essentially constant property (similarly to \eqref{constantproperty}) of $\zeta_N$ gives, with $2^{j-1} \leq r \leq 2^j$,
\[ \left(\left\lvert \widehat{\mu} \right\rvert \ast \zeta_N \right)(rG_t(y)) \lesssim_N 2^{j\epsilon N} \left(\left\lvert \widehat{\mu} \right\rvert \ast \zeta_N \right)\left( rtG\left(y + \frac{x_y\sqrt{1-t^2}}{t} \right)\right). \]
Combining this with \eqref{convolve} and applying the change of variables $\widetilde{y} = y + \frac{x_y\sqrt{1-t^2}}{t}$ gives, for $j$ large enough,
\[  \int_{[0,1]^{d-1}}  \left\lvert \widehat{\mu}\left( r  G_t(y) \right) \right\rvert^2 \, dy \lesssim_N 2^{j \epsilon N} \int_{[0,1]^{d-1}+B(0,\delta)}  \left(\left\lvert \widehat{\mu} \right\rvert \ast \zeta_N \right)(rtG(y))^2 \, dy, \]
where $t \geq 2^{-(j-1)(1-\epsilon)}$. Putting this into the two innermost integrals of \eqref{surface2} and applying Minkowski's inequality results in 
\begin{align} \notag &\int_{2^{j-1}}^{2^j} \int_{[0,1]^{d-1}} \left\lvert \widehat{\mu}\left( r  G_t(y) \right) \right\rvert^2 \, dy \, dr \\
\notag  &\quad \lesssim_N 2^{j \epsilon N} \int_{2^{j-1}}^{2^j} \int_{[0,1]^{d-1}+B(0,\delta)}  \left(\left\lvert \widehat{\mu} \right\rvert \ast \zeta_N \right)(rtG(y))^2 dy \, dr \\
\label{complex} &\quad \lesssim_{\mu}  2^{j \epsilon N}  \left(\int_{B\left(0, 2^{j \sqrt{\epsilon}}\right)}  \left(\int_{2^{j-1}}^{2^j} \int_{[0,1]^{d-1}+B(0,\delta)}  \left\lvert \widehat{\mu_{z,t}}(rG(y)) \right\rvert^2 \, dy \, dr\right)^{1/2} \, dz\right)^2 \\
\notag &\qquad + 2^{j \left(\epsilon -\sqrt{\epsilon}\right)N+j + (d+1) \sqrt{\epsilon}}\lVert \mu\rVert^2, \end{align}
where for each $t \in [1/2,1]$ and $z \in B\left(0, 2^{j \sqrt{\epsilon} } \right)$, $\mu_{z,t}$ is the complex measure defined by 
\[ \int f(x) \, d\mu_{z,t}(x) = \int f(tx) e^{2\pi i \langle z, x \rangle } \, d\mu(x). \]
The positive measure $\lvert \mu_{z,t}\rvert$ has support of comparable size to $\mu$ and satisfies $c_{\alpha}(\lvert \mu_{z,t}\rvert) \lesssim c_{\alpha}(\mu)$, so applying the triangle inequality and the definition (see  \eqref{surfacebound2}) of $\beta(\alpha)$ to \eqref{complex} gives 
\[\int_{2^{j-1}}^{2^j} \int_{[0,1]^{d-1}} \left\lvert \widehat{\mu}\left( r  G_t(y) \right) \right\rvert^2 \, dy \, dr \lesssim_{\epsilon}   \lVert \mu\rVert c_{\alpha}(\mu) 2^{j( (d+1) \sqrt{\epsilon}+\epsilon)} 2^{j(1- \beta(\alpha))},  \]
by choosing $N \sim \epsilon^{-3/4}$, since $\epsilon$ is very small. Putting this into \eqref{surface2} yields  
\begin{align*} \eqref{surface2} &\lesssim_{\epsilon} \lVert \mu\rVert c_{\alpha}(\mu) \\
&\quad + \lVert\mu\rVert c_{\alpha}(\mu) \sum_{j \geq (1-\kappa)^{-1}} 2^{j(\alpha^*-\beta(\alpha)+ (d+1) \sqrt{\epsilon}+\epsilon^{1/4})}  \int_{1-2^{-(j-1)(1-\epsilon)}}^1   \frac{1}{\sqrt{1-t^2}}  \, dt \\
& \lesssim \lVert \mu\rVert c_{\alpha}(\mu)\left(1+ \sum_{j \geq (1-\kappa)^{-1}} 2^{j(\alpha^*-\beta(\alpha) - \frac{1}{2}+ (d+1) \sqrt{\epsilon}+\epsilon+\epsilon^{1/4})}\right) \\
&\lesssim \lVert\mu\rVert c_{\alpha}(\mu)  \end{align*}
by the definition of $\alpha^*$ in \eqref{alphadefn5}.  

 This bounds the remaining piece, finishing the bound on \eqref{singular2}, which gives \eqref{L2step2} and finishes the proof. %Hence
%\[ \dim \pi_y(B) \geq \min\left\{2,s -\epsilon, \beta(s)+ \frac{1}{2}-10(d+1)\epsilon^{1/4}\right\} \quad \text{for a.e.~$y \in [0,1]^{d-1}$.} \] 
%Letting $\epsilon \to 0$ gives
%\[ \dim \pi_y(B) \geq \min\left\{2, \dim B, \beta( \dim B)+ \frac{1}{2}\right\} \quad \text{for a.e.~$y \in [0,1]^{d-1}$,} \] 
%see e.g.~\cite[Lemma~3.1]{wolff} for the continuity of $\beta(\cdot)$. If $\min\left\{\dim B, \beta( \dim B)+ \frac{1}{2}\right\}> 2$, then $\alpha = 2$ for $\epsilon$ small enough, and the finiteness of \eqref{L2step2} yields
%\[ \left\lVert  U_{y\#}\pi_{y\#}\mu \right\rVert_{2,0, \mathbb{R}^2} = \left\lVert  \pi_{y\#}\mu \right\rVert_{L^2(\mathbb{R}^{d+1}, \mathcal{H}^2)}<\infty, \]
%for a.e.~$y \in [0,1]^{d-1}$ (see e.g.~\cite[Theorem~3.3]{mattila4} for the version of Plancherel's Theorem for measures). For such $y$ the set $\pi_y(B)$ therefore supports a nonzero function in $L^2\big(\mathbb{R}^{d+1}, \mathcal{H}^2\big)$, implying that $\mathcal{H}^2(\pi_y(B)) >0$. This finishes the proof.  
\end{proof} 

\subsection{Weighted Fourier restriction inequality for curves of nonvanishing geodesic curvature, and proof of Theorem~\ref{curvature}} \label{corollaries}

The proof of Theorem~\ref{curvature} will require the following (sharp) Fourier restriction inequality from~\cite{oberlin} (see also~\cite{erdogan}).

\begin{theorem}[{\cite[Theorem~1.1]{oberlin}}] \label{torsionfourier}
If $\gamma: [a,b] \to S^2$ is $C^2$ with $\det\left(\gamma,\gamma',\gamma''\right)$ nonvanishing, then for any $\epsilon >0$ and $\alpha \in [0,3]$,  
\begin{equation} \label{surfacebound} \int_a^b \int_{1/2}^1 \left\lvert \widehat{\mu} (R \rho \gamma(\theta) ) \right\rvert^2 \, d\rho \, d\theta \leq C_{\epsilon, \gamma} \lVert \mu\rVert c_{\alpha}(\mu) R^{\epsilon - \beta(\alpha)}, \end{equation}
and 
\begin{equation} \label{nhdbound} \int_{\mathcal{N}_{1/R}(\Gamma(\gamma))} \left\lvert\widehat{\mu}(R\xi)\right\rvert^2 \, d\xi \leq C_{\epsilon, \gamma}\lVert \mu\rVert c_{\alpha}(\mu) R^{\epsilon-1- \beta(\alpha)}, \end{equation}
for all $R \geq 1$ and any positive Borel measure $\mu$ supported in the unit ball of $\mathbb{R}^3$, where
\[ \beta(\alpha) := \begin{cases} \alpha, & \alpha \in [0,1/2] \\
1/2, & \alpha \in (1/2,1] \\
\alpha/2, & \alpha \in (1,2] \\
\alpha -1, & \alpha \in (2,3]. \end{cases} \]  \end{theorem}
In~\cite[Theorem~1.1]{oberlin}, the dependence on $\lVert \mu\rVert$, $c_{\alpha}(\mu)$ in \eqref{surfacebound} and \eqref{nhdbound} is not made explicit, but follows from the proof in~\cite{oberlin} (more precisely, the $c_{\alpha}(\mu)$ term comes from~\cite[Lemma~3.1]{oberlin}, and the $\lVert \mu\rVert$ term comes through deducing~\cite[Eq. 3.1]{oberlin} from~\cite[Eq. 3.4]{oberlin}). In~\cite{oberlin} the left-hand side of \eqref{surfacebound} is replaced by the localised version
\[ \int_{-1/2}^{1/2} \int_{1/2}^1 \left\lvert \widehat{\mu} (R \rho (t, \phi(t), 1) ) \right\rvert^2 \, d\rho \, dt, \]
under the assumption that $\phi : [-1/2,1/2] \to \mathbb{R}$ is $C^2$ with nonvanishing first and second derivatives. The version in \eqref{surfacebound} can be deduced from the local one by localising around $t=0$, rotating so that $\gamma(0) = (0,0,1)$, $\gamma'(0) = \frac{1}{\sqrt{2}}(1,1,0)$, letting $t = \frac{\gamma_1(\theta)}{\gamma_3(\theta)}$, $\widetilde{\rho} = \rho\gamma_3(\theta)$,  $\phi(t) = \frac{\gamma_2(\theta)}{\gamma_3(\theta)}$ and making the change of variables from $(\rho, \theta)$ to $\left(\widetilde{\rho}, t \right)$. Under this change of variables, $\phi$ satisfies 
\[ \phi'(0) = \frac{\gamma_2'(0)}{\gamma_1'(0)} = 1, \quad \phi''(0) = \frac{1}{\gamma_1'(0)} \det\begin{pmatrix} \gamma(0) & \gamma'(0) & \gamma''(0) \end{pmatrix} \neq 0, \]
which verifies the nonvanishing first and second derivative conditions in a neighbourhood of $0$. The version in \eqref{nhdbound} is equivalent to the one in \eqref{surfacebound} by the uncertainty principle (the implication $\eqref{nhdbound} \Rightarrow \eqref{surfacebound}$ is shown in~\cite[Eq. 3.2]{oberlin}). Alternatively, the version in \eqref{nhdbound} is the one proved directly in~\cite[Eq. 3.1]{oberlin} (again for the localised curve $(t,\phi(t), 1)$, but the version in \eqref{nhdbound} follows similarly to \eqref{surfacebound}). 

Finally, Theorem~\ref{torsionfourier} is only stated in~\cite{oberlin} for $\alpha >1$, but the case $\alpha \leq 1$ follows from Theorem~1 in~\cite{erdogan} (see also~\cite[Theorem~3.8]{mattila}), whose proof is much simpler than the $\alpha >1$ case. The assumptions of~\cite[Theorem~1]{erdogan} are satisfied with $a = 1/2$ and $b=2$ by rotation invariance and the van der Corput lemma.  

\begin{proof}[Proof of Theorem~\ref{curvature}] Assume $\Omega = (a,b)$, and let $\gamma: [a,b] \to S^2$ be a $C^2$ curve with nonvanishing geodesic curvature $\det\left(\gamma,\gamma',\gamma''\right)$. Without loss of generality assume that $\gamma$ has unit speed. Let $F = \gamma'$ and $G = \gamma \times \gamma'$, so that $\gamma^{\perp} = \spn\{ F, G \}$. Differentiating both sides of $\langle \gamma, \gamma' \rangle =0$ gives $\langle \gamma, \gamma'' \rangle = -1$, and hence 
\[ \rank\begin{pmatrix} G & DF \end{pmatrix} = \rank\begin{pmatrix} \gamma \times \gamma' & \gamma'' \end{pmatrix} = 2. \]
For the other condition in \eqref{rankassumption},
\begin{equation} \label{rank1} DG = (\gamma \times \gamma')' = \gamma \times \gamma'' = -\det\begin{pmatrix} \gamma & \gamma' & \gamma'' \end{pmatrix} \gamma'. \end{equation}
The last equality can be seen by expanding out $\gamma \times \gamma''$ in the basis $\gamma, \gamma', \gamma \times \gamma'$. In this expansion, the coefficient of $\gamma$ is clearly zero, and differentiating $\langle \gamma \times \gamma', \gamma \times \gamma' \rangle$ shows that the coefficient of $\gamma \times \gamma'$ is zero. The coefficient of $\gamma'$ can then be found through the scalar triple product formula $\det\left( a, b, c \right) = \langle a, b \times c \rangle$.  Eq. \eqref{rank1} implies that  
\[ \rank\begin{pmatrix} F & DG \end{pmatrix} = \rank\begin{pmatrix} \gamma' & \gamma \times \gamma'' \end{pmatrix} = 1=  \rank \gamma \times \gamma'' = \rank DG.  \]
This automatically implies \eqref{cdn2}.
% and applying Theorem~\ref{proposition} with $d=2$ then gives 
%\[ \dim \pi_{\theta} (B) \geq \min\left\{2, \dim B, \beta(\dim B) + \frac{1}{2} \right\}  \quad \text{for a.e.~$\theta \in [0,1]$,} \]
%and if $\min\left\{ \dim B, \beta(\dim B) + \frac{1}{2}\right\} >2$, then $\mathcal{H}^2(\pi_{\theta}(B)) >0$ for a.e.~$\theta \in [0,1]^{d-1}$.
By~\cite[Lemma~3.2]{fasslerorponen14}, the curve $\gamma \times \gamma'$ has nonvanishing geodesic curvature since $\gamma$ has nonvanishing geodesic curvature. Substituting the value of $\beta(\alpha)$ from Theorem~\ref{torsionfourier} and applying Proposition~\ref{proposition} with $d=2$ therefore gives \eqref{energyone} and \eqref{energytwo}.
%gives
%\[ \begin{aligned}  \dim \pi_{\theta} (B) &= \dim B, \quad && \dim B \in [0,1]  \\
%\dim \pi_{\theta} (B) &\geq \frac{\dim B +1}{2}, \quad && \dim B \in (1,2]  \\
%\dim \pi_{\theta} (B) &\geq \dim B - \frac{1}{2}, \quad && \dim B \in (2,5/2]  \\
%\mathcal{H}^{2}(\pi_{\theta} (B)) &>0,  \quad && \dim B > 5/2, \end{aligned} \]
%for a.e.~$\theta \in [a,b]$. This proves Theorem~\ref{curvature}. 
\end{proof}

\subsection{Weighted Fourier restriction for the cone, and projections onto 2-dimensional planes in higher dimensions}

For $d \geq 2$ let $\Gamma^d$ be the $d$-dimensional truncated cone in $\mathbb{R}^{d+1}$:
\[ \Gamma^d := \left\{ (\xi,\lvert\xi\rvert) \in \mathbb{R}^d \times \mathbb{R} : \frac{1}{2} \leq \lvert\xi\rvert \leq 1 \right\}. \]
For $d=2$, $\Gamma^2 \cap S^2 = \left\{ \frac{1}{\sqrt{2}}(\cos\theta, \sin \theta, 1) : \theta \in [0,2\pi) \right\}$ is the curve from before, which serves as the model example of a curve on the sphere $S^2 \subseteq \mathbb{R}^3$ with strictly positive geodesic curvature. For $d \geq 2$, $\Gamma^d \cap S^d$ is a codimension 1 submanifold of $S^d$ with second fundamental form corresponding to a constant multiple of the identity matrix at every point, and it therefore serves as a simple example of a codimension 1 submanifold of $S^d$ which is ``positively curved''; meaning the second fundamental form is everywhere strictly positive definite. A restricted family of projections parametrised by $\Gamma^d \cap S^d$ will be constructed here using vector fields on the sphere $S^{d-1}$. 

A function $F: S^{d-1} \to \mathbb{R}^d$ is called a vector field on $S^{d-1}$ if $\langle F(x), x \rangle =0$ for every $x \in S^{d-1}$. A vector field $F: S^{d-1} \to \mathbb{R}^d$ on $S^{d-1}$ is called linear if there exists a $d \times d$ matrix $A$ such that $F(x) = Ax$ for every $x \in S^{d-1}$, and $F$ is called a unit vector field if $\lvert F\rvert \equiv 1$. A $d \times d$ matrix $A$ corresponds to a linear vector field on $S^{d-1}$ if and only if $A^*=-A$.  It is well known that there are no nonvanishing continuous vector fields on $S^{d-1}$ if $d$ is odd. Assume then that $d$ is even, fix a linear unit vector field $A$ on $S^{d-1}$ (e.g. multiplication by $i$), and for $v \in S^{d-1}$ let $\pi_v = \pi_{v,A}$ be the orthogonal projection onto the $2$-dimensional plane
\begin{equation} \label{kplanes} \spn \left\{ \frac{1}{\sqrt{2}}(v,-1), ( Av,0) \right\} \subseteq \mathbb{R}^{d+1}. \end{equation}
This family forms a $(d-1)$-dimensional submanifold of the $2(d-1)$-dimensional Grassmannian $\gr(d+1,2)$. If $d=2$, the family in \eqref{kplanes} parametrises the planes orthogonal to $\frac{1}{\sqrt{2}}\left( \cos \theta, \sin \theta, 1 \right)$, so this generalises the projection family occurring in Theorem~\ref{neat}.\end{sloppypar}

\begin{proposition} \label{projection} Fix  an even integer $d \geq 4$.  For any analytic subset $B$ of $\mathbb{R}^{d+1}$,
\[ \begin{aligned}  \dim \pi_v (B) &= \dim B, && \dim B \in [0,2] \\
\mathcal{H}^{2}(\pi_v (B)) &>0, &&\dim B > 2, \end{aligned} \]
for $\mathcal{H}^{d-1}$-a.e.~$v \in S^{d-1}$. \end{proposition}

%The proof of Proposition~\ref{projection} given here uses results from weighted Fourier restriction on the cone. 
The Peres-Schlag projection theorem does not imply Proposition~\ref{projection}, since the family of projections fails their transversality condition (see Appendix~\ref{appendix}). Moreover, Proposition~\ref{projection} is not implied (at least not for all possibilities) by the results of~\cite{jarvenpaa1,jarvenpaa2} for more general families of planes without curvature assumptions. I don't know if it is possible to prove Proposition~\ref{projection} directly via sublevel set estimates. If an analogous projection family $\pi_v$ existed for $d=3$, then Fourier analytic method here would give the sharp projection theorem, and this would not (trivially) follow from sublevel set estimates. If the number of vector fields is increased so that the projections are onto $k$-dimensional planes with $k \geq 3$, e.g. $\spn \left\{ \frac{1}{\sqrt{2}}(v,-1), ( A_1v,0), ( A_2v,0) \right\}$, then the sharp bound can be proved by a simple change of variables, without Fourier restriction.

For $\alpha \in [0,d+1]$ let $\beta\big(\alpha, \Gamma^d\big)$ be the supremum over all $\beta \geq 0$ satisfying
\[ \int_{\Gamma^d} \left\lvert \widehat{\mu}(R\xi) \right\rvert^2 \, d\sigma_{\Gamma}(\xi) \lesssim_{\beta} c_{\alpha}(\mu) \mu\big(\mathbb{R}^{d+1}\big)  R^{-\beta}, \]
for all $R>0$ where $\sigma_{\Gamma}$ is the surface measure on the truncated cone $\Gamma^d$. By ignoring constant factors, $\sigma_{\Gamma}$ is essentially equal to the pushforward of the Lebesgue measure on $B(0,1) \setminus B(0,1/2)$ under $\xi \mapsto (\xi,\lvert \xi\rvert)$.

For $d \geq 4$, the currently known lower bound on $\beta\big(\alpha,\Gamma^d\big)$ is
\begin{equation} \label{betabound} \beta\big(\alpha,\Gamma^d\big) \geq \left\{ \begin{aligned} 
&\alpha,														&												& \alpha \in \left[ 0, \frac{d-1}{2} \right] 						 &&\text{(\cite{mattila})}, \\
&\alpha - \frac{1}{2} \left( \alpha - \frac{(d-1)}{2} \right), & &\alpha \in \left( \frac{d-1}{2}, \frac{d+1}{2} \right] && \text{(\cite{cho})}, \\
&\alpha-1 + \frac{d-\alpha}{d-1}, 		&													& \alpha \in \left(\frac{d+1}{2}, d \right]						 &&\text{(\cite{harris})}, \\
&\alpha-1,													&												&\alpha \in \left( d, d+1 \right] 											 &&\text{(\cite{sjolin})}. \end{aligned} \right. \end{equation}
Only the first bound, due to Mattila, will be used here. The same lower bound holds and is sharp for $d=3$; this is due entirely to~\cite{cho} (the lower bound $\beta\big(\alpha,\Gamma^d\big) \geq \alpha-1$ holds for any $d \geq 2$ as a straightforward consequence of duality and Plancherel~\cite{erdogan,cho}).
\begin{proof}[Proof of Proposition~\ref{projection}] Let $B \subseteq \mathbb{R}^{d+1}$ be an analytic set. Let $\phi:B_{d-1}(0,1) \to S^{d-1}$ be the map
\[ y \mapsto \left(y, \sqrt{1-\lvert y\rvert^2} \right). \]
For the first part define $\widetilde{G}, \widetilde{F}: S^{d-1} \to S^d$ by
\[ \widetilde{G}(v) =  \frac{1}{\sqrt{2}}(v,-1), \quad \widetilde{F}(v) =  (Av,0),  \]
 Define $G,F: B_{d-1}(0,1/2) \to S^d$ by $G = \widetilde{G} \circ \phi$ and $F = \widetilde{F} \circ \phi$, so that 
\[ G(y) = \frac{1}{\sqrt{2}}\left(y, \sqrt{1-\lvert y\rvert^2}, -1 \right) \in S^d \subseteq \mathbb{R}^{d-1} \times \mathbb{R} \times \mathbb{R}, \]
and 
\[ F(y) = \left(A \begin{pmatrix} y \\ \sqrt{1-\lvert y\rvert^2} \end{pmatrix}, 0 \right) \in S^d \subseteq \mathbb{R}^{d-1} \times \mathbb{R} \times \mathbb{R}. \]
Then $F$ is a diffeomorphism onto its $(d-1)$-dimensional image, and therefore $\rank DF \equiv d-1$. Hence 
\[ \rank \begin{pmatrix} G & DF \end{pmatrix}  = d, \]
since the last entry of $G$ is everywhere nonzero. The derivative of $G$ is the $(d+1) \times (d-1)$ matrix 
\[ DG(y) = \frac{1}{\sqrt{2}}\begin{pmatrix} I_{d-1} \\ \frac{-y}{\sqrt{1-\lvert y\rvert^2}} \\ 0 \end{pmatrix}. \]
 Moreover, 
\[ F(y) = \sqrt{2} \sum_{j=1}^{d-1} (Ay)_j \frac{ \partial G}{\partial y_j }. \]
This can be checked by comparing with the formula for $DG(y)$; the only nontrivial equality to check is in the $d$-th coordinate, which holds since $A$ is a vector field on $S^{d-1}$. This implies that
\[ \rank\begin{pmatrix} F & DG \end{pmatrix} \equiv \rank DG = d-1, \]
which verifies \eqref{rankassumption}. 

For \eqref{cdn2}, let $U$ be the set $(1,\infty) \times \mathbb{R} \times B_{d-1}(0,1/2)$ and define $H: U \to \mathbb{R}^{d+1}$ by $H(\lambda_1, \lambda_2, y ) = \lambda_1 F(y) + \lambda_2G(y)$. Then for any positive function $g \in L^1\left(\mathbb{R}^{d+1}\right)$ supported in $H(U)$,
\begin{align*} &\int_{\mathbb{R}^{d+1}} \left( \lambda_1^2 + \lambda_2^2 \right)^{\frac{d-2}{2}} (g \circ H)(\lambda_1, \lambda_2, y)  \, d\lambda_1 \, d\lambda_2 \, dy  \\
&\quad =\int_1^{\infty} \int_{-\infty}^{\infty} \int_{B_{d-1}(0,1/2)} \left( \lambda_1^2 + \lambda_2^2 \right)^{\frac{d-2}{2}} (g \circ H)(\lambda_1 \lambda_2, y) \, d\lambda_1 \, d\lambda_2 \, dy  \\
&\quad \lesssim \int_1^{\infty} \int_{-\infty}^{\infty} \int_{S^{d-1}}  \left( \lambda_1^2 + \lambda_2^2 \right)^{\frac{d-2}{2}} g \begin{pmatrix} \lambda_1 Av + \lambda_2 v \\ -\lambda_2 \end{pmatrix} \, d\sigma(v) \, d\lambda_1 \, d\lambda_2 \\
&\quad = \int_1^{\infty} \int_{-\infty}^{\infty} \int_{S^{d-1}}  \left( \lambda_1^2 + \lambda_2^2 \right)^{\frac{d-2}{2}} g \begin{pmatrix} \left(\lambda_1^2 +\lambda_2^2\right)^{1/2} v \\ -\lambda_2 \end{pmatrix} \, d\sigma(v) \, d\lambda_1 \, d\lambda_2 \\
&\quad \lesssim \int_{-\infty}^{\infty} \int_{\lvert \lambda_2\rvert}^{\infty} \int_{S^{d-1}} r^{d-1} g \begin{pmatrix} r v \\ -\lambda_2 \end{pmatrix} \, d\sigma(v) \, dr \, d\lambda_2  \\
&\quad\lesssim \int_{\mathbb{R}^{d+1}} g(\xi) \, d\xi  \\
&\quad= \int_{\mathbb{R}^{d+1}} \left\lvert \det DH (\lambda_1, \lambda_2, y ) \right\rvert (g \circ H)(\lambda_1, \lambda_2, y)  \, d\lambda_1 \, d\lambda_2 \, dy. \end{align*} 
It follows that 
\begin{multline*} \left\lvert \det DH (\lambda_1, \lambda_2, y ) \right\rvert  = \left\lvert\det\begin{pmatrix} F(y) & G(y) & \lambda_1DF(y) + \lambda_2 DG(y) \end{pmatrix}\right\rvert \\
\gtrsim \left(\lambda_1^2 + \lambda_2^2 \right)^{\frac{d-2}{2}}, \end{multline*} 
for all $\lambda_1 \geq 1$, $\lambda_2 \in \mathbb{R}$ and $y \in B_{d-1}(0,1/2)$. Setting $\lambda_1 = 1$ gives 
\[ \min_{\substack{\lambda \in \mathbb{R} \\
y \in B_{d-1}(0,1/2)  }}\left\lvert\det\begin{pmatrix} F(y) & G(y) & DF(y) + \lambda DG(y) \end{pmatrix}\right\rvert >0,  \]
which verifies \eqref{cdn2} with $\Omega =B_{d-1}(0,1/2)$. 

Applying Proposition~\ref{proposition} and using symmetry gives
\[  \dim \pi_v (B) \geq \min\left\{2, \dim B, \beta\big(\dim B, \Gamma^d\big) + \frac{1}{2} \right\}  \quad \text{for $\mathcal{H}^{d-1}$-a.e.~$v \in S^{d-1}$,} \]
and $\mathcal{H}^{2}(\pi_v (B)) >0$ for $\mathcal{H}^{d-1}$-a.e.~$v \in S^{d-1}$ if $\min\left\{\dim B, \beta\big(\dim B, \Gamma^d\big) + \frac{1}{2} \right\}>2$. Since $d \geq 4$, the lower bound for $\beta\left( \cdot, \Gamma^d \right)$ in \eqref{betabound} shows that
\[ \beta\big(\dim B, \Gamma^d\big) + \frac{1}{2} \geq \min\{2,\dim B\}, \]
with strict inequality  if $\dim B \neq 2$. Hence
\[ \begin{aligned}  \dim \pi_v (B) &= \dim B, && \dim B \in [0,2] \\  
\mathcal{H}^{2}(\pi_v (B)) &>0, && \dim B > 2, \end{aligned} \]
for $\mathcal{H}^{d-1}$-a.e.~$v \in S^{d-1}$. This proves Proposition~\ref{projection}.  \end{proof}
\appendix
\section{Proof that the Oberlin-Oberlin inequality implies \texorpdfstring{Lemma~\ref{venierilemma}}{lemma} when \texorpdfstring{$s\geq 9/4$}{s exceeds 2.25}} \label{appendix2}
Before considering the case $s \geq 9/4$, the $\nu\left(\mathbb{R}^3\right)$ factor in Lemma~\ref{venierilemma} will briefly be explained in the case $s < 9/4$. In~\cite[p.~12]{venieri} an inequality of the form 
\[ \nu^4\left\{ (z,z_1, z_2, z_3) \in Z' \times B_0^3 : z_j \sim_j z \text{ for all } 1 \leq j \leq 3 \right\} \lessapprox t^{2s} \delta^s, \]
is obtained by taking a $\nu^3$ triple integral of the bound
\[ \nu\left\{ z \in Z' : z_j \sim_j z \text{ for all } 1 \leq j \leq 3 \right\} \lessapprox \delta^s, \quad z_1, z_2, z_3 \in B_0. \]
The inner two integrals in $z_2, z_3$ apply the $s$-Frostman condition to the inequalities $\lvert z_2-z_1\rvert, \lvert z_3-z_1\rvert \lesssim t$ for fixed $z_1$, whilst the remaining $\nu$-integral over $z_1$ is 1 since $\nu$ is a probability measure. In the case where $\nu$ is not a probability measure, this last bound gives a factor of $\nu\left(\mathbb{R}^3\right)$ whilst the other steps remain unchanged. 

\begin{proof}[Proof of Lemma~\ref{venierilemma} for $s \geq 9/4$] Fix 
\[ \kappa > \left\{ 0, \min\left\{ \frac{2s}{3} -1 , \frac{1}{2} \right\} \right\} = \frac{1}{2}, \]
since $s \geq 9/4$. All sets occurring in the statement of the lemma are measurable; see e.g.~\cite{harris2}. Let $\eta>0$ be a small constant to be chosen, and let 
\[ Z= Z(\delta) = \left\{ y \in \mathbb{R}^3: \mathcal{H}^1\left\{ \theta \in [0,2\pi) : \pi_{\theta\#} \nu\left( B\left( \pi_{\theta}(y), \delta \right)\right) > \delta^{s-\kappa} \right\} \geq \delta^{\eta} \right\}. \]
It is required to show that $\nu(Z) \leq \nu\left(\mathbb{R}^3\right)\delta^{\eta}$ for all $\delta$ sufficiently small. Let $t = s-\kappa + 100\eta$. By three applications of Chebychev's inequality,
\begin{align*} \nu(Z) &\leq \int_{\mathbb{R}^3} \delta^{-\eta} \int_0^{2 \pi} \frac{ \left(\pi_{\theta \#} \nu\right)\left( B\left(\pi_{\theta}(y), \delta \right) \right) }{\delta^{s-\kappa} } \, d\theta \, d\nu(y) \\
&= \delta^{-(s-\kappa+\eta) } \int_0^{2 \pi} \int_{\mathbb{R}^3 } \int_{ \left\{x \in \mathbb{R}^3 : \left\lvert \pi_{\theta}(x-y) \right\rvert < \delta \right\}} \, d\nu(x) \, d\nu(y) \, d\theta \\
&\leq \delta^{99 \eta } \int_0^{2 \pi} I_t\left(\pi_{\theta\#} \nu \right) \, d\theta \\
&\leq \delta^{\eta}\nu\left( \mathbb{R}^3 \right), \end{align*}
where the last inequality is from~\cite{oberlin} (see also Theorem~\ref{curvature}), and holds provided $\eta$ is chosen small enough to ensure that $s-\kappa+100\eta < s-1/2$, and $\delta_0$ is chosen small enough (after $\eta$). \end{proof} 

\section{Failure of transversality} \label{appendix}
Peres and Schlag's Theorem~4.8 from~\cite{peres} gives a projection theorem for a general family of projections satisfying a transversality condition. Here it will be shown that this transversality condition is not satisfied by the family of projections occurring in Proposition~\ref{projection}. Let $Q \subseteq \mathbb{R}^{d-1}$ be an open connected set. Let $v: Q \to S^{d-1}_+$ parametrise the open upper hemisphere smoothly. Define $\Pi: Q \times \mathbb{R}^{d+1} \to \mathbb{R}^2$ by 
\[ \Pi(\lambda, x) = \begin{pmatrix} \left\langle x, \frac{1}{\sqrt{2}} \begin{pmatrix} v(\lambda) \\ -1 \end{pmatrix} \right\rangle  \\
\left\langle x, \begin{pmatrix} Av(\lambda) \\ 0 \end{pmatrix} \right\rangle \end{pmatrix}. \]
For distinct $x,y \in \mathbb{R}^{d+1}$, let
\[ \Phi_{\lambda}(x,y) = \frac{ \Pi(\lambda, x) - \Pi(\lambda, y) }{\lvert x-y\rvert} = \Pi\left(\lambda, \frac{x-y}{\lvert x-y\rvert} \right) = \Pi\left(\lambda, z\right), \quad z := \frac{x-y}{\lvert x-y\rvert}. \] 
Write $x=(x',x_{d+1}) \in \mathbb{R}^d \times \mathbb{R}$, so that for fixed $x$, 
\[ \partial_{\lambda_i} \Pi(\lambda, x)= \begin{pmatrix} \frac{1}{\sqrt{2}}  \langle x', \partial_i v(\lambda) \rangle \\
\langle -Ax', \partial_i v(\lambda) \rangle \end{pmatrix}. \]
The derivative of $\Phi_{(\cdot)}(x,y)$ with respect to $\lambda \in Q$ is therefore the $2 \times (d-1)$ matrix
\[ D_{\lambda} \Phi_{\lambda}(x,y)  = \begin{pmatrix} \frac{1}{\sqrt{2}} \langle z', \partial_1 v(\lambda) \rangle & \cdots & \frac{1}{\sqrt{2}} \langle z', \partial_{d-1} v(\lambda) \rangle  \\
\langle -Az', \partial_1 v(\lambda)\rangle & \cdots & \langle -Az', \partial_{d-1} v(\lambda) \rangle \end{pmatrix}. \]
The $\beta$-transversality condition is: for any fixed compact $\Omega \subseteq \mathbb{R}^{d+1}$, there exists a constant $C_{\beta}>0$ such that for all for all $\lambda \in Q$ and all distinct $x,y \in \Omega$, the condition
\begin{equation} \label{condition1} \lvert \Pi\left(\lambda, z\right)\rvert\leq C_{\beta}\lvert x-y\rvert^{\beta}, \end{equation}
implies that
\begin{equation} \label{condition2} \det\left(D_{\lambda}\Phi(x,y) \left[D_{\lambda}\Phi(x,y) \right]^*\right) \geq C_{\beta}^2 \lvert x-y\rvert^{2\beta}. \end{equation}
This definition is from~\cite{peres,falconer}; the one in~\cite{peres} has an extraneous condition (probably a typo). 

Let $\Omega$ be the closed unit ball centred at the origin, and suppose for a contradiction that the $\beta$-transversality condition is satisfied for some positive constant $C_{\beta}$, for some $\beta \geq 0$. Fix any $\lambda \in Q$ and define $x,y \in \Omega$ by
\[ x = \frac{1}{2\sqrt{2}}\begin{pmatrix} v(\lambda) \\ 1 \end{pmatrix}= -y, \quad \text{so that} \quad x-y = \frac{1}{\sqrt{2}}\begin{pmatrix} v(\lambda) \\ 1 \end{pmatrix}. \]
The left-hand side of \eqref{condition1} vanishes. Since the tangent space to any point on $S^{d-1}$ is the plane orthogonal to that point, all entries in the top row of $D_{\lambda}\Phi(x,y)$ are zero. Hence the $2 \times 2$ matrix $D_{\lambda}\Phi(x,y)\left[D_{\lambda}\Phi(x,y) \right]^*$ has at most 1 nonzero entry, and its determinant vanishes. Therefore \eqref{condition2} contradicts the assumption that $C_{\beta}$ is positive.

\end{document}